\documentclass[12pt]{article}
\usepackage{amssymb, amsmath, amsthm}
\usepackage[all]{xy}
\usepackage{mathptmx,stmaryrd}
\usepackage{hyperref}
\usepackage{accents}

\DeclareFontEncoding{OT2}{}{}

\newcommand{\mr}[1]{\mathrm{#1}}
\newcommand{\mf}[1]{\mathfrak{#1}}
\newcommand{\mc}[1]{\mathcal{#1}}
\newcommand{\mb}[1]{\mathbb{#1}}

\newcommand{\Z}{\mb{Z}}
\newcommand{\Q}{\mb{Q}}
\newcommand{\zp}{\mb{Z}_p}
\newcommand{\qp}{\mb{Q}_p}
\newcommand{\C}{\mb{C}}

\newcommand{\F}{\mb{F}}

\newcommand{\qbar}{\overline{\Q}}
\newcommand{\otil}[1]{\tilde{#1}}
\DeclareMathAccent{\wtilde}{\mathord}{largesymbols}{"65}

\newcommand{\qpbar}{\overline{\Q_p\!\!}\,\,}
\newcommand{\qpbt}{\overline{\Q_p\!\!}{\!\!\!\!\!\!\!\phantom{\bar{\Q}}}^{\times}}

\newcommand{\id}{\mr{id}}

\newcommand{\Eis}{I}

\newcommand{\cts}{\mr{cts}}
\newcommand{\et}{\text{\'et}}
\newcommand{\ur}{\mr{ur}}
\newcommand{\Fr}{\mr{Fr}}

\newcommand{\La}{\Lambda}
\newcommand{\Lai}{\La^{\iota}}

\newcommand{\Ga}{\Gamma}

\newcommand{\ord}{\mr{ord}}

\newcommand{\Iw}{\mr{Iw}}
\newcommand{\sub}{\mr{sub}}
\newcommand{\quo}{\mr{quo}}

\newcommand{\cusps}{\{\mr{cusps}\}}

\newcommand{\mcS}{\mc{S}}

\newcommand{\tGa}{\widetilde{\Gamma}}

\newcommand{\w}{w	}
\newcommand{\mcT}{\mc{T}}
\newcommand{\tmcT}{\tilde{\mc{T}}}
\newcommand{\mfS}{\mf{S}}

\newcommand{\mfC}{\mf{C}}

\newcommand{\invlim}[1]{\varprojlim_{#1} \,}
\newcommand{\ps}[1]{\llbracket #1 \rrbracket}

\newcommand{\cotimes}[1]{\,\hat{\otimes}_{#1} \,}
\newcommand{\cozp}{\cotimes{\zp}}

\DeclareMathOperator{\Hom}{Hom} 
\DeclareMathOperator{\End}{End} \DeclareMathOperator{\Gal}{Gal}
 
 \DeclareMathOperator{\Cone}{Cone}

\DeclareMathOperator{\Cole}{Col}

\DeclareMathOperator{\cor}{cor}
\DeclareMathOperator{\Inf}{Inf}

\DeclareMathOperator{\inv}{inv}
\DeclareMathOperator{\ev}{ev}

\DeclareMathOperator{\Spec}{Spec}
\DeclareMathOperator{\loc}{loc}

\newtheorem{theorem}{Theorem}[section]
\newtheorem{proposition}[theorem]{Proposition}
\newtheorem{lemma}[theorem]{Lemma}
\newtheorem{corollary}[theorem]{Corollary}
\newtheorem*{thm}{Theorem}
\newtheorem{hypothesis}[theorem]{Hypothesis}

\newtheorem{conjecture}[theorem]{Conjecture}

\theoremstyle{definition}
\newtheorem{definition}[theorem]{Definition}
\newtheorem{notation}[theorem]{Notation}

\theoremstyle{remark}
\newtheorem{remark}[theorem]{Remark}
\newtheorem{question}[theorem]{Question}
\newtheorem*{ack}{Acknowledgments}

\renewcommand{\baselinestretch}{1.2}
\numberwithin{equation}{section}

\newcommand\extrafootertext[1]{%
    \bgroup
    \renewcommand\thefootnote{\fnsymbol{footnote}}%
    \renewcommand\thempfootnote{\fnsymbol{mpfootnote}}%
    \footnotetext[0]{#1}%
    \egroup
}

\begin{document}

\title{An extension of the Fukaya-Kato method\extrafootertext{\emph{Address}: Department of Mathematics, University of California, Los Angeles, 520 Portola Plaza, Los Angeles, CA 90095. \emph{Email}: sharifi@math.ucla.edu}
}
\author{Romyar T. Sharifi}
\date{}
\maketitle

\begin{abstract}
In the groundbreaking paper \cite{fk-pf}, T.~Fukaya and K.~Kato proved a slight weakening of a conjecture of the author's \cite{me-Lfn} under an assumption that a Kubota-Leopoldt $p$-adic $L$-function has no multiple zeros.  This article describes a refinement of their method that sheds light on the role of the $p$-adic $L$-function.
\end{abstract}

\section{Introduction}

\subsection{Overview}

For a positive integer $M$, we explore the conjectural relationship between 
\begin{itemize}
	\item modular symbols in the quotient $P$ of the real part of a first homology group of a modular curve of level $M$ by the action
	of an Eisenstein ideal $I$, and 
	\item cup products of cyclotomic units in a second Galois cohomology group $Y$ of the cyclotomic field $\Q(\mu_M)$ with 
	ramification restricted to $M$.
\end{itemize}	 
More specifically, we consider the likewise-denoted maximal quotients of $p$-parts of the inverse limits $P$ and $Y$ of the above groups
in towers of levels $M = Np^r$ on which $(\Z/Np\Z)^{\times}$ acts through a given even character $\theta$ via diamond operators and Galois elements, respectively. For the precise conditions on $p$, $N$, and $\theta$, see \ref{pNcond} and \ref{thetahyp}, or the next subsection.

In \cite{me-Lfn}, we constructed two maps $\varpi \colon P \to Y$ and $\Upsilon \colon Y \to P$ and conjectured them to be inverse to each other, up to a canonical unit suspected to be $1$ (see Conjecture \ref{sconj}).  The map $\varpi$ was defined explicitly to take a modular symbol to a sum of cup products of cyclotomic $Np$-units, while $\Upsilon$ was defined through the Galois action on the homology of a modular curve, or a tower thereof, in the spirit of the Mazur-Wiles method of proof of the main conjecture.  By the main conjecture, both the homology group $P$ and the Galois cohomology group $Y$ are annihilated by a power series $\xi$ in the Iwasawa algebra corresponding to a $p$-adic $L$-function.  This power series $\xi$ is (roughly) both a generator of the characteristic ideal of $Y$ and the constant term of an ordinary family of Eisenstein series determined by $\theta$.

In a 2012 preprint (now published \cite{fk-pf}), Fukaya and Kato proved the key identity that 
\begin{equation} \label{keyident}
	\xi' \Upsilon \circ \varpi = \xi'
\end{equation}
modulo $p$-power torsion in $P$, where $\xi'$ is the derivative of $\xi$ with
respect to the variable of the $p$-adic $L$-function. In Theorem \ref{FKident}, we show that this identity holds in $P$ itself, employing joint work from \cite{fks2}. At least up to torsion in $P$, the conjecture follows if $\xi'$ happens to be relatively prime to $\xi$ in the relevant Iwasawa algebra.

Considerable progress has been made in the study of $\Upsilon$ by Wake and Wang-Erickson \cite{wwe} and Ohta \cite{ohta-mu}, by different methods. In cases that $\Upsilon$ is known to be an isomorphism and $Y$ is pseudo-cyclic, the identity of Fukaya and Kato implies the original conjecture, i.e., up to unit. This pseudo-cyclicity was related to the question of localizations of Hecke algebras being Gorenstein in the work of Wake and Wang-Erickson, as well as to the question of $\Upsilon$ being a pseudo-isomorphism. Ohta shows that $\Upsilon$ is in fact an isomorphism under an assumption on the relevant Dirichlet character that holds in the case of trivial tame level, supposing a certain nonvanishing of $L$-values modulo $p$. We note that this implies in particular that $P$ has no torsion in such eigenspaces, as $Y$ does not.

The pseudo-cyclicity of $Y$ is expected to hold as a consequence of a well-known and widely believed conjecture of Greenberg's on the finiteness of the plus part of the unramified Iwasawa module. Moreover, since the $p$-adic $L$-functions in question are unlikely to ever have multiple zeros, one would expect the unit in our conjecture to always be $1$, as in its stronger form. Nevertheless, this might appear to reduce the conjecture to chance, which is less than desirable. This motivates us to attempt a finer study.

Our primary aim in this paper is to study the role of $\xi'$ in the work of Fukaya-Kato and ask whether it is possible to remove it in the method. For this, we will need to reconcile the distinct natures of the two occurrences of $\xi'$ in the Fukaya-Kato identity \eqref{keyident}. On the left-hand side, $\xi' \colon Y \to Y$ is identified with a Bockstein-like connecting map between two Galois cohomology groups, both isomorphic to $Y$. 
On the right-hand side, $\xi' \colon P \to P$ arises as an Eisenstein reduction of a composition of a zeta map carrying modular symbols to sums of Beilinson-Kato elements and a Coleman map that serves as a $p$-adic regulator.
The zeta map is the tie that binds the two sides of \eqref{keyident} together, for when its negative is composed with 
pullback by the cusp $\infty$, one obtains the map $\varpi \colon P \to Y$. 

As we shall see, it would be possible to remove $\xi'$ from both sides of \eqref{keyident} but for a global obstruction
related to the existence of a zeta map at an intermediate level.  
Shapiro's lemma allows one to view an Iwasawa cohomology group over the cyclotomic $\zp$-extension $\Q_{\infty}$ of $\Q$, which is to say an inverse limit of cohomology groups up the cyclotomic tower under corestriction, as a Galois cohomology group over $\Q$ with coefficients in an induced module.  The latter module is the completed tensor product of the original coefficients with a free cyclic module $\La^{\iota}$ over the Iwasawa algebra $\La = \zp\ps{X}$ on which Galois elements act as inverses of group elements.

Our key innovation is the consideration of the cohomology of a quotient of these induced coefficients by an arithmetically relevant two-variable power series divisible by the first variable $X$. In particular, the cohomology of this \emph{intermediate quotient} (see Definition \ref{interquot}) is not the cohomology of any intermediate extension, yet it lies between cohomology over $\Q$ and Iwasawa cohomology over $\Q_{\infty}$. 
In Theorem \ref{loccohdag}, we give a surprisingly clean and quite general construction of \emph{intermediate Coleman maps} on the local-at-$p$ first cohomology groups of intermediate quotients for Tate twists of unramified $\zp\ps{G_{\qp}}$-modules. We hope these will find application beyond this work.

We show that the global obstruction to removing $\xi'$ would vanish under a divisibility of Beilinson-Kato elements by one minus the $p$th Hecke operator at an intermediate stage between Iwasawa cohomology and cohomology at the ground level: see Question \ref{zdagger}. This ``intermediate global divisibility'' can be rephrased as the existence of a certain \emph{intermediate zeta map}. 
The main result of this paper, Theorem \ref{equivform}, states that the global obstruction to our conjecture is equivalent to the weaker statement of existence of what would be a reduction of this map modulo the Eisenstein ideal. This reduced map
is required to be compatible with the reduced zeta map at the ground level of $\Q$ and the reduction of a local-at-$p$ intermediate zeta map that we construct.

Of course, this leaves us with the question of whether these global intermediate zeta maps are likely to exist. As such, we perform a feasibility check for an analogue of the conditions of Theorem \ref{equivform} in a simpler setting, with cyclotomic units in place of Beilinson-Kato elements. That is, in Section \ref{testcase}, we explore the analogues of global obstruction and divisibility for cohomology with coefficients in a Tate module, rather than the \'etale homology of a tower of modular curves. We show that the global obstruction in the cohomology of the intermediate quotient does in fact vanish in this setting, while verifying intermediate global divisibility only under an assumption of vanishing of a $p$-part of a class
group of a totally real abelian field. This is in line with our suspicions that intermediate global divisibility may be too much to hope for in general, while still lending some credence to the conjecture that $\Upsilon$ and $\varpi$ are indeed inverse maps, and not just by chance.

\subsection{Detailed summary}

As a guide to a thorough reading of this rather involved work, we review the method of Fukaya and Kato (see \cite[Section 2]{fks} for
another review) in order to describe our refinement and the constructions it involves.

\subsubsection{Lattices and their reductions}

Let $N \ge 1$, and fix a prime $p \ge 5$ such that $p \nmid N\varphi(N)$.
Let $\theta$ be an even $p$-adic character of $(\Z/Np\Z)^{\times}$ of conductor divisible by $N$ 
such that $p \mid B_{2,\theta^{-1}}$ and such that $\theta$ does not induce the mod $p$ cyclotomic character on
decomposition at $p$.
The $\theta$-isotypical component $\mf{h}_{\theta}$ of the Eisenstein localization of Hida's ordinary Hecke algebra
 is a module over a Iwasawa algebra $\La_{\theta}$ of inverse diamond operators, with coefficient ring the $\zp$-algebra generated by 
 the values of $\theta$. That the Eisenstein quotient has the form
$(\mf{h}/I)_{\theta} \cong \La_{\theta}/\xi$
was proven in \cite{wiles} as a consequence of the Iwasawa main conjecture. 

Let $\mcT_{\theta}$ be the inverse limit of $\theta$-Eisenstein components
of the \'etale cohomology groups $H^1_{\et}(X_1(Np^r)_{/\qbar},\zp(1))$ under trace maps. It is a rank two $\mf{h}_{\theta}$-module with a commuting
global Galois action that on decomposition at $p$ that fits in a short exact sequence
$$
	0 \to \mcT_{\sub} \to \mcT_{\theta} \to \mcT_{\quo} \to 0,
$$
where $\mcT_{\sub} \cong \mf{h}_{\theta}$ as a $\mf{h}_{\theta}$-module, and $\mcT_{\quo}$ has unramified Galois action and 
is isomorphic over $\mf{h}_{\theta}$ to the $\theta$-Eisenstein component $\mcS_{\theta}$ of the $\La$-adic cusp forms.

The Eisenstein quotient $T = \mcT_{\theta}/I\mcT_{\theta}$ fits in an exact sequence
\begin{equation} \label{keyexseq}
	0 \to P \to T \to Q \to 0
\end{equation}
of global, locally split Galois modules with $P \xrightarrow{\sim} \mcT_{\quo}/I\mcT_{\quo}$ and 
$\mcT_{\sub}/I\mcT_{\sub} \xrightarrow{\sim} Q$ via the canonical maps.
Here, $Q$ can be identified with the Tate twist of $\Lai_{\theta}/\xi$ by considering its generator that is Poincar\'e dual to 
a particular modular symbol. This identification is the ingredient that makes $\Upsilon$ canonical in \cite{me-Lfn}.
For more detail, see Section \ref{background}.

\subsubsection{Left-hand side}

We can reinterpret the left-hand side $\xi' \Upsilon \circ \varpi$ of \eqref{keyident} via the following identifications.
\begin{itemize}
	\item The map $\Upsilon \colon Y \to P$ is a connecting map 
	$H^2(Q(1)) \to H^3_c(P(1))$ for the Tate twist of \eqref{keyexseq}, from unramified outside $Np$ 
	to compactly supported cohomology over $\Q$.
\end{itemize}

In fact, we have identifications $H^i(Q(1)) = Y$ for $i \in \{1,2\}$, while $H^i_c(P(1)) = P$ for $i \in \{2,3\}$.
See Section \ref{redlat} for more detail.
	
\begin{itemize}
	\item The map $-\varpi \colon P \to Y$ is the composition $\bar{z}_Q^{\sharp}$ 
	of the reduction $\bar{z}^{\sharp} \colon P \to H^1(T(1))$ of a zeta map $z^{\sharp} \colon \mcS_{\theta} \to H^1(\mc{T}_{\theta}(1))$
	with the canonical map $H^1(T(1)) \to H^1(Q(1))$. 
\end{itemize}

Here, the composition $H^1(T(1)) \to Y$ has a reinterpretation as the pullback by the cusp $\infty$ mentioned earlier, and 
$\bar{z}^{\sharp}$ is the reduction of a zeta map $z^{\sharp} \colon \mcS_{\theta} \to H^1(\mc{T}_{\theta}(1))$
(see Theorems \ref{zsharp} and \ref{zsharpvarpi}).	
	
\begin{itemize}
	\item The map $\xi' \colon Y \to Y$ is a connecting map $\partial \colon H^1(Q(1)) \to H^2(Q(1))$ for the tensor product of $Q(1)$ with
	the exact sequence
	\begin{equation} \label{Bocksteinseq}
		0 \to \zp \xrightarrow{X} \Lai/X^2 \to \zp \to 0.
	\end{equation}
\end{itemize}

This is recalled in Theorem \ref{FKgalcoh} (see \cite[\S 9.3]{fk-pf}). The rough idea expounded upon in Remark \ref{diffquot} is that 
$\partial$ is the derivative of $\xi$ in $Q = (\La_{\theta}/\xi)^{\iota}(1)$ with respect to the $X$ in the map of \eqref{Bocksteinseq}.

The left-hand side of \eqref{keyident} may then be viewed as coming from the commutative diagram
\begin{equation} \label{leftside}
	\SelectTips{cm}{} \xymatrix{
	P \ar[r]^-{-\bar{z}^{\sharp}_Q} \ar[rd]_{\varpi} & H^1(Q(1)) \ar@{=}[d] \ar[r]^{\partial} & H^2(Q(1)) \ar@{=}[d] \ar[r] & 
	H^3_c(P(1)) \ar@{=}[d], \\
	& Y \ar[r]^{\xi'} & Y \ar[r]^{\Upsilon} & P
	}
\end{equation}

\subsubsection{Intermediate quotients}

We remove $\xi'$ from the center square of \eqref{leftside} using intermediate quotients.
For a compact $\mf{h}_{\theta}$-module $M$ with a commuting continuous Galois action,
the intermediate quotient for $M$ with respect to an element $\alpha$ of the completed tensor product $\La \cozp \mf{h}_{\theta}$ is 
defined (cf. Definition \ref{interquot}) as
$$
	M^{\dagger} = \frac{\La^{\iota} \cozp M}{X\alpha(\La^{\iota} \cozp M)},
$$
a quotient of $\Lai \cozp M$ of which $M$ is a quotient.
If $I$ annihilates $M$, then so does $\xi \in \mf{h}_{\theta}$. In this case, 
we typically take $X\alpha$ to be the slight modification 
$\otil{\xi} - 1 \otimes \xi$ of the diagonalization $\otil{\xi} \in \La \cozp \La_{\theta}$ of $\xi$, which makes it divisible by $X$:
see Section \ref{intermedquo}.

In place of the tensor product of $Q$ with \eqref{Bocksteinseq}, 
we consider in Section \ref{refinedcoh} the exact sequence
\begin{equation} \label{modexseq}
	0 \to Q \xrightarrow{\otil{\xi}} (\La^{\iota} \cozp Q)/X\otil{\xi} \to Q^{\dagger} \to 0.
\end{equation}
In Theorem \ref{partial}, we prove the following.
\begin{thm}
	The connecting map 
	$$
		\partial^{\dagger} \colon H^1(Q^{\dagger}(1)) \to H^2(Q(1))
	$$ 
	arising from \eqref{modexseq} 
	factors through the quotient 
	$H^1(Q(1))$ and induces the identity on $Y$.
\end{thm}
As explained in Remark \ref{diffquot}, since the first map in \eqref{modexseq} is now $\otil{\xi}$ with $\otil{\xi}(0) = \xi$,
the map $\partial^{\dagger}$ induces on $Y$ multiplication by the derivative $1$ of $\xi$ with respect to itself, instead of $X$.

Finally, there exists a map $\bar{z}_Q^{\dagger} \colon P \to H^1(Q^{\dagger}(1))$ (see the proof of Theorem \ref{equivform}) 
making the triangle in the following diagram commute
$$
	\SelectTips{cm}{} \xymatrix{
	P \ar[r]^-{-\bar{z}_Q^{\dagger}} \ar[rd]_-{\varpi} 
	& H^1(Q^{\dagger}(1)) \ar@{->>}[d] \ar[r]^{\partial^{\dagger}} & H^2(Q(1)) \ar@{=}[d] \ar[r] & 
	H^3_c(P(1)) \ar@{=}[d], \\
	& Y \ar[r]^{1} & Y \ar[r]^{\Upsilon} & P.
	}
$$

\subsubsection{Right-hand side}

The right-hand side of \eqref{keyident} is the composition of a zeta map and a Coleman map.
The Coleman map for the Tate module takes a norm compatible
sequence of units in the cyclotomic $\zp$-extension of $\qp$ to a logarithm of its Coleman power series \cite{coleman}.
For Coleman maps in general, see Section \ref{coleman}. We focus here on the case central to \cite{fk-pf}.
 
A modified logarithm and $\La$-adic Eichler-Shimura \cite{ohta-es} provide an isomorphism 
$$
	\Cole^{\flat} \colon H^1_{\loc}(\mcT_{\quo}(1)) \xrightarrow{\sim} \mfS_{\theta}
$$ 
from the absolute Galois cohomology of $\qp$. 
Here, $\mfS_{\theta}$ is a group of $\La$-adic cusp forms non-canonically isomorphic to $\mcS_{\theta}$.
Up the cyclotomic $\zp$-extension, there is a Coleman map 
$$
	\Cole \colon H^1_{\Iw,\loc}(\mcT_{\quo}(1)) \xrightarrow{\sim} \La \cozp \mfS_{\theta}.
$$ 
on the Iwasawa cohomology group
$H^1_{\Iw,\loc}(\mcT_{\quo}(1)) = H^1_{\loc}(\Lai \cozp \mcT_{\quo}(1))$.
Both $\mcS_{\theta}$ and $\mfS_{\theta}$ are canonically $P$ modulo $I$, and 
we fix an isomorphism $\mcS_{\theta} \xrightarrow{\sim} \mfS_{\theta}$ that reduces to the identity.

The remaining key ingredient from \cite{fk-pf} and \cite{fks2} is a zeta map
$$
	z \colon \La \cozp \mcS_{\theta} \to H^1_{\Iw}(\mcT_{\theta}(1))
$$ 
to Iwasawa cohomology over $\Q_{\infty}$.
Let $z_{\quo}$ be the composition of $z$ with the canonical map to
$H^1_{\Iw,\loc}(\mcT_{\quo}(1))$.
A $p$-adic regulator computation recalled from \cite{fk-pf} in Theorem \ref{zetamap} gives
\begin{equation} \label{colzeta}
	\Cole \circ z_{\quo} = \alpha \colon \La \cozp \mfS_{\theta} \to \La \cozp \mcS_{\theta}
\end{equation}
for a particular $\alpha \in \La \cozp \mf{h}_{\theta}$ with $\xi' = \alpha(0) \bmod I$.

The ground level maps $z^{\sharp}$ and $\Cole^{\flat}$  
compare with $z$ and $\Cole$ after applying evaluation-at-$0$ maps $\ev_0 \colon \La \to \zp$
and inverse limits of corestriction maps as follows:
\begin{eqnarray} \label{zetacolrelns}
	(1-U_p) z^{\sharp} \circ \ev_0 = \cor \circ z &\mr{and}& \Cole^{\flat} \circ \cor = (1-U_p) \ev_0 \circ \Cole
\end{eqnarray} 
(see Theorem \ref{zsharp} and Proposition \ref{Col_vs_flat}).
Combining \eqref{colzeta} and \eqref{zetacolrelns} allows one to show that
$$
	\Cole^{\flat} \circ z^{\sharp}_{\quo} = \alpha(0) \colon \mcS_{\theta} \to \mfS_{\theta},
$$ 
where $z^{\sharp}_{\quo}$ is the composition of $z^{\sharp}$ with $H^1(\mcT_{\theta}(1)) \to H^1_{\loc}(\mcT_{\quo}(1))$,
and $\Cole^{\flat} \circ z^{\sharp}_{\quo} \bmod I$ provides the $\xi' \colon P \to P$ on the right-hand side of \eqref{keyident}:
see Proposition \ref{zquosharp}.

\subsubsection{Equating the two sides}

Let $\bar{z}_P^{\sharp}$ denote the composition of $\bar{z}^{\sharp} \colon P \to H^1(T(1))$ with the canonical map $H^1(T(1)) \to H^1_{\loc}(P(1))$, i.e., the reduction of $z_{\quo}^{\sharp}$. Fukaya and Kato obtain a commutative diagram
\begin{equation} \label{rightside}
	\SelectTips{cm}{} \xymatrix{
	P \ar[r]^-{-\bar{z}_P^{\sharp}} \ar@{=}[d] & H^1_{\loc}(P(1)) \ar[r] \ar[rd]^-{-\Cole^{\flat}} & 
	H^2_c(P(1)) \ar[r]^{\partial} \ar@{=}[d] & H^3_c(P(1)) \ar@{=}[d] \\
	P \ar[rr]_{\xi'} & & P \ar[r]_1 & P,
	}
\end{equation}
where the middle horizontal arrow is from the Poitou-Tate sequence, and the right-hand one is the connecting map in the tensor product
of \eqref{Bocksteinseq} with $P(1)$ (see Lemmas \ref{globcompconn} and \ref{connectP}).

Additionally, cohomological lemmas (see Corollary \ref{commsquare} and Lemma \ref{conndiag}) provide the commutativity of
\begin{equation} \label{connectside}
	\SelectTips{cm}{} \xymatrix{
	H^1(T(1)) \ar[r] \ar[d] & H^1(Q(1)) \ar[r]^{\partial} \ar[d]^{-\partial} & H^2(Q(1)) \ar[d]^{\partial} \\
	H^1_{\loc}(P(1)) \ar[r] & H^2_c(P(1)) \ar[r]^{\partial} & H^3_c(P(1)).
	}
\end{equation}
Since $\bar{z}^{\sharp} \colon P \to H^1(T(1))$ 
induces both $\bar{z}_P^{\sharp}$ and $\bar{z}_Q^{\sharp}$, 
 the diagrams \eqref{leftside} and \eqref{rightside} can be combined using \eqref{connectside}, as in \eqref{FK_diag}. This yields 
Theorem \ref{FKident} (asserted in \cite{fks2}), which improves the same identity of Fukaya and Kato on $P \otimes \qp$.
\begin{thm}
	We have $\xi' \Upsilon \circ \varpi = \xi'$ as endomorphisms of $P$.
 \end{thm}

\subsubsection{Intermediate Coleman and local zeta maps}

In Section \ref{inter Coleman}, we construct intermediate Coleman maps sitting between $\Cole$ and $\Cole^{\flat}$ but
with properties similar to the latter. The central example for us is the following 
(see Corollary \ref{loccohdagS} of Theorem \ref{loccohdag} in the general setting).

\begin{thm}
	There is an isomorphism of $\La \cozp \mf{h}_{\theta}$-modules 
	$$
		\Cole^{\dagger} \colon H^1_{\loc}(\mcT_{\quo}^{\dagger}(1)) \xrightarrow{\sim} \mfS_{\theta}^{\star} 
		= \frac{(1-U_p,\alpha)(\La \cozp \mfS_{\theta})}{X\alpha(\La \cozp \mfS_{\theta})}
	$$
	that equals $(1-U_p)\Cole$ when precomposed with 
	$H^1_{\Iw,\loc}(\mcT_{\quo}(1)) \to H^1_{\loc}(\mcT_{\quo}^{\dagger}(1))$.
\end{thm}

The map $\Cole^{\dagger}$ is an amalgamation of $\Cole$ modulo $X\alpha$ with the invariant map of local class field theory. The composition 
$\overline{\Cole}^{\dagger}$ of $\Cole^{\dagger}$ with the surjection $\mfS_{\theta}^{\star} \to \mfS_{\theta}/(U_p-1)\mfS_{\theta}$ given by ``reduction modulo $1-U_p$ and division by $\alpha(0)$'' allows us to construct, somewhat artificially, a local intermediate zeta map
$$
	z_{\quo}^{\dagger} \colon \La \cozp \mcS_{\theta} \to H^1_{\loc}(\mcT_{\quo}^{\dagger}(1)) 
$$
with $\overline{\Cole}^{\dagger} \circ z_{\quo}^{\dagger} \colon \La \cozp \mcS_{\theta} \to \mfS_{\theta}/(U_p-1)\mfS_{\theta}$ the 
canonical quotient: see Proposition \ref{interloczeta}. Its reduction modulo $I$ induces $1$ on $P$, which now replaces $\xi'$ on the right of \eqref{keyident}.

\subsubsection{Reduced intermediate zeta maps}

Our refinement of \eqref{rightside}, the cohomological lemmas for which are proven in Section \ref{refinedglobcoh}, is a commutative diagram
$$
	\SelectTips{cm}{} \xymatrix{
	\La \cozp P \ar@{->>}[d]^-{\ev_0} \ar[r]^-{-\bar{z}_P^{\dagger}} & H^1_{\loc}(P^{\dagger}(1)) \ar[r]\ar[rd]^-{-\overline{\Cole}^{\dagger}} & 
	H^2_c(P^{\dagger}(1)) \ar[r]^{\partial} 
	\ar@{=}[d] & H^3_c(P(1)) \ar@{=}[d] \\
	P \ar[rr]_1 & & P \ar[r]_1 & P,
	}
$$
where $\bar{z}_P^{\dagger}$ is the mod $I$ reduction of $\bar{z}_{\quo}^{\dagger}$.
However, we do not know whether an intermediate reduced zeta map 
$\bar{z}^{\dagger}$ that induces both $\bar{z}_P^{\dagger}$  and $\bar{z}_Q^{\dagger}$ exists.
That is, the tie that would bind the two sides of the identity $\Upsilon \circ \varpi = 1$ 
is missing. 
Our main theorem (Theorem \ref{equivform}) is as follows.

\begin{thm}
	Conjecture \ref{sconj} holds if and only if there exists a $\La \otimes_{\zp} (\mf{h}/I)_{\theta}$-module homomorphism
	$$
		\bar{z}^{\dagger} \colon \Lambda \otimes_{\zp} P \to H^1(T^{\dagger}(1))
	$$ 
	compatible with both $\bar{z}^{\sharp}$
	and the reduction modulo $I$ of $z^{\dagger}_{\quo}$.
\end{thm}

In an ideal world, $\bar{z}^{\dagger}$ would be the reduction of an intermediate zeta map
$z^{\dagger} \colon \La \cozp \mcS_{\theta} \to H^1(\mcT_{\theta}^{\dagger}(1))$ that induces $z_{\quo}^{\dagger}$ and
for which $(1-U_p)z^{\dagger}$ agrees with the composition of $z$ with $H^1_{\Iw}(\mcT_{\theta}(1)) \to H^1(\mcT_{\theta}^{\dagger}(1))$.
Unlike the existence of $\bar{z}^{\dagger}$, this seems somewhat unlikely to hold in all cases. See Section \ref{refinedglobcoh} and the test case of Section \ref{testcase} in the cyclotomic setting.

\section{Background} \label{background}

In this section, we introduce many of our objects of study, both modular and Galois cohomological, and known results on them. 

\subsection{Ordinary Hecke modules} \label{Eisparts}

\begin{notation} \label{pNcond}
	Let $p \ge 5$ be a prime and $N \ge 4$ a positive integer with $p \nmid N\varphi(N)$, for $\varphi$ the Euler-phi
	function.
\end{notation}

As described, for instance, in \cite[Section 2]{ohta-es}, let $\mf{h}^{\ord}$ denote Hida's $\zp$-Hecke algebra acting on the space $\mf{S}^{\ord}$ of ordinary ``$\La$-adic'' cusp forms of level $Np^{\infty}$.  Similarly, let $\mf{H}^{\ord}$ denote Hida's Hecke algebra acting on the space $\mf{M}^{\ord}$ of ordinary $\La$-adic forms of level $Np^{\infty}$.  Note that $\mf{h}^{\ord}$ is a quotient of $\mf{H}^{\ord}$.

In addition to $\mfS^{\ord}$, we have two related $\mf{h}^{\ord}$-modules.

\begin{definition} The following inverse limits are taken with respect to trace maps. 
	\begin{enumerate}
		\item[a.] We let $\mcS^{\ord}$ denote the fixed part under complex conjugation (or ``plus part'', denoted by
		the superscript ``$+$'')
		$$
			\mcS^{\ord} = \varprojlim_r H_1(X_1(Np^r)(\C),\zp)^{\ord,+}
		$$
		of the space of ordinary $\La$-adic cuspidal modular symbols.
		\item[b.] We let $\mcT^{\ord}$ denote the inverse limit
		$$
			\mcT^{\ord} = \varprojlim_r H^1_{\et}(X_1(Np^r)_{/\qbar},\zp(1))^{\ord}
		$$
		of ordinary parts of first \'etale cohomology groups of the closed modular curves $X_1(Np^r)$.
	\end{enumerate}
\end{definition}

\begin{remark}
	Viewing $\qbar$ as the algebraic numbers in $\C$,
	we have an isomorphism $\mcT^{\ord,+} \cong \mcS^{\ord}$ of $\mf{h}^{\ord}$-modules, induced by the usual 
	comparisons of \'etale and Betti cohomology at the individual stages of the tower. 
	We note that Hecke actions on inverse limits of cohomology (as opposed to homology) groups are via the dual, or adjoint, operators.
\end{remark}

Similarly but less crucially for our purposes, we have the following $\mf{H}^{\ord}$-modules.

\begin{definition}\
	\begin{enumerate}
		\item[a.] We let $\mc{M}^{\ord}$ denote the plus part
		$$
			\mc{M}^{\ord} = \varprojlim_r H_1(X_1(Np^r)(\C),\{\mr{cusps}\},\zp)^{\ord,+}
		$$
	 	of the space of ordinary $\La$-adic modular symbols.
		\item[b.] We let $\tmcT^{\ord}$ denote the inverse limit
		$$
			\tmcT^{\ord} = \varprojlim_r H^1_{\et}(Y_1(Np^r)_{/\qbar},\zp(1))^{\ord}
		$$
		of ordinary parts of \'etale cohomology groups of open modular curves.
		Similarly, we let $\tmcT_c^{\ord}$ denote the inverse limit of the ordinary parts of the compactly
		supported \'etale cohomology groups $H^1_{c,\et}(Y_1(Np^r)_{/\qbar},\zp(1))$.
	\end{enumerate}
\end{definition}

\begin{remark}
	As in the cuspidal case, the $\mf{H}^{\ord}$-modules $\mc{M}^{\ord}$ and $\tmcT^{\ord,+}$ are isomorphic.
\end{remark}

Throughout this paper, let us use $G_K$ to denote the absolute Galois group of a field $K$.

\begin{remark} \label{poincare}
	Since signs are quite subtle in this work, we mention some conventions of algebraic topology 
	used here and in \cite{fk-pf} (cf. \cite[2.7]{kato}), as well as some calculations which follow from them. 
	Consider the compatible $G_{\Q}$-equivariant Poincar\'e duality pairings on \'etale cohomology:
	\begin{eqnarray*}
		&H^1_{\et}(X_1(Np^r)_{/\qbar},\zp(1)) \times H^1_{\et}(X_1(Np^r)_{/\qbar},\zp(1)) \xrightarrow{\cup} \zp(1), \\
		&H^1_{\et}(Y_1(Np^r)_{/\qbar},\zp(1)) \times H^1_{c,\et}(Y_1(Np^r)_{/\qbar},\zp(1)) \xrightarrow{\cup} \zp(1).
	\end{eqnarray*}
	Viewing $\qbar$ as the algebraic numbers in $\C$, these are compatible with the usual pairings of Poincar\'e duality for 
	the isomorphic Betti cohomology groups of the $\C$-points of the modular curves, which are given by evaluation of the cup product on a 
	fundamental class given by the standard orientation  of the Riemann	surface $X_1(Np^r)(\C)$. 
	These cup products induce identifications
	\begin{eqnarray*}
		&H^1_{\et}(X_1(Np^r)_{/\qbar},\zp(1)) \xrightarrow{\sim} H_1(X_1(Np^r)(\C),\zp), \\
		&H^1_{\et}(Y_1(Np^r)_{/\qbar},\zp(1)) \xrightarrow{\sim} H_1(X_1(Np^r)(\C),\{\mr{cusps}\},\zp)
	\end{eqnarray*}
	that take a class $g$ to the unique homology class $\gamma$ such that the map $h \mapsto g \cup h$
	agrees with evaluating the cohomology class $g$ on $\gamma$.
	
	Now, any unit $g$ on $Y_1(Np^r)_{/\qbar}$ gives rise via Kummer theory to a similarly denoted
	class in $H^1_{\et}(Y_1(Np^r)_{/\qbar},\zp(1))$. The order $\ord_x\, g$ of the zero of $g$
	at a cusp $x$ satisfies
	\begin{equation} \label{poincareident}
		 \ord_x\, g = g \cup h_x = \partial_x g,
	\end{equation}
	where $h_x \in H^1_{c,\et}(Y_1(Np^r)_{/\qbar},\zp(1))$ is the image of $x$ under the canonical connecting map,
	and where $\partial_x g$ is the boundary at $x$ in $H_0(\{x\},\zp) \cong \zp$
	of the relative homology class corresponding to $g$.
	These identities can be verified by comparison with de Rham cohomology:
	$$
		g \cup h_x = \frac{1}{2\pi i} \int_X \frac{dg}{g} \wedge d\eta_x = \frac{1}{2\pi i} \int_{\partial D_x} \frac{dg}{g} = \ord_x \, g
	$$
   	for a smooth function $\eta_x$ that is $1$ on a small closed disk $D_x$ about $x$ and $0$ outside of 
	a larger one in $Y_1(Np^r)$.
	On the other hand, if $g$ is sent to the class of $\gamma$, then 
	$$
		g \cup h_x = \int_{\gamma} d\eta_x = \sum_y \partial_y g \cdot \eta_x(y) = \partial_x g,
	$$
	where the sum is taken over all cusps $y$ of $X_1(Np^r)$.
\end{remark}

\subsection{Iwasawa modules}

\begin{definition}
	Set $\Z_{p,N} = \varprojlim_r \Z/Np^r\Z$, set $\tilde{\La} = \zp\ps{\Z_{p,N}^{\times}/\langle -1 \rangle}$, and
	set $\Delta = (\Z/Np\Z)^{\times}$.
\end{definition}

Note that we have a canonical decomposition $\Z_{p,N}^{\times} \cong \Delta \times (1+p\Z_p)$.

\begin{definition} Set $\Lambda = \zp\ps{1+p\zp}$, let $\chi$ denote the isomorphism
	$$
		\chi = (1-p^{-1})\log \colon 1+p\zp \xrightarrow{\sim} \zp,
	$$  
	let $t \in 1+p\zp$ be such that $\chi(t) = 1$, let $\gamma \in \Lambda$ be the group element defined by $t$,
	and set $X = \gamma-1 \in \La$.
\end{definition}

Note that these definitions allow us to consider $\tilde{\Lambda}$ as the $\Lambda = \zp\ps{X}$-algebra
$\Lambda[ \Delta/\langle -1 \rangle]$.
	
\begin{definition} \
	\begin{enumerate}
		\item[a.] Set $\tilde{\Gamma} = \Gal(\Q(\mu_{Np^{\infty}})^+/\Q)$.
		\item[b.] Let $\Q_{\infty}$ denote the cyclotomic $\zp$-extension of $\Q$, and set 
		$\Gamma = \Gal(\Q_{\infty}/\Q)$.
	\end{enumerate}
\end{definition}

We have an isomorphism $\tilde{\Gamma} \xrightarrow{\sim} \Z_{p,N}^{\times}/\langle -1 \rangle$  
given by the cyclotomic character, which we use to identify $\tilde{\La}$ with $\zp\ps{\tilde{\Gamma}}$.
We similarly identify $\La$ with $\zp\ps{\Ga}$.  We also use this isomorphism to identify $\Delta/\langle -1 \rangle$
with a subgroup (and quotient) of $\tilde{\Gamma}$.

\begin{remark}
	Note that $\mf{h}^{\ord}$ is a $\tilde{\La}$-algebra on which group elements act as inverses of 
	diamond operators. (This choice of inverses, as opposed to actual diamond operators, is made so that the
	maps that feature in Conjecture \ref{sconj} that is the subject of this work are of $\tilde{\La}$-modules.)
	At times, we may work with $\tilde{\La}$-modules with distinct actions of inverse diamond operators
	and Galois elements.  The action that we are considering should be discernable from context.
\end{remark}

\begin{definition} \
	\begin{enumerate}
		\item[a.] Let $\Z_{\infty}$ denote the integer ring of $\Q_{\infty}$.
		\item[b.] Set $\mc{O} = \Z[\frac{1}{Np}]$ and
		$\mc{O}_{\infty} = \Z_{\infty}[\tfrac{1}{Np}]$. For $r \ge 1$, let $\mc{O}_r$ be the ring of $Np$-integers in the
		degree $p^{r-1}$ extension of $\Q$ in $\Q_{\infty}$.
		\item[c.] Let $\Q_{p,\infty}$ denote the cyclotomic $\zp$-extension of $\Q_p$, and
		let $\Q_{p,r}$ denote the unique degree $p^{r-1}$ extension of $\Q_p$ in $\Q_{p,\infty}$.
		\item[d.] For each prime $\ell \neq p$, let $\Q_{\ell,\infty}$ denote the unramified $\zp$-extension of $\Q_{\ell}$.
	\end{enumerate}
\end{definition}

\begin{definition}
	For any algebraic extension $F$ of $\Q$, we consider the set
	$S$ of primes dividing $Np\infty$.  We let $G_{F,S}$ denote the Galois group of the maximal
	$S$-ramified (i.e., unramified outside $S$) extension of $F$.
\end{definition}

\begin{remark}
	For a compact $\zp\ps{G_{\Q,S}}$-module $M$, we use the notation $H^i(\mc{O},M)$ to denote
	the $i$th continuous Galois cohomology group $H^i_{\mr{cts}}(G_{\Q,S},M)$.
	This is consistent with the common interpretation
	of this group as the continuous \'etale cohomology group of the locally constant sheaf on the \'etale
	site of $\Spec \mc{O}$
	attached to $M$ (cf. \cite[Proposition II.2.9]{milne}). Here, note that 
	the continuous Galois and \'etale cohomology groups of $M$ 
	are isomorphic to the corresponding inverse limits of cohomology groups of the finite 
	$G_{\Q,S}$-quotients of $M$ \cite[\S 3.2]{lim}.
	
	We also have compactly supported and local-at-$\ell$
	cohomology groups, denoted 
	\begin{eqnarray*}
		H^i_c(\mc{O},M) = H^i_{c,\cts}(G_{\Q,S},M) &\mr{and}& 
		H^i(\Q_{\ell},M) = H^i_{\cts}(G_{\Q_{\ell}},M)
	\end{eqnarray*} 
	respectively, 
	in the latter case for $M$ a compact $\zp\ps{G_{\Q_{\ell}}}$-module.
	Since $p$ is odd, the compactly supported Galois cohomology groups agree with compactly 
	supported \'etale cohomology in its usual sense (cf. \cite[\S II.2]{milne}).
	
	We extend this notation to define $H^i(\mc{O}_r,M)$ as the continuous 
	$S$-ramified cohomology group of the fraction field of $\mc{O}_r$ with $M$-coefficients, and likewise in other
	settings.
\end{remark}

We may view $\tilde{\Gamma}$ as a quotient of $G_{\Q,S}$.

\begin{definition}
	For a $\tilde{\La}$-module $M$, we consider $M$ as a $\tilde{\La}\ps{G_{\Q,S}}$-module $M^{\iota}$ by 
	letting $\sigma \in G_{\Q,S}$ act by multiplication by the inverse of its image in $\tilde{\Gamma}$.
\end{definition}

We may define Iwasawa cohomology groups by taking completed tensor products with $\Lai$: see \cite{me-selmer}.
(Note that $\sigma \in G_{\Q,S}$ acts on $\Lai$ by multiplication by the inverse of its image in $\Gamma$.)

\begin{definition}
	For a compact $\zp\ps{G_{\Q,S}}$-module $M$, the $i$th \emph{$S$-ramified Iwasawa cohomology group} of $M$ is
	$$
		H^i_{\Iw}(\mc{O}_{\infty},M) = H^i(\mc{O}, \Lai \cozp M). 
	$$
\end{definition}

\begin{remark}
	We have compactly supported Iwasawa cohomology groups and local-at-$\ell$, for primes $\ell$,
	Iwasawa cohomology groups
	\begin{eqnarray*}
		H^i_{c,\Iw}(\mc{O}_{\infty},M) = H^i_c(\mc{O}, \Lai \cozp M) &\mr{and}&
		H^i_{\Iw}(\Q_{\ell,\infty},M)= H^i(\Q_{\ell}, \zp\ps{\Gamma_{\ell}}^{\iota} \cozp M)
	\end{eqnarray*}
	of a compact $\zp\ps{G_{\Q,S}}$-module or, respectively, $\zp\ps{G_{\Q_{\ell}}}$-module $M$,
	where $\Gamma_{\ell}$ denotes the decomposition group at $\ell$ in $\Gamma$.
	We will also consider Iwasawa cohomology for $\mc{O}_{\infty}[\mu_{Np}]$, defined
	using $\tilde{\Lambda}$ in place of $\La$ (and similarly, local-at-$\ell$ cohomology for the
	decomposition group).
\end{remark}

We make some remarks on of Galois cohomology with restricted ramification, 
Poitou-Tate duality, and Iwasawa cohomology. We refer the reader to \cite{nsw}, \cite{lim}, and \cite[Section 2]{me-selmer} for a more thorough review of their various properties, as well as, for instance, the analogous but simpler Tate duality.

\begin{remark} \label{galcohfacts}
	Let $M$ be a compact $\zp\ps{G_{\Q,S}}$-module.
	\begin{enumerate}
		\item[a.] Via Shapiro's lemma, we have an identification
		$$ 	
			H^i_{\Iw}(\mc{O}_{\infty},M) \cong \varprojlim_r H^i(\mc{O}_r, M)
		$$
		where the inverse limit is taken with respect to corestriction maps,
		and similarly for the other sorts of Iwasawa cohomology groups.
		For a prime $\ell$, the group $H^i(\Q_{\ell},\Lai \cozp M)$ is 
		the product of Iwasawa cohomology groups $H^i_{\Iw}(\Q_{\ell,\infty},M)$ 
		over the finite set of primes of $\Q_{\infty}$ dividing $\ell$. For $\ell = p$, this
		is just a single prime.
		\item[b.] Poitou-Tate duality provides canonical isomorphisms 
		$$
			H^i_c(\mc{O},M(1)) \cong H^{3-i}(G_{\Q,S},M^{\vee})^{\vee},
		$$
		where on the right we use the profinite $G_{\Q,S}$-cohomology group of the discrete Pontryagin dual $M^{\vee}$. 
		This in turn gives isomorphisms
		$$
			H^i_{c,\Iw}(\mc{O}_{\infty},M(1)) \cong H^{3-i}(G_{\Q_{\infty},S},M^{\vee})^{\vee}
		$$ 
		of $\La$-modules.
		For $i = 3$, we obtain the \emph{invariant maps} 
		\begin{eqnarray*}
			H^3_c(\mc{O},M(1)) \cong M_{G_{\Q,S}} &\mr{and}&
			H^3_{c,\Iw}(\mc{O}_{\infty},M(1)) \cong M_{G_{\Q_{\infty},S}},
		\end{eqnarray*}
		and in particular third compactly supported cohomology functors are right exact.
	\end{enumerate}
\end{remark}

\subsection{Local actions at $p$} \label{localact}

\begin{notation}
	We fix an even $p$-adic Dirichlet character $\theta \colon \Delta \to \qpbt$
	and let $R$ denote the $\zp$-algebra generated by the values of $\theta$.  
\end{notation}

We consider $R$ as a quotient of $\zp[\Delta]$ via the $\zp$-linear map to $R$ 
induced by $\theta$.

\begin{definition}
	The \emph{$\theta$-part} $M_{\theta}$ of a $\zp[\Delta]$-module $M$ is the $R$-module $M_{\theta} = M \otimes_{\zp[\Delta]} R$.
\end{definition}
 
\begin{remark}
	Given a compact $\tilde{\Lambda}$-module $M$, we view $M_{\theta}$ as a module over the complete local ring
	$\La_{\theta} := R\ps{\Gamma} = R\ps{X}$.  
	We will most typically think of $\La_{\theta}$ as the $\theta$-part of the algebra of inverse diamond 
	operators, whereas $\La$ will often be viewed as an algebra of Galois elements.
\end{remark}

\begin{definition} \label{Tquosub} \
\begin{enumerate}
	\item[a.] Let $\mcT_{\quo}^{\ord}$ and $\tmcT_{\quo}^{\ord}$ denote the maximal 
	unramified $\mf{H}_{\theta}\ps{G_{\qp}}$-quotients of $\mcT_{\theta}^{\ord}$ and $\tmcT_{\theta}^{\ord}$, respectively.
	\item[b.] Let $\mcT_{\sub}^{\ord}$ denote the kernel of the quotient map $\mcT_{\theta}^{\ord} \to 
	\mcT_{\quo}^{\ord}$, 
	which is also the kernel of $\tmcT_{\theta}^{\ord} \to \tmcT_{\quo}^{\ord}$.  
\end{enumerate}
\end{definition}

\begin{definition} \label{Ddef}
	For a compact unramified $\mf{R}\ps{G_{\qp}}$-module $U$ with $\mf{R}$ a compact $\zp$-algebra, we set
	$$
		D(U) = (U \cozp W)^{\Fr_p = 1},
	$$
	i.e., the fixed part of the completed tensor product
	for the diagonal action of the Frobenius $\Fr_p$, 
	where $W$ is the completion of the valuation ring of $\qp^{\ur}$.
\end{definition}

\begin{remark}\label{Disos}
	In the notation of Definition \ref{Ddef}, the following hold.
	\begin{enumerate}
		\item[a.] There is a (noncanonical) natural isomorphism between the forgetful functor from 
		compact unramified $\mf{R}\ps{G_{\qp}}$-modules to compact
		$\mf{R}$-modules and $D$, under which each $U \to D(U)$ is an isomorphism \cite[1.7.6]{fk-pf}.  
		\item[b.] Endowing $D(U)$ for each $U$ with the additional action of $\varphi = 1 \otimes \Fr_p$, 
		any choice of natural isomorphism as above induces canonical isomorphisms
		$$
			U/(1-\Fr_p)U \xrightarrow{\sim} D(U)/(1-\varphi)D(U).
		$$ 
	\end{enumerate}
\end{remark}
 
The following $\La$-adic Eichler-Shimura isomorphisms can be found in \cite[1.7.9]{fk-pf} and extend work of Ohta from \cite{ohta-es}.

\begin{theorem}[Ohta, Fukaya-Kato] \label{DTquo}
	We have canonical isomorphisms 
	\begin{eqnarray*}
		D(\mcT_{\quo}^{\ord}) \cong \mf{S}_{\theta}^{\ord} &\mr{and}& 
		D(\tmcT_{\quo}^{\ord}) \cong \mf{M}_{\theta}^{\ord}
	\end{eqnarray*}	
	of $\mf{H}_{\theta}^{\ord}$-modules.
\end{theorem}

\begin{remark} \label{Lafree}
	A well-known result of Hida theory (see Ohta \cite[Theorem 1.4.3]{ohta-es}) states that $\mcT_{\theta}^{\ord}$
	and $\tmcT_{\theta}^{\ord}$ are $\La_{\theta}$-free of finite rank.  
\end{remark}

Ohta \cite[Section 4]{ohta-es} constructed a perfect ``twisted Poincar\'e duality'' pairing 
\begin{equation} \label{Ohta_pair}
	\langle\, \ ,\ \rangle \colon \mcT_{\theta}^{\ord} \times \mcT_{\theta}^{\ord} \to \La_{\theta}^{\iota}(1)
\end{equation}
of $\La_{\theta}[G_{\Q,S}]$-modules for which $(Tx,y) = (x,Ty)$ for all $x, y \in \mcT_{\theta}$ and $T \in \mf{h}_{\theta}^{\ord}$.
This is compatible with an analogously defined pairing
\begin{equation} \label{Ohta_pair2}
	\langle\, \  ,\ \rangle \colon \tmcT_{\theta}^{\ord} \times \tmcT_{c,\theta}^{\ord} \to \Lai_{\theta}(1)
\end{equation}
of Ohta \cite[Theorem 1.3.3]{ohta-cong} with the same properties, but taking $T \in \mf{H}_{\theta}^{\ord}$.

\begin{remark}
	The submodule $\mcT_{\sub}^{\ord}$ is isotropic with respect to Ohta's pairing, yielding a perfect 
	$\La_{\theta}$-duality between the $\mf{h}_{\theta}^{\ord}\ps{G_{\qp}}$-modules 
	$\mcT_{\sub}^{\ord}$ and $\mcT_{\quo}^{\ord}$ \cite[Theorem 4.3.1]{ohta-es}.
	As a consequence, $\mcT_{\sub}^{\ord}$ is isomorphic
	to $\mf{h}_{\theta}^{\iota}(1)$ as an $\mf{h}_{\theta}^{\ord}\ps{G_{\qp^{\ur}}}$-module, 
	where $\qp^{\ur}$ denotes the maximal unramified extension of $\qp$ \cite[1.7.13]{fk-pf}.
\end{remark}

\subsection{Eisenstein parts and quotients}

For an $\mf{H}^{\ord}$-module $M$, we let $M_{\mf{m}}$ denote its Eisenstein part: the product of its localizations at the maximal ideals containing $T_{\ell}-1-\ell\langle \ell \rangle$ for primes $\ell \nmid Np$
and $U_{\ell}-1$ for primes $\ell \mid Np$.

\begin{definition} \
	\begin{enumerate}
		\item[a.] We define the cuspidal Hecke algebra $\mf{h}$ as the Eisenstein part 
		$\mf{h}_{\mf{m}}^{\ord}$ of Hida's ordinary cuspidal
		Hecke algebra $\mf{h}^{\ord}$.
		\item[b.] The \emph{Eisenstein ideal} $I$ of $\mf{h}$ is the ideal generated 
		by $T_{\ell}-1-\ell\langle \ell \rangle$ for primes $\ell \nmid Np$ and $U_{\ell}-1$ for primes $\ell \mid Np$.
	\end{enumerate}
\end{definition}

We also set $\mf{H} = \mf{H}_{\mf{m}}^{\ord}$ and in general use the following notational convention.  

\begin{notation}
	For an $\mf{H}^{\ord}$-module denoted $M^{\ord}$, we set $M = M^{\ord}_{\mf{m}}$.  
\end{notation}

By applying this convention, we obtain $\mf{H}$-modules $\mf{S}$, $\mf{M}$, $\mc{S}$, $\mc{M}$, $\mcT$, $\tmcT$, $\mc{T}_{\quo}$, $\tmcT_{\quo}$, and $\mc{T}_{\sub}$.  (Note that $\mc{T}_{\sub}$ and $\mc{T}_{\quo}$
are a submodule and a quotient of $\mcT_{\theta}$, rather than just $\mcT$.)
It is only these Eisenstein parts that will be of use to us in the rest of the paper, so we focus solely on them, eschewing greater generality, but obtaining somewhat finer results in the later consideration of zeta elements.

We make the following assumptions on our even character $\theta$.

\begin{hypothesis} \label{thetahyp} We suppose that the following conditions on $\theta$ hold:
	\begin{enumerate}
		\item[a.] $p$ divides the generalized Bernoulli number $B_{2,\theta^{-1}}$.
		\item[b.] $\theta$ has conductor $N$ or $Np$,
		\item[c.] $\theta \neq 1, \omega^2$ (if $N = 1$),
		\item[d.] either $\theta\omega^{-1}|_{(\Z/p\Z)^{\times}} \neq 1$ or 
		$\theta|_{(\Z/N\Z)^{\times}}(p) \neq 1$, 
	\end{enumerate}
\end{hypothesis}

\begin{remark}	
	Hypothesis \ref{thetahyp}a tells us that $\mf{h}_{\theta} \neq 0$.
\end{remark}

Using Hypothesis \ref{thetahyp}d, we have the following exactly as in the work of Ohta \cite[\S 3.4]{ohta-cong} (cf. \cite[6.3.12]{fk-pf}).

\begin{lemma} \label{latticesplit}
	The exact sequence 
	$$
		0 \to \mcT_{\sub} \to \mcT_{\theta} \to \mcT_{\quo} \to 0
	$$
	is canonically split as a sequence of $\mf{h}_{\theta}$-modules.
\end{lemma}

We consider the following power series corresponding to the Kubota-Leopoldt $p$-adic $L$-function of interest.

\begin{definition}
	Let $\xi \in \La_{\theta}$ be the element characterized by the property that
	$$
		\xi(t^s-1) = L_p(\omega^2\theta^{-1},s-1)
	$$
	for all $s \in \zp$.
\end{definition}

The Mazur-Wiles proof of the main conjecture over $\Q$ implies the following, first stated by Wiles \cite[Theorem 4.1]{wiles} in the more general context of totally real fields, and reproven in the Mazur-Wiles setting by Ohta in \cite[Corollary A.2.4]{ohta-cong}.
	
\begin{theorem}[Wiles] \label{redhecke}
	We have $(\mf{h}/\Eis)_{\theta} \cong \La_{\theta}/\xi$.  
\end{theorem}

\begin{definition}
	We let $T = \mcT_{\theta}/\Eis\mcT_{\theta}$.  
\end{definition}

We recall the following from \cite[Section 6.3]{fk-pf} (cf. \cite[Corollary 4.9]{me-Lfn}).
	
\begin{proposition} \label{Qiso}
	The reduced lattice $T$ has an $(\mf{h}/\Eis)_{\theta}\ps{G_{\Q,S}}$-quotient $Q$ canonically isomorphic to 
	$(\mf{h}/\Eis)_{\theta}^{\iota}(1)$.
\end{proposition}

\begin{proof}
	Consider the Manin-Drinfeld modification of the inverse limit of the first homology groups of $X_1(Np^r)$
	relative to the cusps, which is isomorphic to $\tmcT \otimes_{\mf{H}} \mf{h}$ by \cite[Lemma 4.1]{me-Lfn}.
	Its quotient by $\mcT$ is isomorphic to $\mf{h}/I$, generated by 
	the image $e_{\infty}$ of the compatible sequence 
	of relative homology classes $\{0 \to \infty\}_r$ of the geodesic paths from $0$ to $\infty$ in the
	upper half-plane \cite[Lemma 4.8]{me-Lfn}. The $\La_{\theta}$-module $\tmcT \otimes_{\mf{H}} \mf{h}$ is free,
	as it has no $X$-torsion and its quotient by $X$ is $R$-free as the Manin-Drinfeld modification
	of the Eisenstein part of the relative homology of $X_1(Np)$ (cf. \cite[(6.2.9)]{fk-pf}).
	By Remark \ref{Lafree} and Theorem \ref{redhecke}, we then see that $\xi e_{\infty}$ must be an element
	of a $\La_{\theta}$-basis of $\mcT_{\theta}$ (cf. \cite[(6.2.10)]{fk-pf}).  The desired surjection is given by
	$y \mapsto \langle \xi e_{\infty},y \rangle$ on $y \in \mcT_{\theta}$, using the nondegeneracy of Ohta's pairing \eqref{Ohta_pair}.
\end{proof}

\begin{remark} \label{signchange}
	We have made a sign change here from our original map and that of \cite[6.3.18]{fk-pf}. That is, we pair
	with $\xi e_{\infty}$ on the left, rather than the right. 
\end{remark}

We define $P$ as the kernel of the quotient map $T \to Q$, yielding an exact sequence
\begin{equation} \label{globmodeis}
	0 \to P \to T \to Q \to 0
\end{equation}
of $(\mf{h}/\Eis)_{\theta}\ps{G_{\Q,S}}$-modules.
We recall the following from the main results of \cite[Section 6.3]{fk-pf}.

\begin{proposition} \label{exseqredcoh}
	The canonical maps $P \to \mcT_{\quo}/\Eis\mcT_{\quo}$ and $\mcT_{\sub}/\Eis\mcT_{\sub} \to Q$ are
	isomorphisms of $(\mf{h}/I)_{\theta}\ps{G_{\qp}}$-modules.  Moreover, the action of $G_{\Q,S}$ on 
	$P$ is trivial, and $P$ can be identified with the fixed part of $T$ under any complex conjugation.
\end{proposition}

\begin{proof}
	The cokernel of the map $\pi \colon \mcT_{\sub}/\Eis\mcT_{\sub} \to Q$ is an 
	$(\mf{h}/\Eis)_{\theta}\ps{G_{\qp}}$-module
	quotient of $\mcT_{\quo}/\Eis\mcT_{\quo}$.  The $\Delta_p$-action on $\mcT_{\quo}/\Eis\mcT_{\quo}$
	is trivial, while the $\Delta_p$-action on $Q$ is via $\omega\theta^{-1}$, so by Hypothesis \ref{thetahyp}d,
	we have that $\pi$ is surjective.  Moreover, $\mcT_{\sub}/\Eis\mcT_{\sub}$ and $Q$ are both free
	of rank one over $(\mf{h}/\Eis)_{\theta}$, so $\pi$ must also be injective.
	This forces the other map to be an isomorphism as well.
	
	Next, let us briefly outline the argument of Kurihara and Harder-Pink yielding the triviality of the action on $P$,
	as in \cite[6.3.15]{fk-pf}.
	By Lemma \ref{latticesplit}, we have a direct sum decomposition $T = P \oplus Q$ as 
	$(\mf{h}/I)_{\theta}$-modules, with $P$ being $G_{\Q,S}$-stable.
	The character defining the determinant of the action of $G_{\Q,S}$ on the modular representation
	in which $\mcT_{\theta}$ is a lattice reduces exactly to the character defining the action on $Q$.  
	Consequently, $G_{\Q,S}$ must act trivially on $P$.  Since complex conjugation then acts trivially on $P$
	and as $-1$ on the quotient $Q$, we have the final claim.
\end{proof}

\begin{corollary}
	The maps $\mcT_{\theta}^+ \to \mcT_{\quo}$ and $\tmcT_{\theta}^+ \to \tmcT_{\quo}$ 
	are isomorphisms.  
\end{corollary}

\begin{proof}
	The maps $\mcT_{\theta}/\mcT_{\theta}^+ \to \tmcT_{\theta}/\tmcT_{\theta}^+$ and $\mcT_{\sub} \to
	\tmcT_{\sub}$ are isomorphisms, so it suffices 
	to show that $\mcT_{\sub} \to \mcT_{\theta}/\mcT_{\theta}^+$
	is an isomorphism.  We know that it is surjective by Proposition \ref{exseqredcoh} and Nakayama's lemma.
	But $\mcT_{\sub}$ is a free $\mf{h}_{\theta}$-module of rank $1$, and 
	$\mcT_{\theta}/\mcT_{\theta}^+$ is an $\mf{h}_{\theta}$-module of rank $1$, so the surjectivity forces
	the map to be an isomorphism.
\end{proof}

As in \cite[6.3.4]{fk-pf}, we see that our sequence \eqref{globmodeis} is uniquely locally split.

\begin{proposition} \label{exseqredunr}
	For each prime $\ell \mid Np$, the sequence \eqref{globmodeis} is uniquely split as a sequence of 
	$(\mf{h}/\Eis)_{\theta}\ps{G_{\Q_{\ell}}}$-modules.
\end{proposition}
		
\begin{proof}
	For $\ell = p$, this is a direct consequence of Proposition \ref{exseqredcoh}.  For $\ell \mid N$,
	this follows from Hypothesis \ref{thetahyp}b and the facts that the decomposition group 
	$\Delta_{\ell}$ at $\ell$ in $\Delta$ acts trivially on $P$ and via $\omega\theta^{-1}$ on $Q$.
\end{proof}		
		
We also have the following results on $P$.
		
\begin{remark} \label{quo_results} \
	\begin{enumerate}
		\item[a.] The $G_{\Q,S}$-action on $P$ is trivial, and we have 
		a canonical isomorphism $P \cong \mf{S}_{\theta}/\Eis\mf{S}_{\theta}$ 
		of $\mf{h}_{\theta}$-modules. For this, note that $U_p$ acts as an arithmetic Frobenius on 
		$\mcT_{\quo}$ by \cite[1.8.1]{fk-pf} and that $D(\mcT_{\quo}) \cong \mf{S}_{\theta}$, and apply
		Proposition \ref{exseqredcoh} and Remark \ref{Disos}b.
		\item[b.] The $p$-adic $L$-function $\xi$ divides the $\La_{\theta}$-characteristic ideal 
		of $P$ (for the action of inverse diamond operators) by an argument of Mazur-Wiles and Ohta 
		(see \cite[7.1.3]{fk-pf}).
	\end{enumerate}
\end{remark}

Putting these isomorphisms together with 
Remark \ref{Disos}a and Proposition \ref{DTquo}, we have isomorphisms 
$\mfS_{\theta} \cong \mcT_{\theta}^+ \cong \mcS_{\theta}$ and 
$\mf{M}_{\theta} \cong \tmcT_{\theta}^+ \cong \mc{M}_{\theta}$ on Eisenstein components.  Note that the first of each of these
pairs of isomorphisms is noncanonical, only becoming canonical upon reduction modulo $U_p-1$, but we can and do fix compatible choices.

\section{Cohomological study} \label{galcoh}

In this section, we first introduce known results on the cohomology of the reduced lattice that is the quotient $T$ of $\mcT_{\theta}$ by the Eisenstein ideal. 
We recall the work of Fukaya and Kato \cite{fk-pf} in which the derivative $\xi'$ of a Kubota-Leopoldt $p$-adic $L$-function $\xi$ appears in the study of certain connecting homomorphisms in the cohomology of subquotients of $T(1)$.  We then perform an analogous study, replacing $T$ by the ``intermediate'' quotient 
$$
	T^{\dagger} = (\Lai \cozp T)/\otil{\xi}(\Lai \cozp T) 
$$
of $\Lai \cozp T$, where $\otil{\xi}$ is a diagonalization of $\xi$ in $\La \cozp \La_{\theta}$. We show that in this setting the role of $\xi'$ is played more simply by $1$.

\subsection{Cohomology of the reduced lattice} \label{redlat}

In this subsection, we define the explicit map $\varpi \colon P \to Y$ and the map $\Upsilon \colon Y \to P$ obtained from the Galois action on $T$, and we recall our conjecture that they are mutual inverses. 
We then describe Galois cohomological aspects of the work of Fukaya and Kato on the conjecture.
In particular, we provide an interpretation of $\Upsilon$ as a connecting map in the cohomology of $T(1)$, and we identify a connecting map in the compactly supported cohomology of $P(1)$ with $\xi'$.

\begin{definition} 
	We set $Y = H^2_{\Iw}(\mc{O}_{\infty}[\mu_{Np}],\zp(2))_{\theta}$ and consider it as a $\La_{\theta}$-module
	for the action of inverse diamond operators.
\end{definition}

\begin{remark} \label{Yiwamod}
	Let $Y'$ denote the $\theta$-eigenspace of the 
	Tate twist of the minus part of the unramified Iwasawa module over $\Q(\mu_{Np^{\infty}})$.
	Then the canonical maps 
	$$
		Y' \to H^2_{\Iw}(\Z_{\infty}[\tfrac{1}{p},\mu_{Np}],\zp(2))_{\theta} \hookrightarrow Y
	$$
	are isomorphisms by our hypotheses on $\theta$. In particular, the characteristic ideal of $Y$ is generated by
	$\xi$ by the Iwasawa main conjecture.
\end{remark}

\begin{remark}
	It follows from Shapiro's lemma that
	$$
		H^i(\mc{O},\Lai_{\theta}(2)) \cong H^i_{\Iw}(\mc{O}_{\infty}[\mu_{Np}],\zp(2))_{\theta}
	$$	
	for $i \in \Z$. In particular, we may identify $Y$ with $H^2(\mc{O},\Lai_{\theta}(2))$.
\end{remark}

We recall the following \cite[9.1.4]{fk-pf}.

\begin{lemma}[Fukaya-Kato] \label{Qcoh}
	The cohomology groups $H^i(\mc{O},Q(1))$ are zero for $i \notin \{1,2\}$ and are isomorphic to $Y$ otherwise.
	More precisely, the connecting map in the long exact sequence
	attached to
	$$
		0 \to \Lai_{\theta}(2) \xrightarrow{\xi} \Lai_{\theta}(2) \to Q(1) \to 0
	$$
	induces an isomorphism $H^1(\mc{O},Q(1)) \xrightarrow{\sim} Y$, and the quotient map in said sequence induces
	an isomorphism $Y \xrightarrow{\sim} H^2(\mc{O},Q(1))$.
\end{lemma}

\begin{proof}
	The group $H^1(\mc{O},\Lai_{\theta}(2))$ vanishes since it is isomorphic to the Tate twist of the group of norm compatible
	systems of $p$-completions of $p$-units in the cyclotomic $\zp$-extension of $\Q(\mu_{Np})$, its $\theta$-eigenspace
	is zero since $\theta$ is even, not equal to $\omega^2$, and Hypothesis \ref{thetahyp}d holds.
	Since $G_{\Q,S}$ has $p$-cohomological
	dimension $2$ \cite[Proposition 10.11.3]{nsw}, we have an exact sequence
	$$
		0 \to H^1(\mc{O},Q(1)) 
		\to H^2(\mc{O},\Lai_{\theta}(2)) \xrightarrow{\xi} H^2(\mc{O},\Lai_{\theta}(2))
		\to H^2(\mc{O},Q(1)) \to 0
	$$
	in which the middle map is zero by Stickelberger theory (or the main conjecture and the fact that $Y$
	has no $p$-torsion).
\end{proof}

We also note the following simple lemma on the compactly supported cohomology of $P$, employing 
Poitou-Tate duality as in Remark \ref{galcohfacts}b.

\begin{lemma} \label{Pcoh}
	The compactly supported cohomology groups $H^i_c(\mc{O},P(1))$ are zero for $i \notin \{2,3\}$ and
	are isomorphic to $P$ otherwise. For $i=3$, the isomorphism is given by the invariant
	map, whereas for $i=2$, we have a canonical isomorphism $H^2_c(\mc{O},P(1)) \cong \Gamma \cozp P$
	of Poitou-Tate duality that we compose with the map induced by $-\chi \colon \Gamma \to \zp$.
	Moreover, the natural maps 
	$$
		H^i_c(\Z[\tfrac{1}{p}],P(1)) \to H^i_c(\mc{O},P(1)) 
	$$ 
	are isomorphisms.
\end{lemma}

\begin{proof}
	As $P$ has trivial Galois action, the invariant map of Remark \ref{galcohfacts}b provides a canonical isomorphism
	$H^3_c(\mc{O},P(1)) \cong P$.
	For $i = 1$, we similarly have
	$$
		H^1_c(\mc{O},P(1)) \cong H^2(\mc{O},P^{\vee})^{\vee} = 0
	$$
	in that $H^2(\mc{O},\Z/p\Z) = 0$ (as seen by a standard argument involving 
	inflation-restriction for $\Q(\mu_p)/\Q$ and Kummer theory).
	Since the above arguments work for any compact $\zp$-module $M$ with trivial $G_{\Q,S}$-action, 
	the functors $M \mapsto H^i_c(\mc{O},M(1))$ are exact for $i = 2,3$ and are trivial for all other $i$.
	The maximal pro-$p$, abelian, $S$-ramified extension of $\Q$ is $\Q_{\infty}$
	in that no prime dividing $N$ is $1$ modulo $p$, so we have
	$$
		H^2_c(\mc{O},P(1)) \cong H^2_c(\mc{O},\zp(1)) \cozp P \cong H^1(\mc{O},\qp/\zp)^{\vee} \cozp P
		\cong \Gamma \cozp P.
	$$
	and we apply the isomorphism $-\chi \colon \Gamma \to \zp$ to obtain the result.  
	A similar argument gives the analogous statements for $\Z[\frac{1}{p}]$ and through it the isomorphisms. 
\end{proof}

We can define a cocycle $b \colon G_{\Q,S} \to \Hom_{\mf{h}}(Q,P)$ using the exact sequence \eqref{globmodeis}
by 
$$
	b(\sigma)(q) = \sigma\left( \widetilde{\sigma^{-1}q}\right) - \tilde{q}
$$
for $q \in Q$, letting $\tilde{x}$ denote the image of $x$ under a fixed $\mf{h}_{\theta}$-module splitting $Q \to T$.
Then $b$ restricts to an everywhere unramified homomorphism on the absolute Galois group of 
$\Q(\mu_{Np^{\infty}})$ by Proposition \ref{exseqredunr}, which we can view as having domain $Y$ by Remark \ref{Yiwamod}.  
Through the isomorphism of Proposition \ref{Qiso}, we have moreover a canonical
isomorphism $\Hom_{\mf{h}}(Q,P) \cong P$ of $\La_{\theta}$-modules. The result is the desired
map $\Upsilon$ (see \cite[Section 4.4]{me-Lfn}, though note that we have not multiplied by any additional unit here).

\begin{definition}
	Let $\Upsilon \colon Y \to P$ denote the homomorphism of $\La_{\theta}$-modules induced by $b$ and 
	Proposition \ref{Qiso}.
\end{definition}

We also have a map in the other direction that takes a trace-compatible system of Manin symbols
to a corestriction compatible system of cup products of cyclotomic units.  

\begin{definition}
	Let $\varpi \colon \mc{S}_{\theta} \to Y$ denote the map constructed in 
	\cite[Proposition 5.7]{me-Lfn}, with reference to \cite[5.3.3]{fk-pf}, where the latter is shown to factor through $P$.  
\end{definition}	

We also use $\varpi$ to denote the induced map $\varpi \colon P \to Y$. 

\begin{remark} \label{varpi}
	We recall that $\varpi$ is the restriction of the inverse limit under trace and corestriction of maps 
	\begin{equation} \label{varpi_r}
		\varpi_r \colon H_1(X_1(Np^r),C_1^0(Np^r),\zp)^+ \to H^2_{\et}(\Z[\mu_{Np^r},\tfrac{1}{Np}],\zp(2))^+,\ 
		[u:v]_r \mapsto (1-\zeta_{Np^r}^u,1-\zeta_{Np^r}^v)_r,
	\end{equation}
	for $r \ge 1$, where $u, v \in \Z/Np^r\Z - \{0\}$ with $(u,v) = (1)$.  We briefly define the symbols that appear.
	
	On the right-hand side of \eqref{varpi_r}, the symbol $(\ ,\ )_r$ denotes the projection to the $+$-part of the 
	pairing on cyclotomic $Np$-units 
	induced by the cup product 
	$$
		H^1_{\et}(\Z[\mu_{Np^r},\tfrac{1}{Np}],\zp(1)) \times 
		H^1_{\et}(\Z[\mu_{Np^r},\tfrac{1}{Np}],\zp(1)) \xrightarrow{\cup} 
		H^2_{\et}(\Z[\mu_{Np^r},\tfrac{1}{Np}],\zp(2))
	$$
	and Kummer theory.
	We also set $\zeta_M = e^{2\pi i/ M}$ for $M \ge 1$, again viewing $\qbar$ inside $\C$.
	 
	On the left-hand side of \eqref{varpi_r}, we have
	$$
		[u:v]_r = \left( w_r\left(\begin{smallmatrix} a&b\\c&d \end{smallmatrix}\right)\{0 \to \infty\}_r \right)^+,
	$$
	where $w_r$ is the Atkin-Lehner involution of level $Np^r$ and the matrix
	$\left(\begin{smallmatrix} a&b\\c&d \end{smallmatrix}\right) \in \mr{SL}_2(\Z)$ has bottom row 
	$(u,v) \bmod Np^r$.  (Note that $w_r\{0 \to \infty\}_r = \{\infty \to 0\}_r$.)  We project the 
	resulting element to the plus part after the operations, denoting this with $(\ \ )^+$.  Since $u, v \neq 0$, the 
	symbol $[u:v]_r$ lies in $H_1(X_1(Np^r),C_1^0(Np^r),\zp)^+$, where 
	$C_1^0(Np^r)$ denotes the cusps not over $0 \in X_0(Np^r)$.
\end{remark}

We recall the conjecture of \cite{me-Lfn}.

\begin{conjecture} \label{sconj}
	The maps $\varpi \colon P \to Y$ and $\Upsilon \colon Y \to P$ are inverse maps.
\end{conjecture}

Actually, Conjecture \ref{sconj} was originally conjectured by the author up to a canonical unit.  There were indications
that this unit might be $1$ (if sign conventions were correct), but while the author advertised this suspicion rather widely and included it in preprint versions of the paper, he opted not to conjecture it in the final published version.  It was the work of Fukaya and Kato in \cite{fk-pf} that finally made it clear that the unit should indeed be $1$, not least because one would expect that the hypotheses under which they can prove it should hold without exception.  Nevertheless, one does not actually know how to prove that their hypotheses always hold.  Indeed, this paper is motivated by a desire to explore where the difficulty lies in removing them.

\begin{remark} \label{equivalent}
	Hida theory tells us that the $\La_{\theta}$-characteristic ideal of $P$ is divisible by $(\xi)$
	in that $\mcT_{\quo}$ is $\mf{h}_{\theta}$-faithful and $(\mf{h}/I)_{\theta}$ is annihilated by (a multiple of) $\xi$.
	Moreover, the main conjecture \cite{mw}
	tells us that the $\La_{\theta}$-characteristic ideal of $Y$ is equal to $(\xi)$.  
	As $Y$ is well known to be $p$-torsion free (i.e., by results of Iwasawa and Ferrero-Washington), 
	Conjecture \ref{sconj} is reduced to showing that $\Upsilon \circ \varpi = 1$ on $P$.
\end{remark}
	
Consider the complex
\begin{equation} \label{selmercplx}
	C_f(\mc{O},T(1)) = \Cone\left( C(\mc{O},T(1)) \to \bigoplus_{\ell \mid Np} C(\Q_{\ell},P(1)) \right)[-1],
\end{equation}
where ``$C$'' here is used to denote the standard inhomogeneous cochain complexes, and the map in the cone uses 
the local splitting $T \to P$.  We have an exact sequence
of complexes
$$
	0 \to C_c(\mc{O},P(1)) \to C_f(\mc{O},T(1)) \to C(\mc{O},Q(1)) \to 0,
$$
where $C_c$ is the complex defining compactly supported cohomology, which has connecting
homomorphisms
\begin{equation} \label{selmerconn}
	H^i(\mc{O},Q(1)) \to H^{i+1}_c(\mc{O},P(1))
\end{equation}
for $i \ge 0$.  For $i = 1$, let us denote this connecting homomorphism by $\Theta$. 
The connecting homomorphism for $i = 2$ can be identified with $\Upsilon$, which follows from \cite[9.4.3]{fk-pf}, noting 
Remark \ref{oppositesign} below.
One could simply take this as the definition of $\Upsilon$ for the purposes of this article. Nevertheless, we give a fairly detailed sketch of the proof using the results of \cite{me-selmer}, as it is by now an old result due independently to the author.

\begin{lemma} \label{connmapUps}
	Under the identifications of Lemmas \ref{Qcoh} and \ref{Pcoh}, the connecting homomorphism
	$$
		H^2(\mc{O},Q(1)) \to H^3_c(\mc{O},P(1))
	$$
	is $\Upsilon \colon Y \to P$.
\end{lemma}

\begin{proof}
	We consider a diagram
	$$
		\SelectTips{cm}{} \xymatrix{
		H^2_{\Iw}(\mc{O}_{\infty}[\mu_{Np}],\zp(2))_{\theta} \ar[r] \ar[d]^{\wr} & H^3_{c,\Iw}(\mc{O}_{\infty}[\mu_{Np}],P(1)) \ar[d]^{\wr} \\
		H^2(\mc{O},Q(1)) \ar[r] & H^3_c(\mc{O},P(1)),
		}
	$$
	where the connecting homomorphism that is the lower map is given by left cup product with $b \colon G_{\Q,S} \to \Hom_{\mf{h}}(Q,P)$
	by \cite[Proposition 2.3.3]{me-selmer}.
	The left vertical map employs the surjection $\La_{\theta}^{\iota}(1) \to Q$
	determined by Proposition \ref{Qiso}, and the right vertical map uses the quotient map $\tilde{\La}^{\iota} \to \zp$, which is
	to say it becomes corestriction via Shapiro's lemma.
	The diagram is then commutative taking the upper horizontal map to be given by left cup product with the 
	cocycle $G_{\Q(\mu_{Np^{\infty}}),S} \to P$ given by following the restriction of $b$ with evaluation at the canonical generator of $Q$.
	Recall that this cocycle is a homomorphism that by definition factors through $\Upsilon \colon Y \to P$. 
	That the upper horizontal map then 
	agrees with $\Upsilon$ via the identifications of the groups with $Y$ and $P$ is seen 
	by noting that it is Pontryagin dual via Poitou-Tate duality to the Pontryagin dual of $\Upsilon$, 
	via an argument mimicking the proof of \cite[Proposition 3.1.3]{me-selmer} (noting Proposition 2.4.3 therein, which in particular implies that 
	the signs agree).
\end{proof}

\begin{remark} \label{oppositesign}
	The connecting map $H^2(\mc{O},Q(1)) \to H^3_c(\mc{O},P(1))$ that we use is 
	the negative of the corresponding map in \cite{fk-pf}, since the identification of $Q$ with $(\mf{h}/I)^{\iota}_{\theta}(1)$
	of Proposition \ref{Qiso}, and hence of $Y$ with $H^2(\mc{O},Q(1))$, 
	is of opposite sign to that of \cite[6.4.3]{fk-pf}. This explains why the connecting map
	is identified with $\Upsilon$ in Lemma \ref{connmapUps}, whereas it is identified with $-\Upsilon$ in \cite[9.4.3]{fk-pf}.
\end{remark}	

\begin{definition} \label{bockstein}
	For a $\zp\ps{G_{\Q,S}}$-module (or $\zp\ps{G_{\Q_{\ell}}}$-module)
	$M$, let $\partial_M$ denote a connecting homomorphism in a long exact
	sequence in cohomology attached to the Tate twist of the exact sequence
	\begin{equation} \label{extnclass}
		0 \to M \xrightarrow{X} \Lai/(X^2) \cozp M \xrightarrow{\bmod X} M \to 0.
	\end{equation}
\end{definition}	
	
\begin{remark} \label{connmapsign} 
	The maps $\partial_M$ for any $\zp\ps{G_{\Q,S}}$-module $M$ agree with left cup product by the cocycle 
	$-\chi$ defining the extension class \eqref{extnclass} (cf. \cite[Proposition 2.3.3]{me-selmer}). 
	As pointed out in \cite[9.3.4]{fk-pf}, the sign in $-\chi$ occurs as $G_{\Q}$ acts on $\Lai$ through left multiplication by the inverse 
	of its quotient map to $\Gamma \subset \La$.
\end{remark}

\begin{lemma} \label{loccompconn}
	Let $M$ be a compact or discrete $\zp\ps{G_{\Q,S}}$-module. Then the diagram
	$$
		\SelectTips{cm}{} \xymatrix{
		\bigoplus_{\ell \mid Np} H^1(\Q_{\ell},M(1)) \ar[r]^{\partial_M} \ar[d] & \bigoplus_{\ell \mid Np} H^2(\Q_{\ell},M(1)) \ar[d] \\
		H^2_c(\mc{O},M(1)) \ar[r]^{\partial_M} & H^3_c(\mc{O},M(1))
		}
	$$
	anticommutes.
\end{lemma}

\begin{proof}
	Recall that 
	$$
		C_c(\mc{O},M(1)) = \ker\left( C(\mc{O},M(1)) \xrightarrow{\mr{res}} \bigoplus_{\ell \mid Np} C(\Q_{\ell},M(1)) \right)[-1]
	$$
	which is to say that 
	$$
		C_c^i(\mc{O},M(1)) = \left( \bigoplus_{\ell \mid Np} C^{i-1}(\Q_{\ell},M(1)) \right) \oplus C^i(\mc{O},M(1)),
	$$
	with the differential taking $(x, y)$ to $(-d^{i-1}(x)-\mr{res}(y),d^i(y))$. The composition
	$$
		H^1(\Q_{\ell},M(1)) \xrightarrow{\partial_M} H^2(\Q_{\ell},M(1)) \to H^3_c(\mc{O},M(1))
	$$
	takes a class $\phi$ to the image of the compactly supported cocycle $(\partial_M(\phi),0)$, whereas the
	composition
	$$
		H^1(\Q_{\ell},M(1)) \to H^2_c(\mc{O},M(1)) \xrightarrow{\partial_M} H^3_c(\mc{O},M(1))
	$$
	takes $\phi$ to $\partial_M(\phi,0) = (-\partial_M(\phi),0)$ in that the differential 
	used to compute the connecting homomorphism restricts to the negative of the local differential.
\end{proof}

We also have the following lemma.

\begin{lemma} \label{connectP}
	The connecting homomorphism $\partial_P \colon H^2_c(\mc{O},P(1)) \to H^3_c(\mc{O},P(1))$
	is identified with the identity map on $P$ via the isomorphisms of Lemma \ref{Pcoh}.
\end{lemma}

\begin{proof}
	As noted in Remark \ref{connmapsign}, the connecting map $\partial_P$ is given by left cup product
	with $-\chi \in H^1(\mc{O},\zp)$. 	
	By the commutativity
	(with elements of the even degree cohomology group $H^2_c(\mc{O},P(1))$) and associativity
	of cup products, $\partial_P$ is Poitou-Tate dual to the map $H^0(\mc{O},P^{\vee}) \to H^1(\mc{O},P^{\vee})$ 
	also given by left cup product with $-\chi$. In turn, this dual is identified with the dual of the isomorphism
	$\Gamma \otimes_{\zp} P \to P$ induced by $-\chi \colon \Gamma \to \zp$.
	Finally, this is exactly our prior identification of $H^2_c(\mc{O},P(1))$ with $P \cong H^3_c(\mc{O},P(1))$.
\end{proof}

The following exercise in Galois cohomology encapsulates a key aspect of the work of Fukaya-Kato 
\cite[Sections 9.3-9.5]{fk-pf}.  We omit the proof, as the reader will find its key ideas contained in the
refined study that follows (cf. Proposition \ref{snakesquare} for the commutativity of the left-hand square
and Lemma \ref{conndiag} for the middle square on the right).

\begin{theorem}[Fukaya-Kato] \label{FKgalcoh}
	Let $\xi' = \xi'_{\theta} \in \La_{\theta}$ be such that 
	$$
		\xi'_{\theta}(t^s-1) = L_p'(\omega^2\theta^{-1},s-1)
	$$
	for all $s \in \zp$.  
	Then the diagram 
	$$
		\SelectTips{cm}{} \xymatrix{
		& Y \ar[r]^{\xi'} & Y \ar[d]^{\wr} \\
		H^1(\mc{O},T(1)) \ar[r] \ar[d] & H^1(\mc{O},Q(1)) \ar[d]^{-\Theta} \ar[r] \ar[u]^{\wr}
		\ar[r]^{\partial_Q} & H^2(\mc{O},Q(1)) \ar[d]^{\Upsilon}  \\
		\bigoplus_{\ell \mid Np} H^1(\Q_{\ell},P(1)) \ar[r] & H^2_c(\mc{O},P(1)) 
		\ar[r]^{\partial_P} & H^3_c(\mc{O},P(1)) \ar[d]^{\wr} \\
		& P \ar[u]^{\wr} \ar[r]^{\id}_{\sim} & P
		}
	$$
	commutes, where the indicated vertical isomorphisms are those of  Lemmas \ref{Qcoh} and \ref{Pcoh},
	and the leftmost vertical map uses the unique local splittings of Proposition \ref{exseqredunr}.
\end{theorem}

In particular, note that $\Theta$ is identified with $-\xi'\Upsilon$ as a map $Y \to P$.
In this section, we aim to remove the derivative by modifying the diagram.

\subsection{Intermediate quotients} \label{intermedquo}

Let $\cozp$ denote the completed tensor product over $\zp$.  We use it consistently even in cases for which
the usual tensor product gives the same module, in part to indicate that our modules carry a compact topology.

In this brief subsection, we introduce and study the diagonalized zeta function $\otil{\xi}$. We then define the intermediate quotient $M^{\dagger} = (\Lai \cozp M)/\otil{\xi}$ of a compact $\La_{\theta}\ps{G_{\Q,S}}$-module $M$ that is annihilated by $\xi$.
This quotient of $\Lai \cozp M$ has $M$ as a quotient, so it
can be said to be intermediate between $\Lai \cozp M$ and $M$. We note that though the $G_{\Q,S}$-cohomology of $M^{\dagger}$ then sits between the Iwasawa cohomology and the $G_{\Q,S}$-cohomology of $M$,
it is not identified with the cohomology of $M$ in an intermediate extension.

In the following, when we write $\La$ (as opposed to $\Lai$), we shall consider this $\La$ as carrying a trivial $G_{\Q,S}$-action.
Let $\otil{\xi} \in \La \cozp \La_{\theta} = R\ps{\Gamma^2}$ be the diagonal image of 
$\xi \in \La_{\theta} = R\ps{\Gamma}$. More concretely,
we can make the following definition.

\begin{definition}
	Write $\xi  = \sum_{i=0}^{\infty} a_i X^i = \sum_{i=0}^{\infty} a_i(\gamma-1)^i$ for some $a_i \in R$. We define
	the element $\otil{\xi}$ of $\La \cozp \La_{\theta}$ by 
	$$
		\otil{\xi} = \sum_{i=0}^{\infty} a_i (\gamma \otimes \gamma-1)^i.
	$$
\end{definition}

\begin{definition}
	For $n \ge 0$, set
	$$
		\xi^{(n)} = (X+1)^n\frac{1}{n!}\frac{d^n\xi}{dX^n}
		= (X+1)^n \sum_{i=n}^{\infty} a_i \binom{i}{n} X^{i-n} \in \La_{\theta}.
	$$
\end{definition}

\begin{remark}
	Note that $\xi^{(1)}$ is $\xi'$ of Proposition \ref{FKgalcoh}.  
\end{remark}

We have $\La \cozp \La_{\theta} \cong R\ps{X \otimes 1,1 \otimes X}$. 
We frequently refer to $X \otimes 1 \in \La \cozp \La_{\theta}$ more simply by $X$, and we set $V = 1 \otimes X$ as needed
to make the identification 
$$
	\La \cozp \La_{\theta} = R\ps{X,V}.
$$

While not used later, the following description of $\otil{\xi}$ gives one some insight into its form; we thank the referee for
suggesting this simple proof.

\begin{proposition} \label{tildef}
	We have $\otil{\xi} = \sum_{n=0}^{\infty} X^n \otimes \xi^{(n)}$.
\end{proposition}

\begin{proof}
	We have
	$\tilde{\xi} = \sum_{i=0}^{\infty} a_i (X + V + XV)^i$.
	By the binomial theorem, we have
	$$
		\sum_{i=0}^{\infty} a_i (X + V + XV)^i = \sum_{i=0}^{\infty} a_i \sum_{n=0}^i \binom{i}{n} X^n (V+1)^n V^{i-n}
		= \sum_{n=0}^{\infty} X^n \cdot (V+1)^n \sum_{i=n}^{\infty} a_i \binom{i}{n} V^{i-n},
	$$
	the latter term being $\sum_{n=0}^{\infty} X^n \otimes \xi^{(n)}$.
\end{proof}

The following applies to $\otil{\xi}$ with $\mf{R} = \La_{\theta}$ by the Ferrero-Washington theorem \cite{fw}, as
the reduction of $\otil{\xi}$ modulo $V$ is $\xi$ under the identification $\La \cozp \La_{\theta}/(V) \cong \La_{\theta}$.

\begin{lemma} \label{diagmult}
	Let $\mf{R}$ be a complete local $\zp$-algebra with residue field $\mf{k}$, and let $M$ be a compact $\mf{R}$-module. 
	Let $\mu \in \La \cozp \mf{R}$ with nontrivial image in $\La \cozp \mf{k}$.
	Then $\mu$ is an injective $\La \cozp \mf{R}$-module endomorphism of $\La \cozp M$.
\end{lemma}

\begin{proof}
	We must verify injectivity. Replacing $M$ by its associated graded for the powers of the maximal ideal of $\mf{R}$,
	it suffices to consider $M = \mf{k}$, so $\La \cozp M = \mf{k}\ps{X}$.
	Since $\mu$ acts as a nonzero element of $\mf{k}\ps{X}$ on the integral domain $\mf{k}\ps{X}$, it is clearly injective.
\end{proof}

If $M$ of Lemma \ref{diagmult} has the further structure of a continuous $\mf{R}$-linear $G_{\Q,S}$-action, then multiplication by $\mu \in \La \cozp \mf{R}$ is a continuous $(\La \cozp \mf{R})\ps{G_{\Q,S}}$-module endomorphism of $\Lai \cozp M$. In the case that $\mf{R} = \La_{\theta}$ and $M$ that is annihilated by $\xi$, our interest is in the intermediate quotient
$$
	M^{\dagger} = (\Lai \cozp M)/\otil{\xi}(\Lai \cozp M),
$$ 
where $\otil{\xi} \colon \Lai \cozp M \to \Lai \cozp M$ is injective by the lemma.
We give this dagger notation and the notion of an intermediate quotient 
a more general definition in Defintion \ref{interquot}. 

\begin{remark} \label{interrem}
	Let us remark on our choice of the word ``intermediate'': since $M$ is annihilated by $\xi$, the intermediate quotient $M^{\dagger}$
	is also the quotient of $\Lai \cozp M$ by $\otil{\xi} - 1 \otimes \xi$, which is divisible by the 
	variable $X$ of $\La$, so indeed $M^{\dagger}$ has $M$ as a quotient. We will often write 
	$$
		\otil{\xi}_1 = X^{-1}(\otil{\xi} - 1 \otimes \xi).
	$$
\end{remark}

\subsection{Refined cohomological study} \label{refinedcoh}

In this subsection, we study the cohomology of intermediate quotients. Specifically, we study the $S$-ramified cohomology of $Q^{\dagger}(1)$, the compactly supported cohomology of $P^{\dagger}(1)$, and their relationship via the cohomology of $T^{\dagger}(1)$. 
We shall replace the commutative diagram of Theorem \ref{FKgalcoh} with a similar commutative diagram
$$
	\SelectTips{cm}{} \xymatrix{
	& Y \ar[r]^{\id} & Y \ar[d]^{\wr} \\
	H^1(\mc{O},T^{\dagger}(1)) \ar[r] \ar[d] & H^1(\mc{O},Q^{\dagger}(1)) \ar[d]^{-\Theta^{\dagger}} \ar[r] \ar@{->>}[u]
	\ar[r]^{\partial_Q^{\dagger}} & H^2(\mc{O},Q(1)) \ar[d]^{\Upsilon}  \\
	\bigoplus_{\ell \mid Np} H^1(\Q_{\ell},P^{\dagger}(1)) \ar[r] & H^2_c(\mc{O},P^{\dagger}(1)) 
	\ar[r]^{\partial_P^{\dagger}} & H^3_c(\mc{O},P(1)) \ar[d]^{\wr} \\
	& P \ar[u]^{\wr} \ar[r]^{\id}_{\sim} & P
	}
$$
which has the distinguishing feature that the connecting map $\partial_Q^{\dagger}$ that takes the place of $\partial_Q$ induces the identity on $Y$,
as opposed to $\xi'$. 

\begin{definition}
	We denote by $\w$ the continuous $\zp$-algebra automorphism of $\La \cozp \La$ 
	and $\zp\ps{G_{\Q,S}}$-module isomorphism 
	$$
		\w \colon \Lai \cozp \Lai \xrightarrow{\sim} \Lai \cozp \La
	$$
	satisfying
	$$
		\w(\sigma \otimes \tau) = \sigma \otimes \sigma^{-1}\tau
	$$
	for $\sigma, \tau \in \Gamma$.
\end{definition}

We will apply the map induced by $\w$ to $\Lai \cozp \Lai_{\theta} = (\Lai \cozp \Lai) \cozp R^{\iota}$,
which we can think of as landing in $\Lai_{\theta} \cotimes{R} \La_{\theta} = (\Lai \cozp \La) \cozp R^{\iota}$.

\begin{proposition} \label{cohquot}
	For $i \in \Z$, we have isomorphisms 
	$$
		H^i(\mc{O},Q^{\dagger}(1)) \cong H^i(\mc{O},Q(1)) \cotimes{R} (\La_{\theta}/\xi)
	$$
	of $\La \cozp \La_{\theta}$-modules, where $f \in \La \cozp \La_{\theta}$ 
	acts as usual on the left and as $\w(f)$ on the right.
\end{proposition}

\begin{proof}
	Consider the diagram
	\begin{equation} \label{coulddosnake}
		\SelectTips{cm}{} \xymatrix{
		& 0 \ar[d] & 0 \ar[d] & 0 \ar[d] & \\
		0 \ar[r] & \Lai_{\theta} \cotimes{R} \La_{\theta}(1) \ar[r]^{\xi \otimes 1} \ar[d]^{1 \otimes \xi} 
		&  \Lai_{\theta}  \cotimes{R} \La_{\theta}(1) \ar[d]^{1 \otimes \xi} \ar[r]
		& Q \cotimes{R} \La_{\theta} \ \ar[r] \ar[d]^{1 \otimes \xi} & 0 \\
		0 \ar[r] & \Lai_{\theta} \cotimes{R} \La_{\theta}(1) \ar[d] \ar[r]^{\xi \otimes 1} &  
		\Lai_{\theta} \cotimes{R} \La_{\theta}(1) \ar[r] \ar[d]&
		Q \cotimes{R} \La_{\theta} \ar[r] \ar[d] & 0 \\
		0 \ar[r] &  \Lai \cozp Q \ar[d] \ar[r]^{\otil{\xi}} & \Lai \cozp Q \ar[r] \ar[d] & 
		Q^{\dagger} \ar[r] \ar[d] & 0\\
		&0&0&0&
		}
	\end{equation}
	with exact rows and columns, where the maps
	$\Lai_{\theta} \cotimes{R} \La_{\theta}(1) \to \Lai \cozp Q$ 
	are given by the composition of $\w^{-1}$ with the quotient by $1 \otimes \xi$.
	Since $1 \otimes \xi = \w(1 \otimes \xi)$
	and $\xi \otimes 1 = \w(\otil{\xi})$ in $\La_{\theta} \cotimes{R} \La_{\theta}$, the diagram commutes.
	If we let $f \in \La \cozp \La_{\theta}$ act as $\w(f)$ on the terms in the first
	two rows and as usual in the third, then the diagram is of
	$(\La \cozp \La_{\theta})\ps{G_{\Q,S}}$-modules.

	In particular, we have a canonical isomorphism
	$Q^{\dagger} \cong Q \cotimes{R} (\La_{\theta}/\xi)$
	of $(\La \cozp \La_{\theta})\ps{G_{\Q,S}}$-modules, 
	again understanding that $f \in \La \cozp \La_{\theta}$
	acts on the right by $\w(f)$ (so, e.g., $1 \otimes \xi$ acts as zero).
	As $\La_{\theta}$ is free profinite over $R$ with trivial $G_{\Q,S}$-action, we then have
	\begin{eqnarray*}
		H^i(\mc{O},Q \cotimes{R} \La_{\theta}(1)) \cong 
		H^i(\mc{O},Q(1)) \cotimes{R} \La_{\theta}
	\end{eqnarray*}
	for all $i$.  So, we have an exact sequence
	\begin{multline*}
		\cdots \to H^i(\mc{O},Q(1)) \cotimes{R} \La_{\theta} \xrightarrow{1 \otimes \xi} 
		H^i(\mc{O},Q(1)) \cotimes{R} \La_{\theta} \to H^i(\mc{O}, Q^{\dagger}(1)) \\ \to 
		H^{i+1}(\mc{O},Q(1)) \cotimes{R} \La_{\theta}  \xrightarrow{1 \otimes \xi} 
		H^{i+1}(\mc{O},Q(1)) \cotimes{R} \La_{\theta}  \to \cdots
	\end{multline*}
	of $\La \cozp \La_{\theta}$-modules. Since $1 \otimes \xi$ has trivial kernel on 
	$H^{i+1}(\mc{O},Q(1)) \cotimes{R} \La_{\theta}$ by the Ferrero-Washington theorem,
	the exact sequence provides the result.
\end{proof}

For a compact $\La_{\theta}\ps{G_{\Q,S}}$-module $M$ annihilated by $\xi$, we have exact sequences
\begin{gather}
	0 \to M^{\dagger} \xrightarrow{X} (\Lai \cozp M)/X\otil{\xi} \to M \to 0, \label{horiz1} \\
	0 \to M \xrightarrow{\otil{\xi}} (\Lai \cozp M)/X\otil{\xi} \to M^{\dagger} \to 0 \label{horiz2}
\end{gather}
of $(\La \cozp \La_{\theta})\ps{G_{\Q,S}}$-modules, the left exactness of which is a consequence of Lemma \ref{diagmult}.

\begin{notation}
	We use $\partial_M^{\dagger}$ to denote connecting maps in the cohomology of the Tate twist of \eqref{horiz2}.
\end{notation}

The following refines \cite[9.3.3]{fk-pf} in our case of interest.

\begin{theorem} \label{partial}
	Let $\varepsilon_Q$ denote the composition
	$$
		H^1(\mc{O},Q(1)) \xrightarrow{\sim} H^2(\mc{O},\Lai_{\theta}(2)) \xrightarrow{\sim} H^2(\mc{O},Q(1)),
	$$
	of the isomorphisms of Lemma \ref{Qcoh}, which is identified with the identity on $Y$. 
	Then $\partial_Q^{\dagger}$ equals the composition
	$$
		H^1(\mc{O},Q^{\dagger}(1)) \twoheadrightarrow H^1(\mc{O},Q(1)) \xrightarrow{\varepsilon_Q} H^2(\mc{O},Q(1)),
	$$
	where the first map is a surjection induced by the quotient map $Q^{\dagger} \twoheadrightarrow Q$.
\end{theorem}

\begin{proof}
	The commutative diagram
	$$
		\SelectTips{cm}{} \xymatrix{
		0 \ar[r] & \Lai_{\theta}(1) \ar[r]^-{\xi} & \Lai_{\theta}(1) \ar[r] & Q \ar[r]  & 0\\
		0 \ar[r] & \Lai_{\theta}(1) \ar[r]^-{\otil{\xi}} \ar@{->>}[d] \ar@{=}[u] & 
		(\La \cozp \Lai_{\theta}(1))/X\otil{\xi} 	\ar[r] \ar@{->>}[d] \ar@{->>}[u]_{\bmod X} &
		Q \cotimes{R} \La_{\theta} \ar[r] \ar@{->>}[d] \ar@{->>}[u]_{\bmod \w(X)} & 0 \\
		0 \ar[r] & Q \ar[r]^-{\otil{\xi}} & (\Lai \cozp Q)/X\otil{\xi} \ar[r] & 
		Q^{\dagger} \ar[r] & 0
		}
	$$
	with exact rows (the right horizontal arrow of the middle row being induced by $w$ and the quotient)
	gives rise to a commutative diagram of connecting maps
	\begin{equation} \label{techdiag}
		\SelectTips{cm}{} \xymatrix{
		 H^1(\mc{O},Q(1)) \ar[r]^{\sim} & H^2(\mc{O},\Lai_{\theta}(2)) \\
		H^1(\mc{O},Q(1)) \cotimes{R} \La_{\theta} \ar@{->>}[r] \ar@{->>}[u]_{\bmod \w(X)} 
		\ar@{->>}[d] & H^2(\mc{O},\Lai_{\theta}(2)) 
		\ar@{=}[u] \ar[d]^{\wr} \\
		 H^1(\mc{O},Q^{\dagger}(1)) \ar[r] & H^2(\mc{O},Q(1)),
		}
	\end{equation}
	where the composition $H^1(\mc{O},Q(1)) \to H^2(\mc{O},Q(1))$ is $\varepsilon_Q$.
	
	By Proposition \ref{cohquot}, the map $H^1(\mc{O},Q^{\dagger}(1)) \to H^1(\mc{O},Q(1))$ is identified with
	the reduction modulo $\w(X) = \gamma \otimes \gamma^{-1} - 1$ map
	\begin{equation} \label{H1quo}
		H^1(\mc{O},Q(1)) \cotimes{R} (\La_{\theta}/\xi)  \twoheadrightarrow H^1(\mc{O},Q(1)).
	\end{equation}
	Combining this with \eqref{techdiag} gives the result.
\end{proof}

\begin{remark} \label{corQ}
	Given \eqref{H1quo} and the identifications of Lemma \ref{Qcoh},
	we may rephrase Theorem \ref{partial} as the assertion that $\partial_Q^{\dagger} \colon Y \cotimes{R} (\La_{\theta}/\xi) \to Y$
	is given by multiplication: i.e., satisfies $\partial_Q^{\dagger}(y \otimes f) = f \cdot y$ for $y \in Y$ and $f \in \La_{\theta}/\xi$. 
\end{remark}

\begin{remark} \label{diffquot}
	Let us mention another approach, closer to the original
	work of Fukaya and Kato in \cite[9.3]{fk-pf}, to elucidate why this removes the $\xi'$ from
	the left-hand side of \eqref{keyident}.
	For this, we recall (cf. \cite[9.3.5]{fk-pf}) that $\partial_Q^{\dagger}$ is the negative of a
	snake lemma map $H^1(\mc{O},Q^{\dagger}(1)) \to H^2_{\Iw}(\mc{O}_{\infty},Q(1))$
	for the cohomology of the Tate twist of the diagram \eqref{coulddosnake} followed by
	corestriction $H^2_{\Iw}(\mc{O}_{\infty},Q(1)) \to H^2(\mc{O},Q(1))$. 
	The map $H^2_{\Iw}(\mc{O}_{\infty},Q(1)) \to H^2(\mc{O},Q^{\dagger}(1))$ is an isomorphism, so the former group, like 
	$H^1(\mc{O},Q^{\dagger}(1))$, is canonically identified with $Y \cotimes{R} (\La_{\theta}/\xi)$ via Lemma \ref{cohquot}.
	Under these identifications, we claim that the negative of the snake lemma map
	is the identity on $Y \cotimes{R} (\La_{\theta}/\xi)$.
	
	To see this, we start with an element of $H^1(\mc{O},Q^{\dagger}(1))$, which is necessarily
	annihilated by $\otil{\xi}$.
	We must apply $1 \otimes \xi$ to a lift to $H^2(\mc{O},\Lai_{\theta}(2)) \cotimes{R} \La_{\theta}$ of its image under the 
	connecting map induced by the rightmost column
	of \eqref{coulddosnake}. This amounts to applying $1 \otimes \xi - \otil{\xi}$ to said lift, taking into account that $f \in \La \cozp \La$
	acts on the latter group by $\w(f)$. We must then ``divide'' this quantity by $\otil{\xi} = \w^{-1}(\xi \otimes 1)$,
	which in the quotient $H^2_{\Iw}(\mc{O}_{\infty},Q(1)) \cong Y \cotimes{R} (\La_{\theta}/\xi)$ 
	by $1 \otimes \xi$ can be taken as division
	by $\otil{\xi} - 1 \otimes \xi$. Then we just note that
	$$
		-\frac{1 \otimes \xi - \otil{\xi}}{\otil{\xi} - 1 \otimes \xi} = 1.
	$$
	
	If, as in \cite[9.3.1]{fk-pf}, we had used $X$ in place of $\otil{\xi}$ and the Bockstein exact sequence
	arising from Definition \ref{bockstein} in place of \eqref{horiz2}, then the map $\partial_Q$ would similarly
	have been given by the negative of dividing $1 \otimes \xi - \xi \otimes 1$ by $\w(X)$
	and then taking its image under the multiplication map $\La_{\theta} \cotimes{R} \La_{\theta} \to \La_{\theta}$. 
	To compute this, we can first apply $\w^{-1}$ to $\w(X)^{-1}(\xi \otimes 1 - 1 \otimes \xi)$ to obtain $\otil{\xi}_1$,
	and then the multiplication map becomes evaluation at $X = 0$, without changing the result.
	This yields the quantity $\xi' = \otil{\xi}_1(0)$ on the left of \eqref{keyident}.
\end{remark}

\begin{proposition} \label{Psquare}
	We have a commutative square of isomorphisms
	$$
		\SelectTips{cm}{} \xymatrix{
			H^2_c(\mc{O},P(1)) \ar[r]^-{\sim} \ar[d]^{\wr} & 
			H^3_c(\mc{O},P^{\dagger}(1)) \ar[d]^{\wr} \\
			H^2_c(\mc{O},P^{\dagger}(1)) \ar[r]^-{\sim}_{\partial_P^{\dagger}} & H^3_c(\mc{O},P(1)),  \\
		}	
	$$
	between $\La$-modules canonically isomorphic to $P$, in which every vertical and horizontal map
	is identified with the identity map on $P$, and where the horizontal maps are the connecting homomorphisms 
	from \eqref{horiz2} and \eqref{horiz1} and the vertical maps are induced by multiplication by $\otil{\xi}_1$
	and the canonical quotient. The same holds with $\mc{O}$ replaced with $\Z[\tfrac{1}{p}]$.
\end{proposition}

\begin{proof}
	Commutativity follows from the morphism of exact sequences
	$$
		\SelectTips{cm}{} \xymatrix{
		0 \ar[r] & P \ar[r]^-{\otil{\xi}} \ar[d]^{\tilde{\xi}_1} & (\Lai \cozp P)/X\otil{\xi} \ar[r] \ar@{=}[d] & P^{\dagger} \ar[r] 
		\ar[d] & 0 \\
		0 \ar[r] & P^{\dagger} \ar[r]^-{X} & (\Lai \cozp P)/X\otil{\xi} \ar[r] & P \ar[r] & 0.
		}
	$$
	Via Poitou-Tate duality as in Remark \ref{galcohfacts}b (which in particular tells us that $M \mapsto H^3_c(\mc{O},M)$ 
	is right-exact), we have
	$$
		H^3_c(\mc{O},P^{\dagger}(1)) \cong H^3_{c,\Iw}(\mc{O}_{\infty},P(1))/\otil{\xi}
		\cong P/\otil{\xi}P \cong P,
	$$
	and an exact sequence
	$$
		H^2_{c,\Iw}(\mc{O}_{\infty},P(1)) \xrightarrow{\otil{\xi}} 
		H^2_{c,\Iw}(\mc{O}_{\infty},P(1)) \to H^2_c(\mc{O}, P^{\dagger}(1)) \to H^3_{c,\Iw}(\mc{O}_{\infty},P(1))  \xrightarrow{\otil{\xi}} 
		H^3_{c,\Iw}(\mc{O}_{\infty},P(1)).
	$$
	Note that $H^2_{c,\Iw}(\mc{O}_{\infty},P(1)) \cong H^1(G_{\Q_{\infty},S},P^{\vee})^{\vee}$ 
	is isomorphic to the tensor product with $P$ of the 
	Galois group of the maximal abelian pro-$p$, $S$-ramified extension of $\Q_{\infty}$, which is 
	trivial (since no prime dividing $N$ is $1$ modulo $p$),  
	so the first two terms are zero.  The last map is also zero since multiplication by $\otil{\xi}$ is
	trivial on $P$.  Thus, we have
	$$
		H^2_c(\mc{O},P^{\dagger}(1)) \xrightarrow{\sim} H^3_{c,\Iw}(\mc{O}_{\infty},P(1)) \xrightarrow{\sim} P,
	$$
	We choose the identification of $H^2_c(\mc{O},P^{\dagger}(1))$ with $P$ which makes this the identity map,
	and the right-hand vertical map is identified with the identity map on $P$ via invariant maps.   
	As for the upper map, note that it factors as 
	$$
		H^2_c(\mc{O},P(1)) \to H^3_{c,\Iw}(\mc{O}_{\infty},P(1)) \to H^3_c(\mc{O},P^{\dagger}(1)),
	$$
	where the first map is the connecting homomorphism,
	which is seen to be the
	identity map by using Poitou-Tate duality as in Lemma \ref{connectP}, and the second map is again clearly
	identified with the identity map on $P$.  The same argument works with $\mc{O}$ replaced
	by $\Z[\tfrac{1}{p}]$. 
\end{proof}

\begin{definition}
	We use $\varepsilon_P \colon H^2_c(\mc{O},P(1)) \to H^3_c(\mc{O},P(1))$ to denote the composition of isomorphisms
	in Proposition \ref{Psquare}.
\end{definition}

As a consequence of Lemma \ref{diagmult}, the sequence
$$
	0 \to P^{\dagger} \to T^{\dagger} \to Q^{\dagger} \to 0
$$
induced by \eqref{globmodeis} is exact. We define $\Theta^{\dagger}$ as the connecting map $H^1(\mc{O},Q^{\dagger}(1))
\to H^2_c(\mc{O},P^{\dagger}(1))$ analogous to $\Theta$ (see \eqref{selmerconn}).

\begin{lemma} \label{conndiag}
    The diagram of connecting homomorphisms
    $$
    	\SelectTips{cm}{} \xymatrix{
    		H^1(\mc{O},Q^{\dagger}(1)) \ar[r]^{\partial_Q^{\dagger}} \ar[d]^{\Theta^{\dagger}} & 
    		H^2(\mc{O},Q(1)) \ar[d]^{\Upsilon} \\
    		H^2_c(\mc{O},P^{\dagger}(1)) \ar[r]^{\partial_P^{\dagger}} & H^3_c(\mc{O},P(1))  \\
    	}	
    $$
    is anticommutative.
\end{lemma}

\begin{proof}
	Let us set $M^{\ddagger} = (\Lai \cozp M)/X\otil{\xi}$ for any compact $\La_{\theta}\ps{G_{\Q,S}}$-module $M$.
	We have a commutative diagram of exact sequences of complexes
	$$
		\SelectTips{cm}{} \xymatrix{
			& 0 \ar[d] & 0 \ar[d] & 0 \ar[d] \\
			0 \ar[r] & C_c(\mc{O},P(1)) \ar[r] \ar[d]  & C_c(\mc{O},P^{\ddagger}(1)) \ar[r] \ar[d]  
			& C_c(\mc{O},P^{\dagger}(1)) \ar[r] \ar[d] & 0\\
			0 \ar[r]& C_f(\mc{O},T(1)) \ar[r] \ar[d] 
			 & C_f(\mc{O},T^{\ddagger}(1)) \ar[r] \ar[d] & C_f(\mc{O},T^{\dagger}(1)) \ar[r] \ar[d] & 0\\
			0 \ar[r] & C(\mc{O},Q(1)) \ar[r] \ar[d] & C(\mc{O},Q^{\ddagger}(1)) \ar[r] \ar[d] 
			& C(\mc{O},Q^{\dagger}(1)) \ar[r] \ar[d] & 0  \\
			& 0 & 0 & 0,
		}
	$$
	where the Selmer complexes of the middle row are defined as in \eqref{selmercplx}. The lemma is then
	just the anticommutativity of connecting homomorphisms for a commutative square of short
	exact sequences of complexes.
\end{proof}

It follows from Theorem \ref{partial}, Proposition \ref{Psquare}, and Lemma \ref{conndiag}
that $\Theta^{\dagger}$ factors as
$$
	H^1(\mc{O},Q^{\dagger}(1)) \to H^1(\mc{O},Q(1)) \xrightarrow{\Phi} H^2_c(\mc{O},P(1))
	\xrightarrow{\sim} H^2_c(\mc{O},P^{\dagger}(1))
$$
for $\Phi = -\varepsilon_P^{-1} \circ \Upsilon \circ \varepsilon_Q$, which we may also view as a map 
$\Phi \colon Y \to P$.

\begin{lemma} \label{connzero}
	The connecting homomorphism $H^1(\mc{O},Q^{\dagger}(1)) \to H^2(\mc{O},P^{\dagger}(1))$
	is zero.
\end{lemma}

\begin{proof}
	Consider the commutative diagram
	\begin{equation} \label{zerodiag}
		\SelectTips{cm}{} \xymatrix{
			H^1(\mc{O},Q^{\dagger}(1)) \ar[r] \ar[d] & H^2(\mc{O},P^{\dagger}(1)) \ar[d] \\
			\bigoplus_{\ell \mid Np} H^1(\Q_{\ell},Q^{\dagger}(1)) \ar[r] \ar[d] & \bigoplus_{\ell \mid Np} 
			H^2(\Q_{\ell},P^{\dagger}(1)) \ar[d] \\ H^2_c(\mc{O},Q^{\dagger}(1)) \ar[r] 
			& H^3_c(\mc{O},P^{\dagger}(1))
		}
	\end{equation}
	with exact columns. Since $G_{\Q,S}$ and $G_{\Q_{\ell}}$ for $\ell \mid Np$ all have $p$-cohomological dimension $2$
	\cite[10.11.3, 7.1.8]{nsw}, and similarly third compactly supported cohomology is right-exact by 
	Remark \ref{galcohfacts}b,
	the right-hand column of \eqref{zerodiag} is isomorphic to the quotient by $\tilde{\xi}$ of the middle
	terms of the short exact sequence
	$$	
		0 \to H^2_{\Iw}(\mc{O}_{\infty},P(1)) \to \bigoplus_{\ell \mid Np} H^2(\Q_{\ell},\Lai \cozp P(1))
		\to H^3_{c,\Iw}(\mc{O}_{\infty},P(1)) \to 0,
	$$
	where the first map is injective since $H^2_{c,\Iw}(\mc{O}_{\infty},P(1)) = 0$ (as noted in the proof of Proposition \ref{Psquare}).
	In the second sum, for $\ell \mid Np$, we have
	$H^2(\Q_{\ell},\Lai \cozp P(1)) \cong \bigoplus_{v \mid \ell} P$, with the sum over primes of $\Q_{\infty}$, by 
	Remark \ref{galcohfacts}a and Tate duality. The third term is isomorphic to 
	$P$ via the invariant map. As these groups are killed by $\tilde{\xi}$, the sequence remains exact upon taking the quotient
	by the action of $\tilde{\xi}$, and the map 
	$$
		H^2(\mc{O},P^{\dagger}(1)) \to \bigoplus_{\ell \mid N} H^2(\Q_{\ell},P^{\dagger}(1))
	$$
	is an isomorphism.  
	By the diagram \eqref{zerodiag}, it therefore suffices to show that 
	$H^1(\Q_{\ell},Q^{\dagger}(1)) = 0$ for all primes $\ell \mid N$.  We verify this claim.
	
	Let $K_{\ell} = \Q_{\ell}(\mu_{Np^{\infty}})$, and set $\tGa_{\ell} = \Gal(K_{\ell}/\Q_{\ell})$,
	\begin{eqnarray*}
		\Delta_{\ell} = \Gal(K_{\ell}/\Q_{\ell,\infty}), &\mr{and}& 
		\Gamma_{\ell} = \Gal(K_{\ell}/\Q_{\ell}(\mu_{Np})).
	\end{eqnarray*}
	Inflation-restriction provides an exact sequence
	$$
		0 \to H^1(\tGa_{\ell},Q^{\dagger}(1))
		\to H^1(\Q_{\ell},Q^{\dagger}(1)) \to 
		H^1(K_{\ell},Q^{\dagger}(1))^{\tGa_{\ell}}.
	$$
 	We have $H^1(K_{\ell},Q^{\dagger}(1)) \cong Q^{\dagger}$
	by Kummer theory and the valuation map (since all roots of unity are infinitely divisible by $p$ in 
	$K_{\ell}^{\times}$). 
	As $\Delta_{\ell}$ acts on $Q^{\dagger}$ through the restriction 
	of $\theta^{-1}$, the $\Delta_{\ell}$-invariants of $Q^{\dagger}$ are trivial by Hypothesis 
	\ref{thetahyp}b.  So, we have $H^1(K_{\ell},Q^{\dagger}(1))^{\tGa_{\ell}} = 0$.
	Moreover, since $\Delta_{\ell}$ has prime-to-$p$
	order, inflation provides an isomorphism
	$$
		H^1(\Ga_{\ell},(Q^{\dagger}(1))^{\Delta_{\ell}}) \xrightarrow{\sim} 
		H^1(\tGa_{\ell},Q^{\dagger}(1)),
	$$
	and again the inertia subgroup of $\Delta_{\ell}$ acts nontrivially on $Q^{\dagger}(1)$ by assumption.
\end{proof}

\begin{lemma}
	The connecting homomorphisms $\partial_P^{\dagger} \colon H^1(\Q_{\ell},P^{\dagger}(1)) \to H^2(\Q_{\ell},\Lai \cozp P(1))$ for 
	$\ell \mid N$ are all isomorphisms.
\end{lemma}

\begin{proof}
	We have an exact sequence
	$$
		H^1(\Q_{\ell},\Lai \cozp P(1)) \to H^1(\Q_{\ell},P^{\dagger}(1))
		\xrightarrow{\partial_P^{\dagger}} H^2(\Q_{\ell},\Lai \cozp P(1)) \xrightarrow{\tilde{\xi}} H^2(\Q_{\ell},\Lai \cozp P(1)),
	$$
	and $H^1(\Q_{\ell},\Lai \cozp P(1)) = 0$ since $\ell$ is unramified in $\Q_{\infty}$, 
	while $H^2(\Q_{\ell},\Lai \cozp P(1)) \cong \bigoplus_{v \mid \ell} P$ via the invariant map.  Since $\tilde{\xi}$
	acts as $\xi$ on $P$, and $\xi$ kills $P$, we have the result.
\end{proof}

\begin{proposition} \label{snakesquare}
	The square 
	$$
		\SelectTips{cm}{} \xymatrix{
			H^1(\mc{O},T^{\dagger}(1)) \ar[r] \ar[d] & H^1(\mc{O},Q^{\dagger}(1)) \ar[d]^{-\Theta^{\dagger}} \\
			\bigoplus_{\ell \mid Np} H^1(\Q_{\ell},P^{\dagger}(1)) \ar[r] & H^2_c(\mc{O},P^{\dagger}(1)) 
		}
	$$
	is commutative.
\end{proposition}

\begin{proof}
	Applying Lemma \ref{connzero}, we have a diagram 
	$$
		\SelectTips{cm}{} \xymatrix{
			&& H^1(\mc{O},Q^{\dagger}(1)) \ar@{=}[d] \\
			 H^1(\mc{O},P^{\dagger}(1)) \ar[r] \ar[d] & H^1(\mc{O},T^{\dagger}(1)) \ar[r] \ar[d] & 
			H^1(\mc{O},Q^{\dagger}(1)) \ar[r] & 0 \\
			 \bigoplus_{\ell \mid Np} H^1(\Q_{\ell},P^{\dagger}(1)) \ar@{=}[r] \ar[d] & 
			 \bigoplus_{\ell \mid Np} H^1(\Q_{\ell},P^{\dagger}(1)) \\
			 H^2_c(\mc{O},P^{\dagger}(1))
		}
	$$
	with exact rows and columns.
	The snake lemma map from the diagram
	is then the negative of the connecting homomorphism $\Theta^{\dagger}$ by a standard lemma.
\end{proof}

We now have that all squares in the diagram
$$
	\SelectTips{cm}{} \xymatrix@C=15pt{
		H^1(\mc{O},T^{\dagger}(1)) \ar[r] \ar[d] & H^1(\mc{O},Q^{\dagger}(1)) \ar[d]^{-\Theta^{\dagger}} \ar[r] 
		& H^1(\mc{O},Q(1)) \ar[d]^{-\Phi} \ar[r]^-{\varepsilon_Q} & H^2(\mc{O},Q(1)) \ar[d]^{\Upsilon}  \\
		\bigoplus_{\ell \mid Np} H^1(\Q_{\ell},P^{\dagger}(1)) \ar[r] & H^2_c(\mc{O},P^{\dagger}(1)) 
		& H^2_c(\mc{O},P(1)) \ar[l]_-{\sim} \ar[r]^-{\varepsilon_P} & H^3_c(\mc{O},P(1))
	}
$$
are commutative.  

\section{Local study} \label{local_study}

In this section, we let $\mf{R}$ denote a complete Noetherian semi-local $\zp$-algebra. We let $A$ denote a compact, 
unramified $\mf{R}\ps{G_{\qp}}$-module. Exactly when discussing this general setting, we shall allow $p$ to be any odd prime.

In its simplest form, a Coleman map assigns power series 
to norm compatible systems of $p$-units in the cyclotomic $\zp$-extension of $\qp$; this is the case of the $\zp\ps{G_{\qp}}$-module $A = \zp$. 
In general, the Coleman map attached to $A$ is a map
$$
	\Cole_A \colon H^1_{\Iw}(\Q_{p,\infty},A(1)) \to X^{-1}\La \cozp D(A),
$$
closely related to a dual exponential map (see \cite[4.2.10]{fk-pf}).

We construct a Coleman-type map for the intermediate quotient
$$
	A^{\dagger} = (\Lai \cozp A)/X\alpha(\Lai \cozp A)
$$ 
with respect to an element $\alpha \in \La \cozp \mf{R}$. We call this map
$$
	\Cole_A^{\dagger}\colon H^1(\qp,A^{\dagger}(1)) \to \mf{C}^{\star}(A)
$$ 
an intermediate Coleman map, and its value group is a certain $\La \cozp \mf{R}$-module $\mf{C}^{\star}(A)$ annihilated by $X\alpha$.
We show that $(1-\varphi^{-1})\Cole_A$ modulo $X\alpha$ agrees with the composition of $\Cole_A^{\dagger}$ with the map from Iwasawa cohomology, and the quotient of $\mf{C}^{\star}(A)$ by the image of this composition is isomorphic to the maximal unramified quotient of $A$.

For $A = \mcT_{\quo}$, we also show that there exists an intermediate local zeta map 
$$
	z^{\dagger}_{\quo} \colon \La \cozp \mcS_{\theta} \to H^1(\qp,\mcT_{\quo}^{\dagger}(1))
$$ 
such that $(1-U_p)z_{\quo}^{\dagger}$ agrees with a local zeta map $z_{\quo}$ to
Iwasawa cohomology (after mapping to intermediate cohomology). While $\Cole \circ z_{\quo}$ is multiplication by an element
$\alpha \in \La \cozp \mf{h}_{\theta}$ with $\alpha(0) \equiv \xi' \bmod I$, the composition $\overline{\Cole}^{\dagger} \circ z_{\quo}^{\dagger}$ induces a map $P \to P$ that is just the identity: see Proposition \ref{zquodagger}. This occurrence of $1$ serves as the replacement for $\xi'$ on the right-hand side of \eqref{keyident}.

\subsection{Coleman maps} \label{coleman}

In this subsection, we review the theory of Coleman maps, largely following \cite[\S 4]{fk-pf}.

Let $\mc{U}_{\infty}^{\ur}$ (resp., $\mc{K}_{\infty}^{\ur}$) denote the $p$-completion of the group of norm compatible sequences of units (resp., of nonzero elements) in the tower given by the cyclotomic $\zp$-extension $\Q_{p,\infty}^{\ur}$ of $\Q_p^{\ur}$. Recall
that $W$ denotes the completion of the valuation ring of $\Q_p^{\ur}$. 
Attached to a sequence of norm compatible elements $(u_r)_r$ in $\mc{U}_{\infty}^{\ur}$, there is a unique power series 
$f(y) \in W\ps{y}$ with $f(\zeta_{p^r}-1) = \Fr_p^r(u_r)$ for all $r$, known as the Coleman power series of $(u_r)_r$ (see 
\cite[Theorem A]{coleman}).

\begin{definition}
	The \emph{Coleman map} $\Cole \colon \mc{K}_{\infty}^{\ur} \to X^{-1}W\ps{X}$ is 
	the unique map of $\La$-modules restricting to a map $\mc{U}_{\infty}^{\ur} \to W\ps{X} 
	= W\ps{1+p\zp}$
	defined on $(u_r)_{r \ge 1} \in \mc{U}_{\infty}^{\ur}$ with $u_r \in \Q_{p,r}^{\ur}$ by
	$$
		[\Cole((u_r)_{r \ge 1})](x) = \left(1-\frac{\psi}{p}\right)\log(f(x-1)).
	$$
	Here, $W\ps{X}$ acts continuously and $W$-linearly on $W\ps{x}$ with the result of $h \in W\ps{X}$
	acting on $x$ denoted by $[h](x)$, via the action determined by $[a](x) = x^a \in W\ps{x}$ for $a \in 1+p\zp$.
	Also, $f(x-1) \in W\ps{x-1}$ is the Coleman power series of $(u_r)_r$, and 
	$\psi$ is defined on $g(x) \in W[p^{-1}]\ps{x}$ by $\psi(g)(x) = \Fr_p(g)(x^p)$.
\end{definition}

We can extend this definition as follows.

\begin{definition}
	The Coleman map for $A$ is the map
	$$
		\Cole_A \colon H^1_{\Iw}(\Q_{p,\infty},A(1)) \to X^{-1}D(A)\ps{X}
	$$
	of $\La \cozp \mf{R}$-modules defined as the composition
	$$
		H^1_{\Iw}(\Q_{p,\infty},A(1)) \xrightarrow{\Inf}
		(A \cozp \mc{K}_{\infty}^{\ur})^{\Fr_p=1} \xrightarrow{1 \otimes \Cole}
		(A \cozp X^{-1}W\ps{X})^{\Fr_p=1} \xrightarrow{\sim} X^{-1}D(A)\ps{X},
	$$
	where $\Fr_p$ acts diagonally on the tensor products.
\end{definition}

The following is a slight extension, allowing $A^{\Fr_p=1}$ to be nonzero,
of the restriction of \cite[4.2.7]{fk-pf} to invariants for $\Delta \cong \Gal(\qp(\mu_{p^{\infty}})/\Q_{p,\infty})$. Note that $\Cole_A$ agrees with the map denoted $\Cole$ in \cite{fk-pf} on the fixed part under $\Gal(\qp(\mu_p)/\qp)$.

\begin{lemma} \label{Colimage}
	The map $\Cole_A$ is injective with image in $X^{-1}D(A)\ps{X}$ equal to 
	$$
		\mfC(A) = X^{-1}A^{\Fr_p=1} + D(A)\ps{X}.
	$$ 
\end{lemma}

\begin{proof}
	Since $A(1)$ has no $G_{\Q_{p,\infty}^{\ur}}$-fixed part and $\Gal(\Q_{p,\infty}^{\ur}/\Q_{p,\infty}) \cong \hat{\Z}$ has
	cohomological dimension $1$, the inflation map $\Inf$ in the definition of $\Cole_A$ is an isomorphism. It is well known that
	the Coleman map $\Cole$ is injective and, as follows for instance from the proof of \cite[4.2.7]{fk-pf}, it restricts to
	an isomorphism $\mc{U}_{\infty}^{\ur} \xrightarrow{\sim} W\ps{X}$. In particular, $\Cole_A$ is injective.
	
	It follows that we have an exact sequence
	$$
		0 \to A \cozp W\ps{X} \to A \cozp \mc{K}_{\infty}^{\ur} \to A \to 0
	$$
	with the first map being the inverse of $1 \otimes \Cole$ and the second determined by the valuation map on the norm to $\Q_p^{\ur}$
	of an element of $\mc{K}_{\infty}^{\ur}$.
	The kernel of $1-\Fr_p$ applied to this sequence gives
	$$
		0 \to D(A)\ps{X} \to H^1_{\Iw}(\Q_{p,\infty},A(1)) \to A^{\Fr_p=1} \to 0,
	$$
	the surjectivity since $\mc{K}_{\infty}^{\ur}$ contains the Frobenius fixed sequence that is the projection of 
	$(1-\zeta_{p^n})_n$ to the $\Delta$-invariant group. By the injectivity
	of $\Cole_A$, this forces the induced map $A^{\Fr_p=1} \to X^{-1}D(A)\ps{X}/D(A)\ps{X}$ to have image $X^{-1}A^{\Fr_p=1}$.
	Since the image of $\Cole_A$ contains $D(A)\ps{X}$, it must then equal $\mf{C}(A)$.
\end{proof}

\begin{remark}
	 If $A^{\Fr_p = 1} = 0$, then $\Cole_A$ is an isomorphism $H^1_{\Iw}(\Q_{p,\infty},A(1)) \to D(A)\ps{X}$.
	This occurs, for instance, for $A = \mcT_{\quo}$ by \cite[3.3.3]{fk-pf}.
\end{remark}

In addition to $\Cole_A$, we also have a homomorphism at the level of $\qp$, following \cite[4.2.2]{fk-pf}.

\begin{definition}
	We let
	$$
		\Cole_A^{\flat} \colon H^1(\qp,A(1)) \to D(A)
	$$
	denote the map of $\mf{R}$-modules given by the composition
	$$
		H^1(\qp,A(1)) \to 
		(A \cozp (1+pW))^{\Fr_p=1} \xrightarrow{\left(1-\tfrac{\varphi}{p}\right)\log}
		(A \cozp W)^{\Fr_p=1} \xrightarrow{\sim} D(A),
	$$
	where $\varphi = 1 \otimes \Fr_p$ is as in Definition \ref{Disos}, and the first map is induced by restriction to $\Q_p^{\ur}$ 
	and the map
	$$	
		H^1(\qp^{\ur},\zp(1)) \xrightarrow{\sim} p^{\zp} \times (1+pW) \to 1+pW
	$$
	given by projection to the second summand.
\end{definition}

\begin{remark} \label{splitColflat}
	The map $\Cole^{\flat}_A$ is in general only split surjective with kernel $A^{\Fr_p=1}$.
	It has a canonical splitting given by the valuation map
	$$
		H^1(\qp,A(1)) \xrightarrow{\Inf} H^1(\qp^{\ur},A(1))^{\Fr_p=1} \xrightarrow{v_p} A^{\Fr_p = 1},
	$$
	as the valuation map $v_p$ has kernel $(A \cozp (1+pW))^{\Fr_p=1}$.
\end{remark}

Recall that 
\begin{equation} \label{H2}
	H^2(\qp,A(1)) \cong A/(\Fr_p-1)A \cong D(A)/(\varphi-1)D(A)
\end{equation}
via Tate duality, i.e., the invariant map of local class field theory.  The following is \cite[4.2.4]{fk-pf}.

\begin{lemma} \label{Colconn}
	The composition of $\Cole_A^{\flat}$ with the quotient map $D(A) \to D(A)/(\varphi-1)D(A)$ is identified through \eqref{H2} 
	with the connecting homomorphism
	$$
		\partial_A \colon H^1(\qp,A(1)) \to H^2(\qp,A(1))
	$$
	of Definition \ref{bockstein}.
\end{lemma}

\begin{proof}
	By replacing $A$ by $A/(\Fr_p-1)A$, we may suppose that $A$ has trivial Galois action, and it then suffices to consider
	$A = \zp$. The connecting homomorphism $\partial_A$ is given by left cup product with $-\chi$ by Remark \ref{connmapsign}. 
	Note that for $a \in \qp^{\times}$,
	we have $\chi \cup a = \chi(\rho(a))$, where $\rho \colon \qp^{\times} \to G_{\qp}^{\mr{ab}}$ is the local reciprocity map
	(cf. \cite[Chapter XIV, Propositions 1.3 and 2.5]{serre}). 
	But 
	$$
		\rho(u)(\zeta_{p^n}) = \zeta_{p^n}^{u^{-1}}
	$$ 
	for $u \in 1+p\zp$ and $\rho(p)(\zeta_{p^n}) = \zeta_{p^n}$. Then
	$-\chi(\rho(p)) = 0$ and $-\chi(\rho(u)) = (1-p^{-1})\log(u)$ for $u \in 1+p\zp$. Thus, $\partial_{\zp} = \Cole_{\zp}^{\flat}$.
\end{proof}

The relationship between $\Cole$ and $\Cole^{\flat}$ is given by the following \cite[4.2.9]{fk-pf}.

\begin{proposition} \label{Col_vs_flat}
	Let $\ev_0 \colon D(A)\ps{X} \to D(A)$ denote evaluation at $0$, and let 
	$\cor$ be the corestriction map $H^1_{\Iw}(\Q_{p,\infty},A(1)) \to H^1(\qp,A(1))$.  Then we have
	$$
		\ev_0 \circ (1-\varphi^{-1})\Cole_A = \Cole_A^{\flat} \circ \cor
	$$
	as maps $H^1_{\Iw}(\Q_{p,\infty},A(1)) \to D(A)$.
\end{proposition}

Abusing notation for easier comparison with the inverse $\Cole_A^{-1}$ of $\Cole_A$, we 
introduce the following notation for the canonical splitting of $\Cole_A^{\flat}$ determined by Remark \ref{splitColflat}.

\begin{definition} \label{Col_flat_inv}
	We denote the canonical splitting of the surjection $\Cole_A^{\flat}$ by
	$$
		(\Cole_A^{\flat})^{-1} \colon D(A) \to H^1(\qp,A(1)).
	$$
\end{definition}

We then have the following corollary of Proposition \ref{Col_vs_flat}.

\begin{corollary} \label{inverse_equal}
	We have the equality 
	$$
		\cor \circ \Cole_A^{-1} = (\Cole_A^{\flat})^{-1} \circ \ev_0 \circ (1 - \varphi^{-1}) 
	$$
	as maps $D(A)\ps{X} \to H^1(\Q_p,A(1))$.
\end{corollary}

\begin{proof}
	The $p$-adic valuation and $\Cole_A^{\flat}$ provide an isomorphism $H^1(\qp,A(1)) \to A^{\Fr_p=1} 
	\oplus D(A)$, and $(\Cole_A^{\flat})^{-1}$ is the resulting injection from $D(A)$. 
	The image of $D(A)\ps{X}$ under $\Cole_A^{-1}$  is contained in $(A \cozp \mc{U}_{\infty}^{\ur})^{\Fr_p=1}$, 
	as in the proof of Lemma \ref{Colimage}.
	The image of the restriction of $\cor$ to the latter group is $(A \cozp (1+pW))^{\Fr_p=1}$, which is the image of $(\Cole_A^{\flat})^{-1}$.
	Thus, we can compose the equation of 
	Proposition \ref{Col_vs_flat} with $(\Cole^{\flat}_A)^{-1}$ and precompose with $\Cole_A^{-1}$ to obtain
	the equality.
\end{proof}

\subsection{Intermediate Coleman maps} \label{inter Coleman}

In this subsection, we construct a map $\Cole_A^{\dagger}$ that plays an analogous role
to $\Cole_A^{\flat}$ for what we call an intermediate quotient of $\La^{\iota} \cotimes{\zp} A$. 

We suppose that $\mf{R}$ is local to simplify the discussion and fix an element 
$$
	\alpha \in \La \cozp \mf{R} = \mf{R}\ps{X}
$$
with nonzero image in $\mf{k}\ps{X}$ for $\mf{k}$ the residue field of $\mf{R}$. The multiplication-by-$\alpha$
map is then injective on $\La \cozp A$.

\begin{definition} \label{interquot}
	The \emph{intermediate quotient} $A^{\dagger}$ of $\La^{\iota} \cozp A$ with respect to $\alpha$ is the
	compact $(\La \cozp \mf{R})\ps{G_{\qp}}$-module defined as
	$$
		A^{\dagger} = (\Lai \cozp A)/X\alpha(\Lai \cozp A).
	$$
\end{definition}

By Weierstrass preparation, $A^{\dagger}$ is a finite direct sum of copies of $A$ as an $\mf{R}$-module.

\begin{definition}\
	\begin{enumerate}
		\item[a.] Set 
		$\mfC^{\dagger}(A) = \mfC(A)/X\alpha\mfC(A)$,
		where $\mfC(A)$ is the image of $\Cole_A$, as in Lemma \ref{Colimage}.
		\item[b.] 
		Let $\mfC^{\star}(A)$ denote the pushout of the diagram
		$$
			\mfC^{\dagger}(A) \xleftarrow{\alpha} D(A) \xrightarrow{1-\varphi^{-1}} D(A),
		$$
		where the first map sends $a \in D(A)$ to $\alpha (1 \otimes a) \in \mfC^{\dagger}(A)$.
		\item[c.] Let 
		$$
			\bar{A} = A/(\Fr_p-1)A \cong D(A)/(\varphi-1)D(A),
		$$
		\item[d.] Let 
		\begin{eqnarray*}
			\inv_A \colon H^2_{\Iw}(\Q_{p,\infty},A(1)) \xrightarrow{\sim} \bar{A} &\mr{and}& \inv_A \colon H^2(\qp,A(1)) \to \bar{A}
		\end{eqnarray*} 
		denote the invariant maps, which agree via corestriction from Iwasawa to usual cohomology.
	\end{enumerate}
\end{definition}	

The pushout $\mfC^{\star}(A)$ has the following 
relatively simple explicit descriptions in the cases that $A^{\Fr_p=1} = 0$ or $A^{\Fr_p=1} = A$.

\begin{lemma} \label{simpleCstar}
	If $A^{\Fr_p=1} = 0$, then 
	$$
		\mfC^{\star}(A) \cong \frac{(1-\varphi,\alpha)\La \cozp D(A)}{X\alpha (\La \cozp D(A))}
	$$
	as $\La \cozp \mf{R}$-modules.
	Moreover, the injective pushout map from
	$$
		\mfC^{\dagger}(A) = \frac{\La \cozp D(A)}{X\alpha(\La \cozp D(A))}
	$$ 
	to $\mfC^{\star}(A)$ is given by multiplication by $1-\varphi^{-1}$.
\end{lemma}

\begin{proof}
	It suffices to see that
	$$
		\mf{C}^{\star}(A) = \frac{(\mf{C}^{\dagger}(A) \oplus D(A))}{\{ (\alpha x, (\varphi^{-1}-1) x) \mid x \in D(A) \}}
		\xrightarrow{(1-\varphi^{-1},\alpha)} \frac{(1-\varphi,\alpha)\La \cozp D(A)}{X\alpha (\La \cozp D(A))}
	$$
	is an isomorphism. Well-definedness and surjectivity are clear from the definitions.
	
	Let us prove injectivity. Since $\alpha$ reduces to a nonzero element of $\mf{k}\ps{X}$, Lemma \ref{diagmult}
	tells us that it provides an injective
	endomorphism of $\La \cozp (D(A)/(1-\varphi)D(A))$. Thus, if $y, z \in \La \cozp D(A)$ 
	are such that $(1-\varphi^{-1})y + \alpha z = 0$, then $z = (\varphi^{-1}-1)x$
	for some $x \in \La \cozp D(A)$.  As $1-\varphi^{-1}$ is injective
	on $\La \cozp D(A)$ by assumption (see Remark \ref{Disos}b), this forces $y = \alpha x$. Then $(y,z(0)) = 
	(\alpha x, (\varphi^{-1}-1) x(0))$ has trivial image in $\mf{C}^{\star}(A)$.
\end{proof}

\begin{lemma} \label{splitCstar}
	If $\Fr_p$ acts trivially on $A$, then the pushout $\mfC^{\star}(A)$ is the direct sum
	$$
		\mfC^{\star}(A) = \mfC^{\dagger}(A) \oplus D(A)
	$$
	of $\La \cozp \mf{R}$-modules. Moreover, multiplication by $X$ induces an isomorphism 
	$\mfC^{\dagger}(A) \xrightarrow{\sim} A^{\dagger}$.
\end{lemma}

\begin{proof}
	Note that $1-\varphi^{-1}$ is the trivial endomorphism of $D(A)$ by assumption. Given this, we have
	$\mf{C}^{\dagger}(A) = (X^{-1}\La \cozp A)/\alpha(\La \cozp A)$, so $\alpha \colon D(A) \to \mf{C}^{\dagger}(A)$
	is also zero. The second statement is a direct consequence of this description of $\mfC^{\dagger}(A)$ and the 
	definition of $A^{\dagger}$.
\end{proof}

The following defines an \emph{intermediate Coleman map} $\Cole^{\dagger}_A$ for the unramified module $A$.

\begin{theorem} \label{loccohdag}
	There is an isomorphism 
	$$
		\Cole_A^{\dagger} \colon H^1(\qp,A^{\dagger}(1)) 
		\xrightarrow{\sim} \mfC^{\star}(A)
	$$ 
	of $\La \cozp \mf{R}$-modules fitting in an isomorphism of exact sequences
	$$
		\SelectTips{cm}{} \xymatrix{
		0 \ar[r] & 
		H^1_{\Iw}(\Q_{p,\infty},A(1))/X\alpha \ar[r] \ar[d]^{\Cole_A} &
		H^1(\qp,A^{\dagger}(1)) \ar[d]^-{\Cole_A^{\dagger}} \ar[r] &
		H^2_{\Iw}(\Q_{p,\infty},A(1)) \ar[r] \ar[d]^{\inv_A} & 0 \\
		0 \ar[r] & \mfC^{\dagger}(A) \ar[r] & 
		\mfC^{\star}(A) \ar[r]^-{\psi} &
		\bar{A} \ar[r] & 0, \\
		}
	$$
	where the left lower horizontal map is given by the pushout, and
	$\psi$ is inverse to the isomorphism $\bar{A} \to \mfC^{\star}(A)/\mfC^{\dagger}(A)$
	induced by the other pushout map.
\end{theorem}

\begin{proof}
	We shall construct an isomorphism $\lambda \colon \mf{C}^{\star}(A) \to H^1(\qp,A^{\dagger}(1))$ such that the desired map 
	$\Cole_A^{\dagger}$ is its inverse. 
	Consider the composition 
	$$
		D(A) \to H^1(\qp,A^{\dagger}(1))
	$$ 
	of $(\Cole_A^{\flat})^{-1}$ with
	the map $\alpha \colon H^1(\qp,A(1)) \to H^1(\qp,A^{\dagger}(1))$
	induced by $\alpha \colon A \to A^{\dagger}$.
	We claim that the composition 
	$$	
		D(A) \xrightarrow{1-\varphi^{-1}} D(A) \to H^1(\qp,A^{\dagger}(1))
	$$
	agrees with the composition
	$$
		D(A) \xrightarrow{\alpha} \mfC^{\dagger}(A) \xrightarrow{\Cole_A^{-1}} H^1_{\Iw}(\Q_{p,\infty},A(1))/X\alpha \to 
		H^1(\qp,A^{\dagger}(1)),
	$$ 
	where we abuse notation (as in the statement of the theorem) and use $\Cole_A$ to denote
	its reduction modulo $X\alpha$. 
	Given the claim, we define $\lambda$ as the map given by the universal property of the pushout
	$\mfC^{\star}(A)$.
	
	To see the claim, consider the diagram
	$$
		\SelectTips{cm}{} \xymatrix{
		D(A) \ar[r] \ar[d]_{1-\varphi^{-1}} & \mfC(A)/X\mfC(A) \ar[r]^-{\alpha} 
		\ar[d]^{\Cole_A^{-1}} & \mfC(A)/X\alpha \mfC(A) \ar[d]^{\Cole_A^{-1}}  \\
		D(A) \ar[dr]_{(\Cole_A^{\flat})^{-1}} & H^1_{\Iw}(\Q_{p,\infty},A(1))/X \ar[r]^-{\alpha} \ar[d]^{\cor} 
		& H^1_{\Iw}(\Q_{p,\infty},A(1))/X\alpha \ar[d] \\
		& H^1(\qp,A(1)) \ar[r]^-{\alpha} & H^1(\qp,A^{\dagger}(1))
		}
	$$
	in which the two compositions are found by tracing its perimeter.
	The two right-hand squares clearly commute. Since the multiplication-by-$\alpha$ maps in this diagram are all injective, 
	we are reduced to the commutativity of the left part of the diagram. 
	This commutativity is Corollary \ref{inverse_equal}, as it tells us that 
	the two compositions $\mfC(A)/X\mfC(A) \to H^1(\qp,A(1))$ agree on the image of $D(A)$ in 
	$$
		\mfC(A)/X\mfC(A) \cong X^{-1}A^{\Fr_p=1} \oplus D(A)/A^{\Fr_p=1}.
	$$
	Thus, we have the claim, and thereby the map $\lambda$.
	
	The left-hand square in the diagram
	\begin{equation} \label{invdiag}
		\SelectTips{cm}{} \xymatrix{
		0 \ar[r] & \mfC^{\dagger}(A) \ar[r] \ar[d]^-{\Cole_A^{-1}} & 
		\mfC^{\star}(A) \ar[r]^-{\psi} \ar[d]^-{\lambda} &
		\bar{A} \ar[r] \ar[d]^{\inv_A^{-1}} & 0 \\
		0 \ar[r] & H^1_{\Iw}(\Q_{p,\infty},A(1))/X\alpha \ar[r] &
		H^1(\qp,A^{\dagger}(1))  \ar[r]^-{\partial_A^{\dagger}} &
		H^2_{\Iw}(\Q_{p,\infty},A(1)) \ar[r] & 0,
		}
	\end{equation}
	commutes by the universal property defining $\lambda$, where we have departed from our earlier notation to also use 
	$\partial_A^{\dagger}$ to denote the connecting map in $\qp$-cohomology for the Tate twist of  
	$$
		0 \to \Lai \cozp A \xrightarrow{X\alpha} \Lai \cozp A \to A^{\dagger} \to 0.
	$$
	For the commutativity of the right-hand square, consider the diagram
	\begin{equation} \label{Coldagflat}
		\SelectTips{cm}{} \xymatrix{
		\\
		H^2(\qp,A(1)) \ar@/^3pc/[rrr]_{\inv_A}
		& H^1(\qp,A(1))  \ar[d]^{\alpha} \ar[l]_-{\partial_A} & D(A) \ar[l]_-{(\Cole_A^{\flat})^{-1}} 
		\ar[r] \ar[d]^{\alpha} & \bar{A} \ar@{=}[d] \\
		H^2_{\Iw}(\Q_{p,\infty},A(1)) \ar[u]^{\wr}_{\cor}  \ar@/_3pc/[rrr]^{\inv_A}
		& H^1(\qp,A^{\dagger}(1))  \ar[l]_-{\partial_A^{\dagger}} & \mfC^{\star}(A) \ar[l]_-{\lambda} \ar[r]^-{\psi} & \bar{A},\\
		&
		}
	\end{equation}
	in which the outer part and the right-hand square commute by construction.
	The left-hand square of \eqref{Coldagflat} commutes due to the maps of exact sequences
	$$
		\SelectTips{cm}{} \xymatrix{
		0 \ar[r] & A \ar[r]^-X \ar@{=}[d] & \Lai/X^2 \cozp A \ar[r] \ar[d]^{\alpha} & A \ar[r] \ar[d]^{\alpha} & 0 \\
		0 \ar[r] & A \ar[r]^-{X\alpha} & (\Lai \cozp A)/X^2\alpha \ar[r] & A^{\dagger} \ar[r] \ar@{=}[d] & 0 \\
		0 \ar[r] & \Lai \cozp A \ar[r]^-{X\alpha} \ar@{->>}[u] & \Lai \cozp A \ar[r] \ar@{->>}[u] & A^{\dagger} \ar[r] & 0.
		} 
	$$
	The middle square commutes as the larger diagram
	$$
		\SelectTips{cm}{} \xymatrix@C=40pt{
		H^1(\qp,A(1)) \ar[d]^{\alpha} & D(A) \ar[l]_-{(\Cole_A^{\flat})^{-1}} \ar[d]^{\alpha} \\
		H^1_{\Iw}(\Q_{p,\infty},A(1))/X\alpha \ar@{^{(}->}[d] & \mf{C}^{\dagger}(A) \ar[l]_-{\Cole_A^{-1}}^-{\sim} \ar@{^{(}->}[d]  \\
		H^1(\qp,A^{\dagger}(1)) & \mf{C}^{\star}(A) \ar[l]_-{\lambda}
		}
	$$
	does by our earlier claim and the commutativity of the left-hand square of \eqref{invdiag}.
	Finally, the top row commutes with the invariant map by Lemma \ref{Colconn}.
	Consequently, the lower row commutes with the invariant map as well, and therefore the right-hand square in \eqref{invdiag}
	commutes. Setting $\Cole_A^{\dagger} = \lambda^{-1}$,
	we have the theorem.
\end{proof}	

\begin{remark}
	The middle square of the commutative diagram \eqref{Coldagflat} gives a comparison between $\Cole_A^{\flat}$ and 
	$\Cole_A^{\dagger}$. Note that in the case $\alpha = 1$, the map $\Cole_A^{\flat}$ is defined as a split surjection
	(as we have kept the conventions of \cite{fk-pf}), whereas $\Cole_A^{\dagger}$ is an isomorphism to $A^{\Fr_p = 1} \oplus D(A)$.
\end{remark}

\begin{remark}
	In \cite[Section 4]{fk-pf}, Coleman maps are defined on the Iwasawa cohomology of $A(1)$ for the extension 
	$\qp(\mu_{p^{\infty}})$ of $\qp$,
	as opposed to just $\Q_{p,\infty}$. The second Iwasawa cohomology groups of $A(1)$ for each of these 
	extensions are isomorphic via corestriction. Outside of the trivial eigenspace
	for $\Gal(\Q(\mu_{p^{\infty}})/\Q_{p,\infty})$ that we consider here, analogously defined intermediate Coleman maps 
	would simply amount to reductions of the original Coleman maps. In other words, we have restricted our discussion
	exactly to the setting where our constructions can be of interest.
\end{remark}

We end this subsection by explaining what Theorem \ref{loccohdag} tells us in the setting of central interest to this work, 
returning at this point to the notation and hypotheses of earlier sections.
Recall from Theorem \ref{DTquo} that $D(\mcT_{\quo})$ is canonically isomorphic to $\mfS_{\theta}$, and from
Remark \ref{quo_results}a that the action of $\varphi^{-1}$ on $D(\mcT_{\quo})$ agrees with the action of $U_p$ on $\mfS_{\theta}$. As in \cite[3.3.3]{fk-pf}, these facts imply that $\mcT_{\quo}^{\Fr_p=1} = 0$.
Given the identifications of Lemma \ref{simpleCstar}, Theorem \ref{loccohdag} has the following corollary for $A = \mcT_{\quo}$
(for any choice of $\alpha \in \La \cozp \mf{h}_{\theta}$ as above). For this choice of $A$, we drop the subscript from
our Coleman and invariant maps for compactness of notation.

\begin{corollary} \label{loccohdagS}
	Set 
	$$
		\mfS_{\theta}^{\star} = (\alpha,1-U_p)\La \cozp \mfS_{\theta}/X\alpha(\La \cozp \mfS_{\theta})
		\subset \mfS_{\theta}^{\dagger} = \La \cozp \mfS_{\theta}/X\alpha(\La \cozp \mfS_{\theta}).
	$$
	There is an isomorphism
	$$
		\Cole^{\dagger} \colon H^1(\qp,\mcT_{\quo}^{\dagger}(1)) \to \mfS_{\theta}^{\star}
	$$
	of $\La \cozp \mf{h}_{\theta}$-modules fitting in an isomorphism of exact sequences
	$$
		\SelectTips{cm}{} \xymatrix{
		0 \ar[r] & 
		H^1_{\Iw}(\Q_{p,\infty},\mcT_{\quo}(1))/X\alpha \ar[r] \ar[d]^{\Cole} &
		H^1(\qp,\mcT_{\quo}^{\dagger}(1)) \ar[d]^-{\Cole^{\dagger}} \ar[r] &
		H^2_{\Iw}(\Q_{p,\infty},\mcT_{\quo}(1)) \ar[r] \ar[d]^{\inv} & 0 \\
		0 \ar[r] & \mfS_{\theta}^{\dagger} \ar[r]^{1-U_p} & 
		\mfS_{\theta}^{\star} \ar[r]^-{\psi} &
		 S_{\theta} \ar[r] & 0, \\
		}
	$$
	where $S_{\theta} = \mfS_{\theta}/(U_p-1)\mfS_{\theta}$ and where 
	$\psi$ factors through the inverse to the map induced by multiplication by $\alpha$ on the
	cokernel of multiplication by $1-U_p$ on $\mf{S}^{\star}_{\theta}$. 
\end{corollary}

We make the following definition of a ``reduced'' Coleman map for later use.

\begin{definition} \label{Coledagbar}
	We let
	$$
		\overline{\Cole}^{\dagger} \colon H^1(\qp,\mcT_{\quo}^{\dagger}(1)) \to S_{\theta}.
	$$
	denote the composition $\psi \circ \Cole^{\dagger}$.
\end{definition}

\subsection{Local zeta maps} \label{localzeta}

In this subsection, we construct and employ an ad hoc local version of the global zeta map of Fukaya-Kato.  We shall see
how it ties in with global elements in Section \ref{global_study}.
 
Fix an isomorphism $\mc{M}_{\theta} \xrightarrow{\sim} \mf{M}_{\theta}$ of $\mf{H}_{\theta}$-modules that reduces to the canonical isomorphism 
$$
	\mc{M}_{\theta}/(U_p-1)\mc{M}_{\theta} \xrightarrow{\sim} \mf{M}_{\theta}/(U_p-1)\mf{M}_{\theta}.
$$
We use it, in particular, to identify $\mcS_{\theta}$ with $\mfS_{\theta}$ in the remainder of the paper.
We then have isomorphisms
$$
	\Hom_{\La \cozp \mf{h}_{\theta}}(\La \cozp \mc{S}_{\theta},\La \cozp \mfS_{\theta})
	\xrightarrow{\sim} \End_{\La \cozp \mf{h}_{\theta}}(\La \cozp \mfS_{\theta}) \xrightarrow{\sim} 
	\La \cozp \mf{h}_{\theta},
$$
the second being the inverse of the map that takes an element to the endomorphism it defines.

We will specify the following element $\alpha$ precisely in Section \ref{local_study}.

\begin{notation} \label{alpha}
	Let $\alpha \in \La \cozp \mf{h}_{\theta}$ denote an element with image 
	equal to the image of $\otil{\xi}_1 = X^{-1}(\otil{\xi} - 1 \otimes \xi)$ in the quotient of the ring $\La \cozp (\mf{h}/I)_{\theta}$
	by the image of $\otil{\xi}$.
\end{notation}	   

Note that $\alpha$ has nonzero image in $R/p\ps{X}$, as $X\otil{\xi}_1(X,0) \equiv \xi \bmod p$, so it satisfies the condition on
$\alpha$ in Section \ref{inter Coleman}. 
We may then define a \emph{local zeta map}. Its significance lies in that is induced by the restriction of a zeta map of Fukaya and Kato for our later good choice of $\alpha$.

\begin{definition} \label{zquo}
	Let $z_{\quo}$ denote the unique map of $\La \cozp \mf{h}_{\theta}$-modules
	$$
		z_{\quo} \colon \La \cozp \mcS_{\theta}  \to H^1_{\Iw}(\Q_{p,\infty},\mcT_{\quo}(1))
	$$
	such that $\Cole \circ z_{\quo} \colon \La \cozp \mcS_{\theta} \to \La \cozp \mfS_{\theta}$ 
	is identified with multiplication by $\alpha \in \La \cozp \mf{h}_{\theta}$.
\end{definition}

The following is due to Fukaya and Kato (see \cite[3.3.9, 4.3.8, 4.4.3, 8.1.2]{fk-pf}).

\begin{proposition}[Fukaya-Kato] \label{zquosharp}
	There exists a unique $\mf{h}_{\theta}$-module homomorphism
	$$
		z_{\quo}^{\sharp} \colon \mcS_{\theta} \to H^1(\qp,\mcT_{\quo}(1))
	$$
	for which 
	$$
		(1-U_p)z_{\quo}^{\sharp} \circ \ev_0 = \cor \circ z_{\quo}
	$$
	on $\La \cozp \mcS_{\theta}$, and such that $\Cole^{\flat} \circ z_{\quo}^{\sharp}$
	is multiplication by $\xi'$ modulo $I$.
\end{proposition}

\begin{proof}
	Since $z_{\quo}$ is defined so that $\Cole \circ z_{\quo}$ is multiplication by $\alpha$
	and $(1-U_p)\ev_0  \circ \Cole = \Cole^{\flat} \circ \cor$ by Proposition \ref{Col_vs_flat}, 
	we have
	that 
	$$
		\Cole^{\flat} \circ \cor \circ z_{\quo} = (1-U_p) \ev_0 \circ \Cole \circ z_{\quo}
		= (1-U_p) \alpha(0) \circ \ev_0.
	$$
	Since $\Cole^{\flat}$ is an isomorphism for $\mcT_{\quo}$, we can define $z_{\quo}^{\sharp}$ to be
	the unique map satisfying $\Cole^{\flat} \circ z_{\quo}^{\sharp} = \alpha(0)$.  As $\alpha(0)$ modulo $I$ 
	is $\tilde{\xi}_1(0) = \xi'$ by definition, we are done.
\end{proof}

We prove an analogue of Proposition \ref{zquosharp} 
not involving the derivative $\xi'$ for the intermediate quotient $\mcT_{\quo}^{\dagger}$.

\begin{proposition} \label{zquodagger}
	There exists a unique map 
	$$
		z_{\quo}^{\dagger} \colon \mcS_{\theta} \to H^1(\qp,\mcT_{\quo}^{\dagger}(1))
	$$
	of $\La \cozp \mf{h}_{\theta}$-modules 
	with the property that the square
	$$
		\SelectTips{cm}{} \xymatrix{
			\La \cozp \mcS_{\theta} \ar[r]^-{z_{\quo}^{\dagger} \circ \ev_0} \ar[d]^{z_{\quo}} 
			& H^1(\qp,\mcT_{\quo}^{\dagger}(1)) \ar[d]^{1-U_p} \\
			H^1_{\Iw}(\Q_{p,\infty},\mcT_{\quo}(1)) \ar[r] &  H^1(\qp,\mcT_{\quo}^{\dagger}(1)) 
		}
	$$
	commutes, and the composition $\overline{\Cole}^{\dagger} \circ z_{\quo}^{\dagger} \colon \mcS_{\theta} \to S_{\theta}$ 
	is reduction modulo $(U_p-1)$.
\end{proposition}

\begin{proof}
	We define $z_{\quo}^{\dagger}$ as the composition
	$$
		\mcS_{\theta} \xrightarrow{\sim} \mfS_{\theta} \xrightarrow{\alpha} \mfS_{\theta}^{\star},
	$$
	which is in particular a map of $\La$-modules as it lands in the kernel of multiplication by $X$ in $\mfS_{\theta}^{\star}$.
	On the other hand, the composition 
	$$
		\La \cozp \mcS_{\theta} \xrightarrow{z_{\quo}} H^1_{\Iw}(\Q_{p,\infty},\mcT_{\quo}(1))
		\to H^1(\qp,\mcT_{\quo}^{\dagger}(1)) \xrightarrow{\Cole^{\dagger}} \mfS_{\theta}^{\star}
	$$
	is induced by multiplication by $(U_p-1)\alpha$ by definition of $\alpha$ and $\Cole^{\dagger}$.
	The composition of $z_{\quo}^{\dagger}$ with $\psi$ of Corollary \ref{loccohdagS} is
	reduction modulo $(U_p-1)$, which gives the final statement.
\end{proof}

\section{Global study} \label{global_study}

This section has two primary goals: first, to recall what goes into and refine 
the main result of \cite{fk-pf} that $\xi' \Upsilon \circ \varpi = \xi'$ as an endomorphism of $P \otimes_{\zp} \qp$.
After some needed results on the cohomology of $\mcT_{\theta}(1)$ and its subquotients of interest, we recall various modular and Manin-type symbols and then Beilinson-Kato elements. We then define the zeta map $z$ and its ground-level analogue $z^{\sharp}$ which were alluded 
to in Section \ref{local_study}; these carry compatible systems of Manin symbols to compatible systems of cup products of Beilinson-Kato elements. We remove denominators present in the construction of $z$ and $z^{\sharp}$ in \cite[Section 3]{fk-pf} by employing the results of \cite{fks2}. In particular, we recall without proof the $p$-adic regulator computation 
$\Cole \circ z = \alpha$ in Theorem \ref{zetamap} from \cite[4.3.6]{fk-pf} and its consequence that 
$\Cole^{\flat} \circ z^{\sharp} = \xi'$. We put everything together to obtain that $\xi' \Upsilon \circ \varpi = \xi'$ 
as an endomorphism of $P$ itself in Theorem \ref{FKident}, filling out the commutative diagram of Theorem \ref{FKgalcoh}
in \eqref{FK_diag}.

We then turn to the study of the cohomology of the intermediate quotient $\mcT_{\theta}^{\dagger}(1)$ and, again, its subquotients of interest.
Our second goal is to prove the equivalence of Conjecture \ref{sconj} that $\Upsilon \circ \varpi = 1$
with the existence of a reduced refined zeta map $\bar{z}^{\dagger}$ compatible with reductions
$z^{\sharp}$ and $z_{\quo}^{\dagger}$ modulo $I$. This is done in Theorem \ref{equivform}. In fact, we provide 
a candidate for $\bar{z}^{\dagger}$ which is compatible with $z^{\sharp}$, and which is compatible with $z_{\quo}^{\dagger}$ 
if and only if Conjecture \ref{sconj} holds. The diagram \eqref{refined_diag} for intermediate cohomology 
refines \eqref{FK_diag} and encapsulates many of our results.

\subsection{Global cohomology} \label{global_coh}

We first consider torsion in global cohomology groups.  As we are working only with the needed
eigenspace of the Eisenstein part of cohomology, we can obtain finer results than \cite[Section 3]{fk-pf} in
our case of interest.

\begin{lemma} \label{torsfree}
	We have two exact sequences
	\begin{eqnarray*}
		&0 \to H^1_{\Iw}(\mc{O}_{\infty},\mcT_{\theta}(1)) \to H^1_{\Iw}(\mc{O}_{\infty},\tmcT_{\theta}(1))
		\to H^1_{\Iw}(\mc{O}_{\infty},\tmcT_{\theta}/\mcT_{\theta}(1))&\\
		&0 \to H^1(\mc{O},\mcT_{\theta}(1)) \to H^1(\mc{O},\tmcT_{\theta}(1))
		\to H^1(\mc{O},\tmcT_{\theta}/\mcT_{\theta}(1))&
	\end{eqnarray*}
	of $\La \cozp \mf{H}_{\theta}$-modules.  In the first, the terms have no nonzero $\La \cozp \La_{\theta}$-torsion, 
	and in the second, they have no $\La_{\theta}$-torsion.
\end{lemma}

\begin{proof}
	The first sequence is automatically exact, as zeroth Iwasawa cohomology groups are trivial.
	Note that $\tmcT_{\theta}/\mcT_{\theta}$ has trivial $G_{\Q(\mu_N)}$-action by 
	\cite[3.2.4]{fk-pf}. (Alternatively, one can see this by observing that the action factors through the Galois group
	of the totally ramified at $p$ extension $\Q(\mu_{Np^{\infty}})/\Q(\mu_N)$,
	since all cusps of $Y_1(Np^r)$ are defined over $\Q(\mu_{Np^r})$, and then that the $G_{\qp}$-action on 
	$\tmcT_{\theta}/\mcT_{\theta} \cong \tmcT_{\quo}/\mcT_{\quo}$ is unramified.)
	So, the second sequence is exact as $H^0(\mc{O},\tmcT_{\theta}/\mcT_{\theta}(1)) = 0$.

	We can filter any $\mf{h}_{\theta}[G_{\Q,S}]$-quotient 
	$M$ of $\mcT_{\theta}$ by the powers of $I$, and we clearly have
	$H^0(\mc{O},M(1)) = 0$
	if $H^0(\mc{O},I^kM/I^{k+1}M(1)) = 0$ for all $k \ge 0$.
	Let $\mu \in \La_{\theta}$ be nonzero, and set $M = \mcT_{\theta}/\mu\mcT_{\theta}$.
	As $\mcT_{\theta}$ is $\La_{\theta}$-free, we have an exact sequence
	$$
		0 \to \mcT_{\theta} \xrightarrow{\mu} \mcT_{\theta} \to M \to 0,
	$$
	so $H^0(\mc{O},M(1))$ surjects onto (in fact, is isomorphic to) 
	the $\mu$-torsion in $H^1(\mc{O},\mcT_{\theta}(1))$.	
	Set $T_k = I^k\mcT_{\theta}/I^{k+1}\mcT_{\theta}$.
	Let $P_k$ denote the $\mf{h}_{\theta}[G_{\Q,S}]$-module that is the image of the multiplication map
	$I^k \otimes_{\mf{h}_{\theta}} P \to T_k$.
	The $G_{\Q}$-action on $P_k$ is then trivial, and in that the quotient 
	$Q_k = T_k/P_k$ is also a quotient of $I^k \otimes_{\mf{h}_{\theta}} Q$,
	the $G_{\Q}$-action on $Q_k$ factors through $\Z_{p,N}^{\times}$ with $\Delta$ acting as
	$\omega\theta^{-1}$.  As a nonzero $\mf{h}_{\theta}[G_{\Q}]$-quotient of $T_k(1)$, it then follows
	(since $\theta \neq \omega^2$ by Hypothesis \ref{thetahyp}c) that 
	$I^kM/I^{k+1}M(1)$ has no nonzero $G_{\Q}$-fixed elements.  Thus, $H^1(\mc{O},\mcT_{\theta}(1))$ has
	no $\mu$-torsion.  Replacing $M$ with $\Lai \cozp M$ and $\mu$ with a nonzero element
	$\lambda \in \La \cozp \La_{\theta}$, a similar argument applies to show that 
	$H^1(\mc{O},\Lai \cozp \mcT_{\theta}(1))$
	has no nonzero $\lambda$-torsion (as the $\Delta$-action on $\Lai$
	is trivial).
	
	It remains to deal with the $\tmcT_{\theta}/\mcT_{\theta}$-terms.
	We first claim that the restriction map 
	$$
		H^1_{\Iw}(\mc{O}_{\infty},\tmcT_{\theta}/\mcT_{\theta}(1)) \to H^1_{\Iw}(\Q_{p,\infty},\tmcT_{\theta}/\mcT_{\theta}(1))
	$$
	is injective. Since $p \nmid \varphi(N)$, it suffices to show this after adjoining $\mu_N$ on both sides. Since
	the $G_{\Q(\mu_N),S}$-action on $\tmcT_{\theta}/\mcT_{\theta}$ is trivial, this then follows from the known weak Leopoldt
	conjecture for the cyclotomic $\zp$-extension of $\Q(\mu_N)$: see \cite[\S 10.3]{nsw}.
	Following this restriction map with the injective Coleman map (Lemma \ref{Colimage}), we obtain an injection
	$$
		H^1_{\Iw}(\mc{O}_{\infty},\tmcT_{\theta}/\mcT_{\theta}(1)) \hookrightarrow X^{-1}\La \cozp 
		\mf{M}_{\theta}/\mf{S}_{\theta} \cong \La \cozp \La_{\theta},
	$$   
	where the latter isomorphism uses \cite[Proposition 3.1.2]{ohta-cong} and Hypothesis \ref{thetahyp}d
	(though in the case said hypothesis fails, we have $(\La \cozp \La_{\theta})^2$ instead, and
	the result is the same).
	Clearly the latter module is $\La \cozp \La_{\theta}$-torsion free.
	
	Next, consider the prime-to-$p$ degree inflation map
	$$
		H^1(\mc{O},\tmcT_{\theta}/\mcT_{\theta}(1)) \xrightarrow{\sim} 
		H^1(\mc{O}[\mu_N],\tmcT_{\theta}/\mcT_{\theta}(1))^{\Gal(\Q(\mu_N)/\Q)}.
	$$
	By Kummer theory, the right-hand side is the $\Gal(\Q(\mu_N)/\Q)$-invariant group of the direct
	sum $\bigoplus_{v \mid Np} \tmcT_{\theta}/\mcT_{\theta}$ over places over $Np$ in $\Q(\mu_N)$
	(this direct sum being the completed tensor product of the $p$-torsion-free group 
	$\tmcT_{\theta}/\mcT_{\theta}$ and the
	$p$-completion of $\mc{O}[\mu_N]^{\times}/\Z[\mu_N]^{\times}$).
	That is, we have the isomorphism in
	$$
		H^1(\mc{O},\tmcT_{\theta}/\mcT_{\theta}(1)) \cong \bigoplus_{\ell \mid Np} 
		(\tmcT_{\theta}/\mcT_{\theta})^{\Gal(\Q_{\ell}(\mu_N)/\Q_{\ell})} 
		\hookrightarrow \bigoplus_{\ell \mid Np} \La_{\theta}.
	$$
	The latter injection (which is actually an isomorphism) is a consequence of Theorem \ref{DTquo}, 
	\cite[Proposition 3.1.2]{ohta-cong}, and 
	Hypothesis \ref{thetahyp}d (the latter again being unnecessary for the result),
	and $\bigoplus_{\ell \mid Np} \La_{\theta}$ is plainly $\La_{\theta}$-torsion free.
\end{proof}

\begin{lemma} \label{nofixedelts}
	Multiplication by $1-U_p$ is an injective endomorphism of $H^1(\mc{O},\mcT_{\theta}(1))$
	and of $H^1(\qp,\mcT_{\quo}(1))$.
\end{lemma}

\begin{proof}
	Since multiplication by $1-U_p$ is injective on $\mcT_{\theta}$, showing that $1-U_p$ is injective on
	$H^1(\mc{O},\mcT_{\theta}(1))$ amounts to showing that
	the Tate twist of $\mcT_{\theta}/(U_p-1)\mcT_{\theta}$ has trivial $G_{\Q}$-invariants.
	Note that the $G_{\qp}$-action $\mcT_{\quo}$ is unramified, and therefore, the action of 
	$G_{\qp^{\ur}}$ on $\mcT_{\quo}(1)$ is given by multiplication by the cyclotomic 
	character.  Therefore, we have 
	$$
		H^0(\qp,(\mcT_{\quo}/(U_p-1)\mcT_{\quo})(1)) = 0
	$$ 
	and the statement for $H^1(\qp,\mcT_{\quo}(1))$. 
	
	Next, note that any element of $(\mcT_{\theta}/(U_p-1)\mcT_{\theta})(1)$ with nontrivial image in 
	$(\mcT_{\quo}/(U_p-1)\mcT_{\quo})(1)$
	is not $G_{\Q}$-fixed in that its image is not $G_{\qp}$-fixed.
	Since $\mcT_{\theta} = \mcT_{\sub} \oplus \mcT_{\quo}$ as $\mf{h}_{\theta}$-modules, it therefore
	suffices to show that no nontrivial element of $(\mcT_{\sub}/(U_p-1)\mcT_{\sub})(1)$ is 
	fixed by $G_{\Q}$ inside the $\mf{h}_{\theta}[G_{\Q}]$-module $(\mcT_{\theta}/(U_p-1)\mcT_{\theta})(1)$.
	
	Now, $\mcT_{\sub}$ is isomorphic to $\mf{h}_{\theta}$ as an $\mf{h}_{\theta}$-module, and
	$\mcT_{\sub}/I\mcT_{\sub}$ is an $\mf{h}_{\theta}[G_{\Q}]$-submodule of $\mcT_{\theta}/I\mcT_{\theta}$ 
	mapping isomorphically to the quotient $Q$ of the latter module.
	For $\mf{m} = \Eis+(p,X)\mf{h}_{\theta}$, we then have
	$$
		(\mcT_{\sub}/\mf{m}\mcT_{\sub})(1) \cong (\La_{\theta}^{\iota}/(p,X))(2) \cong (R/pR)(2)
	$$
	as $G_{\Q}$-modules (where $G_{\Q}$ acts on $R$ through $\theta^{-1}$), 
	and this has no fixed elements since $\theta \not\equiv \omega^2 \bmod p$ by Hypothesis \ref{thetahyp}.  If 
	$x \in \mcT_{\sub}(1) - (U_p-1)\mcT_{\sub}(1)$ has $G_{\Q}$-fixed image in $(\mcT_{\theta}/(U_p-1)\mcT_{\theta})(1)$, 
	then it is also fixed in $\mf{h}_{\theta} x/(\mf{m}x+(U_p-1)\mcT_{\sub}(1))$.  This is 
	isomorphic to a nonzero quotient of $(\mcT_{\sub}/\mf{m}\mcT_{\sub})(1)$ under multiplication by $x$, 
	so it has no fixed elements, which contradicts $x \neq 0$.
\end{proof}

\begin{lemma} \label{cohomatell}
	For primes $\ell \mid N$, the groups 
	$H^1(\Q_{\ell},\mcT_{\theta}(1))$ and
	$H^1(\Q_{\ell},\Lai \cozp \mcT_{\theta}(1))$ are trivial.
\end{lemma}
     
\begin{proof}
	 We remark that 
	$H^1(\Q_{\ell}, \Lai \cozp \mcT_{\theta}(1))$ 
	is isomorphic to a finite product of inverse limits
	$\varprojlim_r H^1(\Q_{\ell,r},\mcT_{\theta}(1))$ as an $\mf{h}_{\theta}$-module, 
	where $\Q_{\ell,r}$ denotes the unramified degree $p^r$ extension of $\Q_{\ell}$.
	It suffices to show that the terms in this inverse limit vanish.
	Fukaya and Kato showed that the inflation map 
	$$
		H^1(\F_{\ell},H^0(\Q_{\ell}^{\ur},\mcT_{\theta}(1))) \to H^1(\Q_{\ell},\mcT_{\theta}(1))
	$$
	is an isomorphism \cite[9.5.2]{fk-pf}, and their argument works with 
	$\Q_{\ell,r}$ and the field $\F_{\ell,r}$ of order $\ell^{p^r}$ replacing $\Q_{\ell}$ and $\F_{\ell}$, respectively. 
	
	Let $\mf{T} = H^0(\Q_{\ell}^{\ur},\mcT_{\theta})$.
	It remains only to show that the groups 
	$H^1(\F_{\ell,r},\mf{T}(1))$ are trivial.  The operator $U_{\ell}$ acts on $\mf{T}$ as an arithmetic 
	Frobenius (see \cite[Theorem 4.2.4]{hida}, \cite[Lemma 4.2.2]{cdt}, \cite[Proposition 3.2.2]{mw}
	for the quotients of $\mcT_{\theta}$ attached to newforms, from which this statement follows by Hida theory), 
	and its eigenvalues are congruent to $1$ modulo the maximal
	ideal of $\mf{h}_{\theta}$ since 
	$U_{\ell}-1 \in I$.
	Since this Frobenius generates the
	procyclic group $G_{\F_{\ell}}$, we have 
	$$
		H^1(\F_{\ell,r},\mf{T}(1)) \cong \mf{T}/((\ell U_{\ell})^{p^r} - 1)\mf{T},
	$$
	and the latter quotient is zero since $\ell^{p^r} \not\equiv 1 \bmod p$.
\end{proof}

\begin{lemma} \label{globcompconn}
	Under the identification of $H^2_c(\mc{O},P(1))$ with $P$ of Lemma \ref{connectP},
	the canonical map $H^1(\qp,P(1)) \to H^2_c(\mc{O},P(1))$ agrees with $-\Cole_P^{\flat}$.
\end{lemma}

\begin{proof}
	Consider the square 
	$$
		\SelectTips{cm}{} \xymatrix{
			H^1(\qp,P(1)) \ar[r]^{\partial_P} \ar[d] & H^2(\qp,P(1)) \ar[d] \\
			H^2_c(\mc{O},P(1)) \ar[r]^{\partial_P} & H^3_c(\mc{O},P(1)), \\
		}
	$$
	which anticommutes by Lemma \ref{loccompconn}.
	The identifications of $H^2(\qp,P(1))$ and $H^3_c(\mc{O},P(1))$
	with $P$ by invariant maps agree, and by Proposition \ref{Psquare} for $\alpha=1$, the latter identification agrees with the identification 
	of $H^2_c(\mc{O},P(1))$ with $P$ via $\partial_P$. On the other hand, the composite map $\inv \circ \partial_P \colon 
	H^1(\qp,P(1)) \to P$ equals $\Cole_P^{\flat}$ by Lemma \ref{Colconn}. The result follows.
\end{proof}

\begin{lemma} \label{splitseqs}
	The exact sequences
	$$
		0 \to H^1(\mc{O},P(1)) \to H^1(\mc{O},T(1)) \to H^1(\mc{O},Q(1)) \to 0
	$$
	and
	$$
		0 \to H^1(\mc{O},P(1)) \to \bigoplus_{\ell \mid Np} H^1(\Q_{\ell},P(1)) \to H^2_c(\mc{O},P(1)) \to 0
	$$
	are canonically split, compatibly with the map from the former sequence to the latter.
\end{lemma}

\begin{proof}
	Since the $G_{\Q}$-action on $P$ is trivial, we have isomorphisms
	$$
		H^1(\Q_{\ell},P(1)) \cong \varprojlim_r\, \Q_{\ell}^{\times}/\Q_{\ell}^{\times p^r} \cozp P
	$$
	for every $\ell$ dividing $Np$.  The $\ell$-adic valuation then induces a map from this group to $P$
	which is an isomorphism if $\ell \neq p$ and otherwise induces a splitting of
	\begin{equation*} \label{pseq}
		0 \to H^1(\Z[\tfrac{1}{p}],P(1)) \to H^1(\qp,P(1)) \to H^2_c(\Z[\tfrac{1}{p}],P(1)) \to 0.
	\end{equation*}
	Noting that $H^2_c(\Z[\tfrac{1}{p}],P(1))$ and $H^2_c(\mc{O},P(1))$ are canonically isomorphic,
	the sum of the $\ell$-adic valuation maps then gives the desired splitting of the injection in the 
	second sequence.  The splitting of the injection in the first sequence is given by the 
	composition
	$$
		H^1(\mc{O},T(1)) \to \bigoplus_{\ell \mid Np} H^1(\Q_{\ell},P(1)) \to H^1(\mc{O},P(1)),
	$$
	where the first map is induced by the local splittings of Proposition \ref{exseqredunr}, and the
	second map is the splitting of the second sequence. The final statement follows.
\end{proof}

We can now slightly refine the left-hand square in Theorem \ref{FKgalcoh}.

\begin{corollary} \label{commsquare}
	The square
	$$
		\SelectTips{cm}{} \xymatrix{
		H^1(\mc{O},T(1)) \ar[r] \ar[d] & H^1(\mc{O},Q(1)) \ar[d]^{-\Theta}  \\
		H^1(\qp,P(1)) \ar[r] & H^2_c(\mc{O},P(1)) 
		}
	$$
	is commutative.
\end{corollary}

\begin{proof}
	Both compositions are clearly trivial on elements of $H^1(\mc{O},P(1))$ inside $H^1(\mc{O},T(1))$.
	On the other hand, Lemma \ref{splitseqs} tells us that the composition
	$$
		H^1(\mc{O},Q(1)) \to H^1(\mc{O},T(1)) \to \bigoplus_{\ell \mid Np} H^1(\Q_{\ell},P(1))
	$$
	takes image in the image of $H^2_c(\mc{O},P(1))$ (using the splittings induced by said lemma).  
	This image
	is contained in $H^1(\qp,P(1))$ inside the direct sum as the kernel of the $p$-adic valuation map
	$H^1(\qp,P(1)) \to P$.  The result then follows from the commutativity of the left-hand square
	in Theorem \ref{FKgalcoh}.
\end{proof}

\subsection{Modular symbols} \label{modsym}

We very briefly review modular symbols and Manin symbols.

\begin{definition}
	For $r \ge 1$ and cusps $\alpha$ and $\beta$ on $X_1(Np^r)(\C)$, the {\em modular symbol} 
	$$
		\{ \alpha \to \beta \}_r \in H_1(X_1(Np^r)(\C),\{\mr{cusps}\},\zp)
	$$ 
	is the class of the geodesic from $\alpha$ to $\beta$ on $X_1(Np^r)(\C)$.
\end{definition}

\begin{definition}
	For $r \ge 1$ and $u, v \in \Z/Np^r\Z$ with $(u,v) = (1)$, the {\em Manin symbol} of level $Np^r$ attached
	to $(u,v)$ is defined as
	$$
		[u:v]_r = \left\{ \frac{-d}{bNp^r} \to \frac{-c}{aNp^r} \right\}_r = w_r \left( \begin{smallmatrix} a&b \\ c&d \end{smallmatrix} \right)
		\{0 \to \infty\}_r
	$$
	for $\left( \begin{smallmatrix} a&b \\ c&d \end{smallmatrix} \right)
	\in \mr{SL}_2(\Z)$ with $u = c \bmod Np^r$ and $v = d \bmod Np^r$, where $w_r$ is the Atkin-Lehner involution.
\end{definition}

\begin{remark}
	The Manin symbols of level $Np^r$ generate $H_1(X_1(Np^r)(\C),\{\mr{cusps}\},\zp)$, and the relations
	$$
		[u:v]_r = -[-v:u]_r = [u:u+v]_r + [u+v:v]_r
	$$
	provide a presentation of said $\zp$-module.
\end{remark}

\begin{definition}
	For $r \ge 1$, $u, v \in \Z/Np^r\Z$ with $(u,v) = (1)$, and integers $c, d > 1$ with $(c,6Np) = (d,6Np) = 1$,
	we have the {\em $(c,d)$-symbol} 
	$$
		{}_{c,d}[u:v]_r = c^2d^2[u:v]_r - c^2[u:dv]_r - d^2[cu:v]_r + [cu:dv]_r.
	$$
\end{definition}

The quotient maps $X_1(Np^{r+1}) \to X_1(Np^r)$ take $[pu:v]_{r+1}$ to $[u:v]_r$ for $u,v \in \Z/Np^{r+1}\Z$ with
$(u,v) = (1)$, and similarly for the $(c,d)$-symbols.  Let $[u:v]_{r,\theta}^+$ denote the image
of $[u:v]_r$ in the $\theta$-eigenspace of the Eisenstein component of $H_1(X_1(Np^r)(\C),\cusps,\zp)^+$.

\begin{definition}
	For integers $u$ and $v$ with $p \nmid v$ and $(u,v,N) = 1$, let 
	$$
		(u:v)_{\theta} = ([p^{r-1} u : v]_{r,\theta}^+)_r \in \mc{M}_{\theta}.
	$$
\end{definition}
	
Note that $(u:v)_{\theta}$ depends upon $u$ only modulo $Np$. 
By \cite[3.2.5]{fk-pf}, the elements $(u:v)_{\theta}$ generate $\mc{M}_{\theta}$, and under Hypothesis \ref{thetahyp},
the group $\mc{S}_{\theta}$ is generated by the symbols $(u:v)_{\theta}$ with $u \not\equiv 0 \bmod Np$ by \cite[6.2.6]{fk-pf}.

\begin{definition}
	Let 
	$$
		\Pi = \{ (c,d,u,v) \in \Z^4 \mid c,d > 1, (c,6Np) = (d,6Np) = (u,v,N) = (v,p) = 1 \},
	$$ 
	and let 
	$$
		\Pi_0 = \{ (c,d,u,v) \in \Pi \mid u \not\equiv 0 \bmod Np \}.
	$$
\end{definition}

We define symbols attached to elements of these sets.
	
\begin{definition}
	For $(c,d,u,v) \in \Pi$, let
	$$
		{}_{c,d}(u:v)_{\theta} = c^2d^2(u:v)_{\theta} - d^2(cu:v)_{\theta} - c^2(u:dv)_{\theta} + (cu:dv)_{\theta} 
		\in \mc{M}_{\theta},
	$$
	and define ${}_{c,d}(u:v]_{\theta} \in \La \cozp \mc{M}_{\theta}$ by
	$$
		{}_{c,d}(u:v]_{\theta} 
		= c^2d^2 \otimes (u:v)_{\theta} - d^2\kappa(c) \otimes (cu:v)_{\theta} - c^2\kappa(d) \otimes 
		(u:dv)_{\theta} + \kappa(cd) \otimes (cu:dv)_{\theta},
	$$
	where $\kappa \colon \zp^{\times} \to \La$ sends a unit to the group element of its projection to $1+p\zp$.  
\end{definition}

\subsection{Zeta elements} \label{zetaelts}

We first very briefly recall the Beilinson-Kato elements (or zeta elements) of \cite[Section 2]{fk-pf}.  We then, in the form we shall require, slightly refine the resulting maps of Fukaya and Kato \cite[Section 3]{fk-pf} and describe the properties of them that we need.

The following definition is from \cite[2.4.2]{fk-pf}.

\begin{definition}
	For $r, s \ge 0$ and $u, v \in \Z$ with $(u,v,Np) = (1)$, 
	and supposing that $u, v \not\equiv 0 \bmod Np^r$ if $s = 0$,
	we define the zeta element ${}_{c,d}z_{r,s}(u:v)$ to be the image
	under the norm and Hochschild-Serre maps
	$$
		H^2(Y(p^s,Np^{r+s})_{/\Z[\frac{1}{Np}]},\zp(2)) \to H^2(Y_1(Np^r)_{/\mc{O}_s},\zp(2))
		\to H^1(\mc{O}_s,H^1(Y_1(Np^r)_{/\qbar},\zp(2)))
	$$
	of the cup product ${}_c g_{\frac{a}{p^s},\frac{c}{Np^{r+s}}} \cup {}_d g_{\frac{b}{p^s},\frac{d}{Np^{r+s}}}$ of Siegel 
	units on $Y(p^s,Np^{r+s})_{/\Z[\frac{1}{Np}]}$, where $\left( \begin{smallmatrix} a&b\\c&d \end{smallmatrix} \right) 
	\in \mr{SL}_2(\Z)$ with $u = c \bmod Np^r$ and $v = d \bmod Np^r$.
\end{definition}

\begin{remark} \label{coreszeta}
	As a consequence of \cite[2.4.4, 3.1.9]{fk-pf}, the elements ${}_{c,d}z_{r,s}(u:v)$ are for $r, s \ge 1$ 
	compatible with the maps induced by quotients of modular curves and corestriction maps for the ring extensions.
	Moreover, the corestriction map
	$$
		H^1(\mc{O}_s,H^1(Y_1(Np^r)_{/\qbar},\zp(2))) \to H^1(\mc{O},H^1(Y_1(Np^r)_{/\qbar},\zp(2)))
	$$
	takes ${}_{c,d}z_{r,s}(u:v)$ to $(1-U_p){}_{c,d}z_{r,0}(u:v)$ if $(c,d,u,v) \in \Pi_0$.
\end{remark}

Let us use ${}_{c,d}z_{r,s}(u:v)_{\theta}$ to denote the projection of ${}_{c,d}z_{r,s}(u:v)$ to the Eisenstein component
for $\theta$.

\begin{definition}
	For $(c,d,u,v) \in \Pi$, we set
	$$
		{}_{c,d} z(u:v)_{\theta} = ({}_{c,d}z_{r,s}(u:v)_{\theta})_{r,s \ge 1} \in H^1_{\Iw}(\mc{O}_{\infty},\tmcT_{\theta}(1)) 
	$$
	and for $(c,d,u,v) \in \Pi_0$, we set
	$$
		{}_{c,d} z^{\sharp}(u:v)_{\theta} = ({}_{c,d}z_{r,0}(u:v)_{\theta})_{r \ge 1} \in H^1(\mc{O},\tmcT_{\theta}(1)).
	$$ 
\end{definition}

By Remark \ref{coreszeta}, the corestriction of ${}_{c,d} z(u:v)_{\theta}$ to $H^1(\mc{O},\tmcT_{\theta}(1))$ is 
$(1-U_p)z^{\sharp}(u:v)_{\theta}$ for $(c,d,u,v) \in \Pi_0$

\begin{definition}
	Let $\mc{Z}$ denote the unique element of $X^{-1}\zp\ps{X}$ such that
	$$
		\mc{Z}(t^s-1) = \zeta_p(1-s)
	$$
	for all $s \in \zp$, where $\zeta_p$ denotes the $p$-adic Riemann zeta function.
\end{definition}
 
The following result, constructing a \emph{zeta map}, is a refinement of a result of Fukaya and Kato \cite[3.3.3]{fk-pf}.  It is in essence a consequence of \cite[Theorem 3.15]{fks2}.

\begin{theorem} \label{zetamap}
	There exists a $\La \cozp \mf{h}_{\theta}$-module homomorphism
	$$
		z \colon \La \cozp \mcS_{\theta} \to H^1_{\Iw}(\mc{O}_{\infty},\mcT_{\theta}(1))
	$$
	such that its composition with the injective map
	$$
		 H^1_{\Iw}(\mc{O}_{\infty},\mcT_{\theta}(1)) \to H^1_{\Iw}(\Q_{p,\infty},\mcT_{\theta}(1))
	$$
	equals the map $z_{\quo}$ of Definition \ref{zquo} for an element $\alpha \in \La \cozp \mf{h}_{\theta}$ with
	image $X\mc{Z}\otil{\xi}_1 \in \La \cozp (\mf{h}/I)_{\theta}$.
\end{theorem}

\begin{proof}
	In \cite[Theorem 3.15]{fks2}, we show the existence of a map 
	$$
		\tilde{z} \colon \La \cozp \mc{M}_{\theta} \to H^1_{\Iw}(\mc{O}_{\infty},\tmcT_{\theta}(1)) 
	$$
	of $\La \cozp \mf{H}_{\theta}$-modules
	 that satisfies
	 $$
		 \tilde{z}({}_{c,d}(u:v]_{\theta}) = -{}_{c,d} z(u:v)_{\theta} \otimes 1
	$$
	 for all $(c,d,u,v) \in \Pi$.  
	 
	By composition with restriction, we obtain a map
	$$
		\tilde{z}_{\quo} \colon \La \cozp \mc{M}_{\theta} \to H^1_{\Iw}(\Q_{p,\infty},\tmcT_{\quo}(1)).
	$$
	Via the fixed isomorphism $\mc{M}_{\theta} \cong \mf{M}_{\theta}$ and the canonical isomorphism
	$$
		\End_{\La \cozp \mf{H}_{\theta}}(\La \cozp \mc{M}_{\theta}) \xrightarrow{\sim} 
		\La \cozp \mf{H}_{\theta},
	$$
	the map $\Cole \circ \tilde{z}_{\quo}$ is given by multiplication by an element $\beta$ in 
	$X^{-1}\La \cozp \mf{H}_{\theta}$.  
	This map induces an endomorphism of
	$\La \cozp \mcS_{\theta}$ by \cite[4.4.3]{fk-pf},
	and the resulting map on the latter module is given by multiplication by the image $\alpha \in \La \cozp \mf{h}_{\theta}$ 
	of $\beta$ in $X^{-1}\La \cozp \mf{h}_{\theta}$.
	The element $\alpha$ reduces to $X\mc{Z}\otil{\xi}_1 \in \La \cozp (\mf{h}/I)_{\theta}$ by \cite[8.1.2(1)]{fk-pf}.
	(This rests on the $p$-adic regulator computation of \cite[4.3.6]{fk-pf}, which we do not reverify in this work.)
	Note that the congruence $X\mc{Z} \equiv 1 \bmod X$ implies that the image of $\alpha$ in 
	$(\mf{h}/I)_{\theta}^{\dagger}$ agrees with that of $\otil{\xi}_1$, so $\alpha$ has the
	property of Definition \ref{alpha}.  Thus, $\tilde{z}_{\quo}$ restricted to $\La \cozp \mcS_{\theta}$
	equals $z_{\quo}$ of Definition \ref{zquo} for this value of $\alpha$.
	
	The composition
	$$
		H^1_{\Iw}(\mc{O}_{\infty},\tmcT_{\theta}(1)) \to 
		H^1_{\Iw}(\Q_{p,\infty},\tmcT_{\theta}(1)) \to X^{-1}\La \cozp \mf{M}_{\theta}
	$$
	is injective by \cite[3.1.4 and 4.2.10]{fk-pf} and Lemma \ref{torsfree}.
	We must prove the claim that the restriction of $\tilde{z}$ to $\La \cozp \mcS_{\theta}$ takes values in 
	$H^1_{\Iw}(\mc{O}_{\infty},\mcT_{\theta}(1))$.  
	By what we have shown above, $\tilde{z}$ carries $\La \cozp \mcS_{\theta}$ to the kernel of 
	$$
		H^1_{\Iw}(\mc{O}_{\infty},\tmcT_{\theta}/\mcT_{\theta}(1)) \to H^1_{\Iw}(\Q_{p,\infty},\tmcT_{\quo}/\mcT_{\quo}(1))
		\hookrightarrow X^{-1}\La \cozp \mf{M}_{\theta}/\mf{S}_{\theta}.
	$$
	Since the surjection $\tmcT_{\theta}/\mcT_{\theta} \to \tmcT_{\quo}/\mcT_{\quo}$ is in fact an isomorphism
	(see Definition \ref{Tquosub}), this kernel is trivial, as shown in the proof of Lemma \ref{torsfree}.
	By the exactness of the first sequence in Lemma \ref{torsfree}, we have the claim.
\end{proof}

From now on, we take $\alpha$ to be as given in Theorem \ref{zetamap}.  
We prove the following slight refinement of \cite[3.3.9]{fk-pf} on a zeta map at the level of $\Q$ 
as a consequence of \cite[Theorem 3.17]{fks2}.

\begin{theorem} \label{zsharp}
	There exists a unique map
	$$
		z^{\sharp} \colon \mcS_{\theta} \to H^1(\mc{O},\mcT_{\theta}(1))
	$$
	of $\mf{h}_{\theta}$-modules with the property that for the map 
	$\cor \colon H^1_{\Iw}(\mc{O}_{\infty},\mcT_{\theta}(1)) \to H^1(\mc{O},\mcT_{\theta}(1))$
	induced by corestriction, we have
	$$
		\cor \circ z = (1-U_p) z^{\sharp} \circ \ev_0.
	$$ 
	The composition of $z^{\sharp}$ with
	$$
		H^1(\mc{O},\mcT_{\theta}(1)) \to H^1(\qp,\mcT_{\theta}(1))
	$$
	equals the map $z^{\sharp}_{\quo}$ of Proposition \ref{zquosharp} for $\alpha$ as in 
	Theorem \ref{zetamap}.
\end{theorem}

\begin{proof}
	In \cite[Theorem 3.17]{fks2} (noting \cite[3.3.14]{fk-pf}), we prove (using Lemma \ref{torsfree} of this paper) the existence of an
	$\mf{H}_{\theta}$-module homomorphism
	$$
		z^{\sharp} \colon \mcS_{\theta} \to H^1(\mc{O},\mcT_{\theta}(1))
	$$	
	with the property that 
	\begin{equation} \label{zsharpManin}
		z^{\sharp}({}_{c,d}(u:v)_{\theta}) = -{}_{c,d}z^{\sharp}(u:v)_{\theta}
	\end{equation}
	for all $(c,d,u,v) \in \Pi_0$. 

	The comparison with $z$ is \cite[3.3.9(ii)]{fk-pf}, the uniqueness being Lemma \ref{nofixedelts}.
	The comparison with
	$z_{\quo}^{\sharp}$ follows from Proposition \ref{zquosharp}, the comparison with $z$, and Theorem \ref{zetamap}.
\end{proof}

Fukaya and Kato prove the following in \cite[5.3.10-11 and 9.2.1]{fk-pf}. We sketch their proof primarily to make clear how to obtain 
the sign in its comparison. That is, there are two sign differences from their proof which effectively cancel each other, and the sign of the second map in the composition in its statement is the opposite of that of \cite[6.3.9]{fk-pf}.

\begin{theorem}[Fukaya-Kato] \label{zsharpvarpi}
	The composition of $z^{\sharp}$ with the maps 
	$$
		H^1(\mc{O},\mcT_{\theta}(1)) \to H^1(\mc{O},Q(1)) 
		\xrightarrow{\sim} Y
	$$	
	given by Proposition \ref{Qiso} and Lemma \ref{Qcoh} equals $-\varpi$.
\end{theorem}

\begin{proof}
	The results of Sections 5.2 and 5.3 of \cite{fk-pf} yield a map
	$$
		\infty^* \colon H^1(\mc{O},\mcT_{\theta}(1)) \xrightarrow{\sim} 
		\invlim{r} H^2(Y_1(Np^r)_{/\mc{O}},\zp(2))_{\mf{m},\theta} \to Y,
	$$
	given by composing the inverse of the Hochschild-Serre map with a specialization-at-$\infty$ map.
	It follows from \cite[5.2.9]{fk-pf} that this composition takes ${}_{c,d}z^{\sharp}(u:v)_{\theta}$ 
	to $\varpi({}_{c,d}(u:v)_{\theta})$. (Note that the sign here is opposite to that in the proof of \cite[5.3.11]{fk-pf} in the current version,
	as an unexplained sign appears in the proof of Claim 1 therein.)
	By \eqref{zsharpManin} and \cite[Proposition 3.1.6]{fks2}, the composition $\infty^* \circ z^{\sharp}$ is then $-\varpi$.
	
	By \cite[9.2.5]{fk-pf}, the map $\infty^*$ agrees with the connecting map
	$$
		H^1(\mc{O},\mcT_{\theta}(1)) \to H^2(\mc{O},\Lai_{\theta}(2))
	$$
	in the long exact sequence associated to the Tate twist of the short exact sequence 
	$$
		0 \to \Lai_{\theta}(1) \to \tmcT_{c,\theta} \to \mcT_{\theta} \to 0,
	$$
	where $1 \in \Lai_{\theta}(1)$ corresponds to the cusp at $\infty$.
	It then suffices to show that there is a commutative diagram of continuous $G_{\Q}$-module homomorphisms
	\begin{equation} \label{cptcohdiag}
		\SelectTips{cm}{} \xymatrix{
			0 \ar[r] & \Lai_{\theta}(1) \ar@{=}[d] \ar[r] & \tmcT_{c,\theta} \ar[r] \ar[d],& \mcT_{\theta} \ar[r] \ar[d] & 0 \\
			0 \ar[r] & \Lai_{\theta}(1) \ar[r]^{\xi} & \Lai_{\theta}(1) \ar[r] & Q \ar[r] & 0, 
		}
	\end{equation}
	where the right-hand vertical map is the surjection of Proposition \ref{Qiso} given by
	$v \mapsto \langle \xi e_{\infty}, v \rangle$ modulo $\xi$.
	
	Consider the element $g \in \tmcT_{\theta}$
	given by the $\theta$-projection of the compatible system of Siegel units $(g_{0,\frac{1}{Np^r}})_r$ and Kummer theory. 
	The boundary map $\tmcT_{\theta} \to \La_{\theta}$ at $0$-cusps of \cite[6.2.5]{fk-pf} carries $g$ to
	$-\xi$, the equivariant sum of its orders of vanishing at the $0$-cusps; 
	for this, see \cite[I.6 (3) and IV.2 (1)]{mw}, as well as the equality of the first and last terms in \eqref{poincareident}. 
	(Note that we obtain here the opposite sign to \cite[6.2.13]{fk-pf}, which leads us to replace $g$ with $g^{-1}$
	in the pairing map below, without further effect.)
	Much as in \cite[9.2.3]{fk-pf}, the desired commutative diagram \eqref{cptcohdiag} is given by taking the 
	center vertical map to be the pairing map $w \mapsto \langle g^{-1},w \rangle$, which is $G_{\Q}$-equivariant as $g$ is $G_{\Q}$-fixed.
	For the right-hand square, the commutativity is just as in \cite[9.2.3]{fk-pf}, but note that we have changed the side on which we are  
	pairing with $\xi e_{\infty}$ in Proposition \ref{Qiso}.
	The left-hand square commutes (instead of anticommutes, as suggested by 
	the proof of \cite[9.2.3]{fk-pf}) as a consequence of the first equality in \eqref{poincareident}, though here
	we use Ohta's twisted Poincar\'e pairing \eqref{Ohta_pair2}, 
	which is a system of compatible pairings, each involving an Atkin-Lehner involution that 
	takes $\infty$-cusps to $0$-cusps (cf. \cite[1.6.2]{fk-pf}).
\end{proof}

The main result \cite[0.14]{fk-pf} in the work of Fukaya and Kato states that $\xi'\Upsilon \circ \varpi$ and $\xi'$
induce the same endomorphism of $P \otimes_{\zp} \qp$.  As $P$ is not known to be $p$-torsion free, this
is slightly weaker than equality as endomorphisms of $P$.  With the results of \cite{fks2} in hand, it is now a relatively 
straightforward matter to show that the stronger statement holds by following the argument of \cite{fk-pf}.

\begin{theorem}[Fukaya-Kato, Fukaya-Kato-S.] \label{FKident}
	One has $\xi'\Upsilon \circ \varpi = \xi' \in \End_{\mf{h}_{\theta}}(P)$.
\end{theorem}

\begin{proof}
	Theorems \ref{FKgalcoh}, \ref{zsharp} and \ref{zsharpvarpi}, Proposition \ref{zquosharp}, 
	Corollary \ref{commsquare}, and Lemma \ref{globcompconn}
	provide a commutative diagram
	\begin{equation} \label{FK_diag}
		\SelectTips{cm}{} \xymatrix@C=35pt@R=35pt{
		&& Y \ar[r]^{\xi'} & Y \ar[d]^{\wr}  \ar@/^3.7pc/[ddd]^-{\Upsilon} \\
		P \ar@/_3.7pc/[rrdd]_-{\xi'} \ar[r]_-{-\bar{z}^{\sharp}} \ar[rd]_{-\bar{z}^{\sharp}_{\quo}} \ar[rru]^-{\varpi} 
		& H^1(\mc{O},T(1)) \ar[r] \ar[d] & H^1(\mc{O},Q(1)) \ar[d]^{-\Theta} \ar[u]^{\wr} \ar[r]^{\partial_Q} 
		& H^2(\mc{O},Q(1)) \ar[d] \\
		& H^1(\qp,P(1)) \ar[r] \ar[rd]_{-\Cole^{\flat}_P} & H^2_c(\mc{O},P(1)) \ar[r]^{\partial_P} & 
		H^3_c(\mc{O},P(1)) 
		\ar[d]^{\wr}  \\
		&& P \ar[u]^{\wr} \ar[r]_{\id}^{\sim} & P.
		}
	\end{equation}
\end{proof}

\subsection{Refined global cohomology} \label{refinedglobcoh}

We prove analogues for intermediate cohomology of earlier results on global cohomology.
We begin with an extension of Lemma \ref{globcompconn}. Let us use $\overline{\Cole}_P^{\dagger} \colon H^1(\qp,P^{\dagger}(1)) \to P$ 
to denote the composition $\psi \circ \Cole_P^{\dagger}$.

\begin{lemma} \label{dagcompconn}
	Under the identification of $H^2_c(\mc{O},P^{\dagger}(1))$ with $P$ of Lemma \ref{connectP},
	the canonical map $H^1(\qp,P^{\dagger}(1)) \to H^2_c(\mc{O},P^{\dagger}(1))$ agrees with $-\overline{\Cole}_P^{\dagger}$.
\end{lemma}

\begin{proof}
	The anticommutativity of the square 
	$$
		\SelectTips{cm}{} \xymatrix{
			H^1(\qp,P^{\dagger}(1)) \ar[r]^{\partial_P^{\dagger}} \ar[d] & H^2(\qp,P(1)) \ar[d] \\
			H^2_c(\mc{O},P^{\dagger}(1)) \ar[r]^{\partial_P^{\dagger}} & H^3_c(\mc{O},P(1)) \\
		}
	$$
	is proven by the analogous argument to Lemma \ref{loccompconn}, and
	the identifications of $H^2(\qp,P(1))$ and $H^3_c(\mc{O},P(1))$
	with $P$ agree as before. By Proposition \ref{Psquare}, the latter identification agrees via $\partial_P^{\dagger}$ with the identification 
	of $H^2_c(\mc{O},P^{\dagger}(1))$ with $P$. Finally,
	$\overline{\Cole}^{\dagger}_P  = \inv \circ \partial_P^{\dagger}$
	by the commutativity of \eqref{Coldagflat}.
\end{proof}

Next, we have an analogue of Lemma \ref{splitseqs}.

\begin{proposition} \label{daggersplit}
	The exact sequences
	$$
		0 \to H^1(\mc{O},P^{\dagger}(1)) \to H^1(\mc{O},T^{\dagger}(1)) \to H^1(\mc{O},Q^{\dagger}(1)) \to 0
	$$
	and
	$$
		0 \to H^1(\mc{O},P^{\dagger}(1)) \to \bigoplus_{\ell \mid Np} 
		H^1(\Q_{\ell},P^{\dagger}(1)) \to H^2_c(\mc{O},P^{\dagger}(1)) \to 0
	$$
	are canonically split, compatibly with the map from the former sequence to the latter.
	The splitting of the surjection in the latter sequence takes image in $H^1(\qp,P^{\dagger}(1))$ inside the direct sum.
	Moreover, the splittings are compatible with the maps of these sequences to (via the quotient map $P^{\dagger} \to P$) and 
	from (via $\alpha \colon P \to P^{\dagger}$) the corresponding split sequences of Lemma \ref{splitseqs}.
\end{proposition}

\begin{proof}
	Since $\Fr_p$ acts trivially on $P$, the exact sequence
	\begin{equation} \label{firstseq}
		0 \to X^{-1}(\La \cozp P)/\alpha(\La \cozp P) \to H^1(\qp,P^{\dagger}(1)) \xrightarrow{\overline{\Cole}^{\dagger}_P} P \to 0
	\end{equation}
	of Theorem \ref{loccohdag} is canonically split as in Lemma \ref{splitCstar}, 
	with the first term identified with $H^1_{\Iw}(\Q_{p,\infty},P(1))/X\alpha$ and
	the third identified with $H^2_{\Iw}(\Q_{p,\infty},P(1)) \xrightarrow{\sim} H^2(\qp,P^{\dagger}(1))$.  
	
	The restriction map 
	$$
		H^1_{\Iw}(\Z_{\infty}[\tfrac{1}{p}],P(1)) \to H^1_{\Iw}(\Q_{p,\infty},P(1))
	$$
	is an isomorphism,
	since the $p$-completion of $\qp^{\times}$ is generated by $p$ and $1+p$, with $p$ generating the universal norms
	from $\Q_{p,\infty}$. Moreover, $H^2_{\Iw}(\Z_{\infty}[\frac{1}{p}],P(1)) = 0$ since $\Q$ has trivial class group and 
	$\Q_{\infty}$ has a unique prime over $p$. Thus, we have
	$$
		H^1_{\Iw}(\Z_{\infty}[\tfrac{1}{p}],P(1))/X\alpha \cong H^1(\Z[\tfrac{1}{p}],P^{\dagger}(1)).
	$$
	Therefore, the first term in \eqref{firstseq} can be identified with and replaced by $H^1(\Z[\frac{1}{p}],P^{\dagger}(1))$. 
	The third term $P$ can then
	be replaced by $H^2_c(\mc{O},P^{\dagger}(1))$ via the negative of our fixed identification from Lemma \ref{Pcoh}. 
	(Here, we use the negative so the resulting sequence is that of Poitou-Tate, as follows from Lemma
	\ref{dagcompconn}.)
	That is, the exact sequence 
	\begin{equation} \label{secondseq}
		0 \to H^1(\Z[\tfrac{1}{p}],P^{\dagger}(1)) \to H^1(\qp,P^{\dagger}(1)) \to H^2_c(\mc{O},P^{\dagger}(1)) \to 0	
	\end{equation}
	canonically split.
	
	We take the splitting of the surjection in the second exact sequence in the statement
	to be given by the composition 
	$$
		H^2_c(\mc{O},P^{\dagger}(1)) \to H^1(\qp,P^{\dagger}(1)) \to \bigoplus_{\ell \mid Np} H^1(\Q_{\ell},P^{\dagger}(1)),
	$$
	of the splitting of the surjection in \eqref{secondseq} and the inclusion of the summand for $\ell = p$.
	The composition
	$$
		H^1(\mc{O},T^{\dagger}(1)) \to \bigoplus_{\ell \mid Np} H^1(\Q_{\ell},P^{\dagger}(1)) \to H^1(\mc{O},P^{\dagger}(1))
	$$
	of the map given by the local splittings of Proposition \ref{exseqredunr}
	and the splitting of the injection in the second exact sequence of the proposition 
	then gives a canonical splitting of the first exact sequence of the proposition, and it is compatible with the map between the two.

	The splitting of $\overline{\Cole}_P^{\dagger} \colon 
	H^1(\qp,P^{\dagger}(1)) \to P$ is by definition 
	given by composing the pushout map $P = D(P) \to \mfC^{\star}(P)$ with $(\Cole^{\dagger})^{-1}$.
	By its construction in Theorem \ref{loccohdag}, it also given by the composition of the canonical splitting 
	$(\Cole_P^{\flat})^{-1} \colon P \to H^1(\qp,P(1))$ of $\Cole_P^{\flat}$ from Definition \ref{Col_flat_inv}
	(determined by the $p$-adic valuation as in Lemma \ref{splitseqs})
	with the map induced by $\alpha \colon P \to P^{\dagger}$. This gives the commutativity of
	$$
		\SelectTips{cm}{} \xymatrix{
		H^1(\qp,P^{\dagger}(1)) \ar[r] \ar[d] & H^2_{\Iw}(\Q_{p,\infty},P(1)) \ar[d]^{\alpha} \ar[r]^-{\inv} & P \ar[d]^{\alpha(0)} \\
		H^1(\qp,P(1)) \ar[r] \ar[d]^{\alpha} & H^2_{\Iw}(\Q_{p,\infty},P(1)) \ar[r]^-{\inv} \ar@{=}[d] & P \ar@{=}[d] \\
		H^1(\qp,P^{\dagger}(1)) \ar[r] & H^2_{\Iw}(\Q_{p,\infty},P(1)) \ar[r]^-{\inv} & P,
		}
	$$
	where the horizontal compositions are $\overline{\Cole}^{\dagger}_P$ in the first and last rows and 
	$\Cole^{\flat}_P$ in the middle row (see \eqref{Coldagflat}) and are compatibly split. 
	The final statement of the proposition follows
	easily from this and the definitions of the splittings in question.
\end{proof}

\begin{corollary} \label{daggersquare}
	The square
	$$
		\SelectTips{cm}{} \xymatrix{
		H^1(\mc{O},T^{\dagger}(1)) \ar[r] \ar[d] & H^1(\mc{O},Q^{\dagger}(1)) \ar[d]^{-\Theta^{\dagger}}  \\
		H^1(\qp,P^{\dagger}(1)) \ar[r] & H^2_c(\mc{O},P^{\dagger}(1)) 
		}
	$$
	is commutative.
\end{corollary}

\begin{proof}
	Both compositions are trivial on the image of $H^1(\mc{O},P^{\dagger}(1))$ in $H^1(\mc{O},T^{\dagger}(1))$.
	Proposition \ref{daggersplit} implies that the composition (using the splittings of said proposition)
	$$
		H^1(\mc{O},Q^{\dagger}(1)) \to H^1(\mc{O},T^{\dagger}(1)) \to \bigoplus_{\ell \mid Np} H^1(\Q_{\ell},P^{\dagger}(1))
	$$
	takes image in the image of $H^2_c(\mc{O},P^{\dagger}(1))$.  This image
	is contained in $H^1(\qp,P^{\dagger}(1))$ by construction.  The result then follows from the commutativity of the square
	in Proposition \ref{snakesquare}.
\end{proof}

\subsection{Refined zeta maps} \label{refinedzeta}

In this subsection, we show how the existence of a \emph{refined zeta map} 
$$
	z^{\dagger} \colon \La \cozp \mcS_{\theta} \to H^1(\mc{O},\mcT_{\theta}^{\dagger}(1)),
$$ 
would imply Conjecture \ref{sconj}: see Proposition \ref{refined_zeta_map}. Here, 
$(1-U_p)z^{\dagger}$ should equal the composition of $z$ with the map
$H^1_{\Iw}(\mc{O}_{\infty},\mcT_{\theta}(1)) \to H^1(\mc{O},\mcT_{\theta}^{\dagger}(1))$,
in analogy with Theorem \ref{zsharp}.

Moreover, we show that the existence of a \emph{reduced} refined zeta map 
$$
	\bar{z}^{\dagger} \colon \La \cozp P \to H^1(\mc{O},T^{\dagger}(1)),
$$ 
one which is compatible with $\bar{z}^{\sharp}$ and $\bar{z}_{\quo}^{\dagger}$,
is equivalent to the conjecture. In fact,  in the proof of the following theorem, we construct a candidate for 
$\bar{z}^{\dagger}$ as a direct sum of maps to $H^1(\mc{O},P^{\dagger}(1))$ and $H^1(\mc{O},Q^{\dagger}(1))$,
and we show that Conjecture \ref{sconj} is equivalent to this candidate being compatible with $\bar{z}_{\quo}^{\dagger}$,
as it is already by construction compatible with $\bar{z}^{\sharp}$.

\begin{theorem} \label{equivform}
	Conjecture \ref{sconj} holds if and only if there exists a homomorphism $\bar{z}^{\dagger} \colon \Lambda \otimes_{\zp} P \to 
	H^1(\mc{O},T^{\dagger}(1))$ of $\La_{\theta} \cozp (\mf{h}/I)_{\theta}$-modules
	making the diagrams
	$$
	\SelectTips{cm}{} \xymatrix@C=35pt@R=8pt{
	\Lambda \otimes_{\zp} P \ar[r]^-{\bar{z}^{\dagger}} \ar@{->>}[dd]  & H^1(\mc{O},T^{\dagger}(1)) \ar[dd]
	& & \Lambda \otimes_{\zp} P  \ar[r]^-{\bar{z}^{\dagger}} \ar@{->>}[dd] & H^1(\mc{O},T^{\dagger}(1)) \ar[dd]  \\
	&& \mr{and} && \\
	P \ar[r]^-{\bar{z}^{\sharp}} & H^1(\mc{O},T(1)) 
	&& P \ar[r]^-{\bar{z}^{\dagger}_{\quo}} & H^1(\qp,P^{\dagger}(1)) 
	} 
	$$
	commute, where the vertical maps in the first diagram are induced by the augmentation on $\Lambda$.
\end{theorem}

\begin{proof}
	By Remark \ref{equivalent}, Conjecture \ref{sconj} holds if and only if $\Upsilon \circ \varpi = 1$. We use this form of the 
	conjecture.
	If a map $\bar{z}^{\dagger}$ as in the statement exists, then the composition of $-\bar{z}^{\dagger}$ with the maps 
	$$
		H^1(\mc{O},T^{\dagger}(1)) \to H^1(\mc{O},Q^{\dagger}(1)) \to H^1(\mc{O},Q(1)) \xrightarrow{\sim} Y
	$$
	is necessarily $\varpi$ (after application of $\ev_0$) 
	by the commutativity of the first diagram and Theorem \ref{zsharpvarpi}, and its further composition via $-\Theta^{\dagger}$ 
	to $H^2_c(\mc{O},P^{\dagger}(1))$
	is necessarily $\Upsilon \circ \varpi$ since we have seen that $-\Theta^{\dagger}$ induces a map $Y \to P$ that equals $\Upsilon$
	in Section \ref{refinedcoh}.
	At the same time, by the commutativity of the second diagram, this composition agrees
	with the composition of $-\bar{z}^{\dagger}_{\quo} \circ \ev_0$ with 
	$H^1(\qp,P^{\dagger}(1)) \to H^2_c(\mc{O},P^{\dagger}(1)) \cong P$, which is 
	$1$ by Lemma \ref{dagcompconn} and Proposition \ref{zquodagger}. Thus $\Upsilon \circ \varpi = 1$.

	Conversely, suppose that $\Upsilon \circ \varpi = 1$. Using the isomorphism 
	$$
		H^1(\mc{O},T^{\dagger}(1)) \cong H^1(\mc{O},P^{\dagger}(1)) \oplus H^1(\mc{O},Q^{\dagger}(1))
	$$
	given by the splitting of Proposition \ref{daggersplit}, we may define the projections of $\bar{z}^{\dagger}$ to these components
	separately: let us label them $\bar{z}^{\dagger}_P$ and $\bar{z}^{\dagger}_Q$, respectively. We will 
	similarly write $\bar{z}^{\sharp}$ as the sum of its components 
	$\bar{z}^{\sharp}_P$ and $\bar{z}^{\sharp}_Q$ corresponding to projection to $H^1(\mc{O},P(1))$ and $H^1(\mc{O},Q(1))$.
		
	Recall from Lemma \ref{cohquot} that we have 
	$$
		H^1(\mc{O},Q^{\dagger}(1)) \cong Y \cotimes{R} (\La_{\theta}/\xi)
	$$
	where $f \in \La \cozp \La_{\theta}$ acts as $w(f)$ on the right. Using this isomorphism to identify the two sides, we define
	$$
		\bar{z}_Q^{\dagger} \colon \La \cozp P \to Y \cotimes{R} \La_{\theta}/\xi
	$$
	to be the unique homomorphism such that $\bar{z}_Q^{\dagger}(1 \otimes p) = -\varpi(p) \otimes 1$ for 
	$p \in P$ and which is a map of $\La \cozp \La_{\theta}$-modules for the usual action of $f \in \La \cozp \La_{\theta}$ on
	the left and its action by $\w(f)$ on the right.
	Via the splitting of Lemma \ref{splitseqs}, we have an isomorphism
	$$
		\bigoplus_{\ell \mid Np} H^1(\Q_{\ell},P^{\dagger}(1)) \cong H^1(\mc{O},P^{\dagger}(1)) \oplus H^2_c(\mc{O},P^{\dagger}(1)),
	$$
	and we let $\bar{z}^{\dagger}_P$ be the projection of $\bar{z}_{\quo}^{\dagger} \circ \ev_0$ to the first component.	
	
	We next check commutativity of the first diagram. By Proposition \ref{daggersplit}, we may do this after projection to the summands
	corresponding to $P$ and $Q$, respectively. 
	For the $P$-components, note that $\bar{z}^{\dagger}_P$ is the projection of $\bar{z}_{\quo}^{\dagger} \circ
	\ev_0$ to $H^1(\mc{O},P^{\dagger}(1))$.  The composition of this map with the surjection to $H^1(\mc{O},P(1))$
	is the projection of $\bar{z}^{\sharp}_{\quo} \circ \ev_0$ to $H^1(\mc{O},P(1))$.
	This equals $\bar{z}_P^{\sharp} \circ \ev_0$ in that the restriction of $\bar{z}_P^{\sharp}$ to $H^1(\Q_{\ell},P(1))$ is trivial
	for primes $\ell \mid N$.
	That is, $\bar{z}_P^{\sharp}$ is a reduction of 
	$z^{\sharp} \colon \mcS_{\theta} \to H^1(\mc{O},\mc{T}_{\theta}(1))$, and
	$H^1(\Q_{\ell},\mc{T}_{\theta}(1))$ is trivial by Lemma \ref{cohomatell}.
	
	For the $Q$-components, we need only remark that
	the composition of $\bar{z}^{\dagger}_Q$ with the map to $H^1(\mc{O},Q(1))$
	is $\bar{z}^{\sharp}_Q = -\varpi$ (see Remark \ref{corQ}), 
	so we see that the first diagram commutes on the summands corresponding to $Q$.	
	
	For the second diagram, we have by definition that $\bar{z}^{\dagger}_P$ equals the projection of 
	$\bar{z}_{\quo}^{\dagger} \circ \ev_0$ to $H^1(\mc{O},P^{\dagger}(1))$. The composition of $-\bar{z}_Q^{\dagger}$ 
	with $-\Theta^{\dagger} \colon H^1(\mc{O},Q^{\dagger}(1)) \to H^2_c(\mc{O},P^{\dagger}(1))$, the latter
	group being identified with $P$,
	factors through the map $\Upsilon \circ \varpi = 1$ on $P$. As the composition of $-\bar{z}_{\quo}^{\dagger}$ with  
	$H^1(\qp,P^{\dagger}(1)) \to H^2_c(\mc{O},P^{\dagger}(1))$ 
	is also identified with $1$, the commutativity holds.
\end{proof}	

Note that the data of $\bar{z}^{\dagger}$ is equivalent to the data of its restriction to an $(\mf{h}/I)_{\theta}$-module homomorphism
$P \to H^1(\mc{O},T^{\dagger}(1))$ sending $x \in P$ to $\bar{z}^{\dagger}(1 \otimes x)$ and fitting in the corresponding
commutative diagrams arising from restriction to $P$.

The above discussion can be summarized by the diagram
\begin{equation} \label{refined_diag}
	\SelectTips{cm}{} \xymatrix@C=35pt@R=35pt{
	P \ar[rr]^-{\varpi} && Y \ar[r]^{\id} & Y \ar[d]^{\wr}  \ar@/^3.7pc/[ddd]^-{\Upsilon} \\
	\Lambda \otimes_{\zp} P \ar@{-->}[r]_-{-\bar{z}^{\dagger}} \ar@{->>}[u]^-{\ev_0} \ar@{->>}[d]_-{\ev_0}
	& H^1(\mc{O},T^{\dagger}(1)) \ar[r] \ar[d] & H^1(\mc{O},Q^{\dagger}(1)) 
	\ar[d]^{-\Theta^{\dagger}} \ar@{->>}[u] \ar[r]^{\varepsilon_Q}
	& H^2(\mc{O},Q(1)) \ar[d]\\
	P \ar[r]_-{-\bar{z}^{\dagger}_{\quo}} \ar@/_1.5pc/[rrd]_-{\id} & H^1(\qp,P^{\dagger}(1)) \ar[r]\ar[rd]_{-\overline{\Cole}^{\dagger}_P} & 
	H^2_c(\mc{O},P^{\dagger}(1)) \ar[r]^{\partial_P^{\dagger}}  & 
	H^3_c(\mc{O},P(1)) 
	\ar[d]^{\wr}  \\
	&& P \ar[u]^{\wr} \ar[r]_{\id} & P,
	} 
\end{equation}
which fully commutes if we know the existence of the conjectural map $\bar{z}^{\dagger}$ in Theorem \ref{equivform}.  
The equality $\Upsilon \circ \varpi = 1$ is then seen by tracing the outside of the diagram. This begs the following question, which would be
in analogy with the construction of $z^{\sharp}$ by Fukaya and Kato were it to hold. 

\begin{question} \label{zdagger}
	Does there exist a $\Lambda \cotimes{\zp} \mf{h}_{\theta}$-module homomorphism 
	$$
		z^{\dagger} \colon \Lambda \cotimes{\zp}  \mc{S}_{\theta} \to H^1(\mc{O},\mc{T}_{\theta}^{\dagger}(1))
	$$ 
	such that the diagram
	$$
		\SelectTips{cm}{}
		\xymatrix@C=30pt{\Lambda \cotimes{\zp} \mc{S}_{\theta} \ar[d]^-z \ar[r]^-{z^{\dagger}} 
		& H^1(\mc{O},\mc{T}_{\theta}^{\dagger}(1)) \ar[d]^{1-U_p} \\
		H^1_{\Iw}(\mc{O}_{\infty},\mc{T}_{\theta}(1)) \ar[r] & H^1(\mc{O},\mc{T}_{\theta}^{\dagger}(1)) }
	$$
	commutes?
\end{question}

Such a map $z^{\dagger}$ is uniquely determined if it exists by the following lemma.

\begin{lemma} \label{nofixeddagger}
	Multiplication by $1-U_p$ is injective on $H^1(\mc{O},\mc{T}_{\theta}^{\dagger}(1))$.
\end{lemma}

\begin{proof}
	As in the proof of Lemma \ref{nofixedelts}, it suffices to see that
	$(\mcT_{\theta}^{\dagger}/(U_p-1)\mcT_{\theta}^{\dagger})(1)$ has trivial $G_{\Q}$-invariants.
	First, we note that $G_{\Q_{p,\infty}}$ acts on $(\mcT_{\quo}^{\dagger}/(U_p-1)\mcT_{\quo}^{\dagger})(1)$ via the
	$p$-adic cyclotomic character, so $(\mcT_{\quo}^{\dagger}/(U_p-1)\mcT_{\quo}^{\dagger})(1)$ has trivial $G_{\qp}$-invariant group.
	
	We claim that the quotient $(\mc{T}_{\sub}^{\dagger}/(U_p-1)\mc{T}_{\sub}^{\dagger})(1)$ has no elements that are fixed by
	$G_{\Q}$ inside $(\mc{T}_{\theta}^{\dagger}/(U_p-1)\mc{T}_{\theta}^{\dagger})(1)$. Let $\mf{m}$ denote the maximal ideal
	of $\mf{h}_{\theta}$.
	We know from the proof of Lemma \ref{nofixedelts} that for any $x \in (\mc{T}_{\sub} - (U_p-1)\mc{T}_{\sub})(1)$, 
	the subgroup $(x\mc{T}_{\sub}/(x\mf{m}+ (U_p-1))\mc{T}_{\sub})(1)$ of
	$(x\mc{T}_{\theta}/(x\mf{m}+ (U_p-1))\mc{T}_{\theta})(1)$ is isomorphic to a quotient of $Q/\mf{m}Q(1)$, hence can
	be viewed as a $G_{\Q}$-module with action factoring through $\Gal(\Q(\mu_{Np})/\Q)$ via
	$\omega^2\theta^{-1}$. Since $G_{\Q}$ acts trivially on $\La^{\iota}$ modulo the maximal ideal of $\Lambda$,
	the image of $y \in \mcT_{\sub}^{\dagger}(1) - (U_p-1)\mcT_{\sub}^{\dagger}(1)$ is then similarly acted on by $G_{\Q}$ 
	through $\omega^2\theta^{-1}$ in the nonzero quotient
	$(\La \cozp \mf{h}_{\theta}) y/(\mf{M}y + (U_p-1)\mcT_{\sub}^{\dagger}(1))$ of 
	$(\mc{T}_{\sub}^{\dagger}/\mf{M}\mc{T}_{\sub}^{\dagger})(1) \cong (\mcT_{\sub}/\mf{m}\mcT_{\sub})(1)$, 
	where $\mf{M}$ is the unique maximal ideal of
	$\La \cotimes{\zp} \mf{h}_{\theta}$. Since $\theta \neq \omega^2$, we have the result.
\end{proof}

A positive answer to Question \ref{zdagger} appears likely to be too much to hope for in general. However, if it does hold, so does Conjecture \ref{sconj}.

\begin{proposition} \label{refined_zeta_map}
	If a map $z^{\dagger}$ as in Question \ref{zdagger} exists, then the reduction $\bar{z}^{\dagger} \colon \La \cotimes{\zp} P 
	\to H^1(\mc{O},T^{\dagger}(1))$ of $z^{\dagger}$ modulo $I$ satisfies the conditions of Theorem \ref{equivform}, and in particular
	Conjecture \ref{sconj} holds.
\end{proposition}

\begin{proof}
	By construction, we have that the composition of $(1-U_p)z^{\dagger}$ with the map 
	$$
		H^1(\mc{O},\mc{T}_{\theta}^{\dagger}(1)) \to H^1(\mc{O},\mc{T}_{\theta}(1))
	$$ 
	is $(1-U_p)z^{\sharp}$.
	By Lemma \ref{nofixedelts}, we see that the composition of $z^{\dagger}$ with the latter map is $z^{\sharp}$.
	In particular, the first diagram in Theorem \ref{equivform} commutes.
	
	Moreover, since $(1-U_p)z_{\quo}^{\dagger}$ agrees with the composition of $z_{\quo}$ with the map
	$$
		H^1_{\Iw}(\Q_{p,\infty},\mc{T}_{\quo}(1)) \to H^1(\qp,\mc{T}_{\quo}^{\dagger}(1)).
	$$
	and $1-U_p$ has trivial kernel on $\mfS_{\theta}^{\star}$, we have that the composition of $z^{\dagger}$ with the map
	$$
		H^1(\mc{O},\mc{T}_{\theta}^{\dagger}(1)) \to H^1(\qp,\mc{T}_{\quo}^{\dagger}(1))
	$$
	is $z_{\quo}^{\dagger}$. That is, the second diagram in Theorem \ref{equivform} commutes.
\end{proof}

\section{Test case} \label{testcase}

We explore the feasibility of the equivalent conditions to Conjecture \ref{sconj} found in Theorem \ref{equivform}, working with cyclotomic units
in place of Beilinson-Kato elements. We find, somewhat reassuringly, that an analogue of the conditions of Theorem \ref{equivform} holds in this setting. 

On the other hand, an analogue of the stronger Question \ref{zdagger}, which amounts to a norm relation for a good choice of $z^{\dagger}$, has a potential obstruction. We show that this norm relation does hold if an even eigenspace of the completely split Iwasawa module vanishes.

\subsection{Notation}

Let us first introduce changes to our notation from the previous sections. Most importantly, we now allow our prime $p$ to divide $\varphi(N)$. 
That is, we let $p$ be an odd prime, and we let $N \ge 3$ be a positive integer with $p \nmid N$. Let $\Delta = (\Z/Np\Z)^{\times}$ as before,
which we identify with $\Gal(\Q(\mu_{Np^{\infty}})/\Q_{\infty}) \cong \Gal(\Q(\mu_{Np})/\Q)$.

\begin{definition}\
\begin{enumerate}
	\item[a.] Let $\Delta_p$ and $\Delta'$ be the Sylow $p$-subgroup of $\Delta$ and its prime-to-$p$ order complement, respectively. 
	\item[b.] Let $\theta \colon \Delta' \to \qpbt$ be a nontrivial even character of $\Delta'$ which is trivial on decomposition at $p$ and primitive at 
	all primes dividing $N$. 
	\item[c.] Let $R$ be the $\zp[\Delta_p]$-algebra of values of $\theta$, which we then view as a $\zp[\Delta]$-module with
	$\Delta'$ acting through $\theta$.
	\item[d.] Let $\La_{\theta} = R\ps{\Gamma} \cong R\ps{X}$, where $\Gamma \cong
	\Gal(\Q(\mu_{Np^{\infty}})/\Q(\mu_{Np}))$, and $X$ is as before.
	\item[e.] Let $\mc{R} = R^{\iota}$ be the $\zp[\Delta]$-module that is $R$ 
	endowed with the inverse of the Galois action described above.
\end{enumerate}
\end{definition}

\begin{definition}
	The $\theta$-part $M_{\theta}$ of a $\zp[\Delta]$-module $M$ is the $R$-module
	$M_{\theta} = M \otimes_{\zp[\Delta]} R$.
\end{definition}

\begin{remark} \label{theta-part_cohom}
	Our choice of $\mc{R}$ is made so that 
	$$
		H^i(\mc{O},\mc{R}(1)) \cong H^i(\Z[\tfrac{1}{Np},\mu_{Np}],\zp(1))_{\theta} \cong H^i(\Z[\tfrac{1}{Np},\mu_N],\zp(1))_{\theta}
	$$ 
	for any $i \ge 0$ by Shapiro's lemma, and similarly for Iwasawa cohomology.
\end{remark}

We shall also use the following.

\begin{definition} \
	\begin{enumerate}
		\item[a.] Let $\sigma$ denote the image of the Frobenius $\Fr_p$ at $p$ in $\Delta_p$.
		\item[b.] Let $R_{\sigma=1}$  denote the maximal quotient and $R^{\sigma=1}$ the maximal submodule of 
		$R$ on which $\sigma$ acts trivially. 
		\item[c.] Let $Y$ (resp., $\tilde{Y}$) denote the Galois group of the maximal completely locally split
		(resp., unramified) abelian pro-$p$ extension of $\Q(\mu_{Np^{\infty}})$.
		\item[d.] Let $\mf{X}$ denote the Galois group of the maximal abelian, unramified outside $Np$, pro-$p$ extension
		of $\Q(\mu_{Np^{\infty}})$.
		\item[e.] Let $\mc{E}$ (resp., $\mc{C}$) denote the group of norm compatible
		systems of $p$-completions of global units (resp., cyclotomic units) in the tower $\Q(\mu_{Np^{\infty}})/\Q$.
	\end{enumerate}
\end{definition}

Noting Remark \ref{theta-part_cohom}, Kummer theory provides an isomorphism 
$\mc{E}_{\theta} \cong H^1_{\Iw}(\mc{O}_{\infty},\mc{R}(1))$. For interpretations
of $Y_{\theta}$ and $\mf{X}_{\theta}$ in terms of \'etale (or Galois) cohomology, see Section \ref{briefcohstudy}.

\subsection{Zeta and Coleman maps}

Using the identifications of Remark \ref{theta-part_cohom}, 
we take our zeta map as having image the $\theta$-part of the cyclotomic units.

\begin{definition} \label{cycl_zeta_map} \
	\begin{enumerate}
	\item[a.]  The \emph{zeta map} $z$ is the $\La_{\theta}$-module homomorphism 
	$$
		z \colon \La_{\theta} \to H^1_{\Iw}(\mc{O}_{\infty},\mc{R}(1))
	$$ 
	that sends $-1$ to the projection of the norm compatible sequence $(1-\zeta_N^{p^{-r}}\zeta_{p^r})_{r \ge 1}$
	to the $\theta$-part of $H^1_{\Iw}(\mc{O}_{\infty}[\mu_{Np}],\zp(1))$.
	\item[b.] We define a $\zp[\Delta]$-module homomorphism 
	$$
		z^{\sharp} \colon R \to H^1(\mc{O},\mc{R}(1))
	$$ 
	as the unique such map taking $-1$ to the projection of $1-\zeta_N$ in $H^1(\mc{O}[\mu_N],\zp(1))_{\theta}$.
	\end{enumerate}
\end{definition}

\begin{remark}
	The use of negative signs in Definition \ref{cycl_zeta_map} is perhaps not the ideal convention, but it is consistent 
	with our conventions for the zeta maps in the prior sections, which took Manin symbols to (compatible systems of) 
	negatives of cup products of Siegel units.
\end{remark}	
	
We use $z_{\quo}$ and $z^{\sharp}_{\quo}$ to denote the precompositions of $z$ and $z^{\sharp}$ with restriction
to the cohomology of $G_{\qp}$.
	
\begin{remark}
	The zeta map and its ground level analogue satisfy the well-known norm relation
	$$
	\SelectTips{cm}{} \xymatrix@C=30pt{
		\La_{\theta} \ar[r]^-z \ar[d]^{z^{\sharp} \circ \ev_0} & H^1_{\Iw}(\mc{O}_{\infty},\mc{R}(1)) \ar[d] \\
		H^1(\mc{O},\mc{R}(1)) \ar[r]^{1-\sigma^{-1}} & H^1(\mc{O},\mc{R}(1)).
	}
	$$
	among cyclotomic units.
\end{remark}

\begin{definition} \label{xi}
	We let $\xi \in \La_{\theta}$ be the unique element satisfying 
	$$
		\tilde{\rho}(\xi(u^{1-s}-1)) = L_p(\theta\rho,s)
	$$ 
	for all $s \in \zp$ and $p$-adic characters $\rho$ of $\Delta_p$,
	where we use $\tilde{\rho}$ to denote the map $R \to \qpbar$ induced by $\rho$.
\end{definition}

We note the following equivariant formulation of a theorem which emanates from work of Iwasawa and is proven by Tsuji 
\cite[Theorem 4.3]{tsuji} in the form and generality we need, improving upon work of Greither \cite{greither}.

\begin{remark} \label{zetacolcycl}
	As an $R$-module, $D(\mc{R})$ is free of rank $1$, and it can be identified with $R$ as a $\zp$-algebra
	after a choice of normal basis of the valuation ring of the unramified extension of $\Q(\mu_p)$ defined by the decomposition 
	group of $\Delta_p$. We can and do choose this identification such that the Coleman map 
	$$
		\Cole = \Cole_{\mc{R}} \colon H^1_{\Iw}(\Q_{p,\infty},\mc{R}(1)) \to X^{-1}\La_{\theta}
	$$
	satisfies $\Cole \circ z = \xi$. While Tsuji proves this equality after application of an arbitrary $p$-adic character of $\Delta_p$,
	this equivariant formulation is immediate from its derivation. In fact, one could take $\xi$ to be defined
	by $\xi = \Cole \circ z$ (for a good choice of basis as in \cite{tsuji}), and 
	Tsuji's result tells us that this $\xi$ satisfies the characterizing property of Definition \ref{xi}.
\end{remark}

To shorten notation, let us write $\mfC$ for the image $\mfC(\mc{R})$ of $\Cole$ and similarly with superscripts
adorning $\mfC$. 
Consider the Coleman map 
$$
	\Cole^{\dagger} \colon H^1(\qp,\mc{R}^{\dagger}(1)) \to \mfC^{\star}
$$ 
for $A = \mc{R}$ and $\alpha = \xi$ of Theorem \ref{loccohdag}. 
The analogous argument to that of Proposition \ref{zquodagger} yields the following.

\begin{proposition} \label{interloczeta}
	There exists a canonical $\La_{\theta}$-module homomorphism
	$$
		z_{\quo}^{\dagger} \colon R \to H^1(\qp,\mc{R}^{\dagger}(1))
	$$
	such that $\overline{\Cole}^{\dagger} \circ z^{\dagger}_{\quo} = 1$ and such that the diagram
	$$
		\SelectTips{cm}{} \xymatrix@C=40pt{
		\La_{\theta} \ar[r]^-{z_{\quo}} \ar[d]^{z_{\quo}^{\dagger} \circ \ev_0} & H^1_{\Iw}(\Q_{p,\infty},\mc{R}(1)) \ar[d] \\
		H^1(\qp,\mc{R}^{\dagger}(1)) \ar[r]^{1-\sigma^{-1}} & H^1(\qp,\mc{R}^{\dagger}(1))
		}
	$$
	commutes. Moreover, $z_{\quo}^{\sharp}$ is the composition of $z_{\quo}^{\dagger}$ with the map 
	$$
		H^1(\qp,\mc{R}^{\dagger}(1)) \to H^1(\qp,\mc{R}(1)).
	$$ 
\end{proposition}

\begin{proof}
	Identifying $H^1(\qp,\mc{R}^{\dagger}(1))$ with $\mfC^{\star}$ via $\Cole^{\dagger}$, we define $z^{\dagger}_{\quo}$ to be 
	as the pushout map $R \cong D(\mc{R}) \to \mfC^{\star}$.
	By definition of $\mfC^{\star}$, following this by $1-\sigma^{-1}$, we get the composition 
	$$
		R \xrightarrow{\xi} \mfC^{\dagger} \to \mfC^{\star},
	$$ 
	which is to say, recalling Remark \ref{zetacolcycl}, the composition of $z_{\quo}$ with  $H^1_{\Iw}(\Q_{p,\infty},\mc{R}(1)) 
	\to H^1(\qp,\mc{R}^{\dagger}(1))$.
\end{proof}

\subsection{Brief cohomological study} \label{briefcohstudy}

We describe the structure of some relevant cohomology groups.

\begin{lemma} \label{localtriv}
	For each prime $\ell \mid N$, the cohomology groups $H^i(\Q_{\ell},\mc{R}(1))$, $H^i(\Q_{\ell},\mc{R}^{\dagger}(1))$, and
	$H^i(\Q_{\ell},\Lai \cozp \mc{R}(1))$ for $i \in \{1,2\}$ are all trivial.
\end{lemma}

\begin{proof}
	By Shapiro's lemma, the group $H^i(\Q_{\ell},\mc{R}(1))$ is isomorphic to the $\theta$-eigenspace of the product 
	of the groups $H^i(\Q_{\ell}(\mu_{Np}),A_{\theta}(1))$ over primes over $\ell$ in the field $\Q(\mu_{Np})$, where 
	$A_{\theta}$ is the $\zp$-algebra of 
	$\theta$-values with the trivial action of Galois. Since the pro-$p$ completion of $\Q_{\ell}(\mu_{Np})^{\times}$ is generated
	by a uniformizer and a $p$-power root of unity, 
	each of these first cohomology groups is isomorphic to a direct sum of $A_{\theta}$ and a quotient of $A_{\theta}(1)$ 
	via the Kummer isomorphism.
	The second cohomology groups are also isomorphic to $A_{\theta}$ via the invariant map.
	Since inertia at $\ell$ in $\Gal(\Q_{\ell}(\mu_{Np})/\Q_{\ell})$ acts trivially on this product and 
	$\theta$ is primitive at $\ell$, the $\theta$-eigenspace of the product is zero.
	
	Note that $\Lai \cozp \mc{R} \cong \La_{\theta}^{\iota}$. We have an exact sequence
	$$
		0 \to H^1(\Q_{\ell},\La_{\theta}^{\iota}(1))/X\xi \to H^1(\Q_{\ell},\mc{R}^{\dagger}(1)) \to H^2(\Q_{\ell},\Lai_{\theta}(1))[X\xi] \to 0,
	$$
	and $H^2(\Q_{\ell},\mc{R}^{\dagger}(1))$ is a quotient of $H^2(\Q_{\ell},\Lai_{\theta}(1))$.
	The groups $H^i(\Q_{\ell},\La_{\theta}^{\iota}(1))$ are zero for $i \in \{1,2\}$, each being isomorphic to the $\theta$-eigenspace
	of the product of the Iwasawa cohomology groups 
	$H^i_{\Iw}(\Q_{\ell}(\mu_{Np^{\infty}}),A_{\theta}(1))$ 
	over primes of $\Q(\mu_{Np^{\infty}})$ over $\ell$.
\end{proof}

The invariant map provides
an isomorphism
$$
	\inv \colon H^2_{\Iw}(\Q_{p,\infty},\mc{R}(1)) \xrightarrow{\sim} R_{\sigma=1},
$$
since $R_{\sigma=1}$ is the maximal unramified quotient of $R$. 
Note also that $H^3_{c,\Iw}(\mc{O}_{\infty},\mc{R}(1)) = 0$ since $\Delta'$ acts on $\mc{R}$ through the nontrivial prime-to-$p$ order character 
$\theta^{-1}$. 

\begin{lemma} \label{H2seq}
	We have an exact sequence of $\Lambda$-modules
	$$
		0 \to Y_{\theta} \to H^2_{\Iw}(\mc{O}_{\infty},\mc{R}(1)) \to R_{\sigma=1} \to 0.
	$$
\end{lemma}

\begin{proof}
	This is immediate from the Poitou-Tate sequence, the invariant map for $\Q_{p,\infty}$, 
	the triviality of $H^2(\Q_{\ell},\La_{\theta}^{\iota}(1))$ for $\ell \mid N$ of Lemma \ref{localtriv}, and
	the triviality of $H^3_{c,\Iw}(\mc{O}_{\infty},\mc{R}(1))$.
\end{proof}

\begin{lemma} \label{cptcoh}
	We have $H^1_c(\mc{O},\mc{R}(1)) = 0$ and $H^i_c(\mc{O},\mc{R}^{\dagger}(1)) \cong \mf{X}_{\theta}$ for $i \in \{1,2\}$.
\end{lemma}

\begin{proof}
	We employ some well-known results of classical Iwasawa theory: see for instance the book of Ochiai \cite{ochiai}. Recall that 
	$\mf{X}_{\theta}$ is pseudo-isomorphic to $(Y^{\iota}(1))_{\theta}$, where $Y^{\iota}$ is $Y$ with the inverse 
	$\Gal(\Q(\mu_{Np^{\infty}})/\Q)$-action. (In this eigenspace, there is no difference between $Y$ and the unramified Iwasawa module.)
	The group $(Y^{\iota}(1))_{\theta}$ is annihilated
	by $\xi$ by Stickelberger theory. Since $\mf{X}_{\theta}$ has no finite $\La_{\theta}$-submodules, it too is annihilated
	by $\xi$. 
	
	Since $H^3_{c,\Iw}(\mc{O}_{\infty},\mc{R}(1))$ is trivial and 
	$$
		H^2_{c,\Iw}(\mc{O}_{\infty},\mc{R}(1)) \cong (H^1(\Z_{\infty}[\mu_{Np},\tfrac{1}{Np}],\qp/\zp)^{\vee})_{\theta} \cong 
		\mf{X}_{\theta} 
	$$
	by Poitou-Tate duality, we have
	$$
		H^2_c(\mc{O},\mc{R}^{\dagger}(1)) \cong H^2_{c,\Iw}(\mc{O}_{\infty},\mc{R}(1))/X\xi \cong \mf{X}_{\theta}/X\xi \cong \mf{X}_{\theta}.
	$$
	Next, note that we have an exact sequence
	$$
		0 \to H^1_{c,\Iw}(\mc{O}_{\infty},\mc{R}(1))/X\xi \to H^1_c(\mc{O},\mc{R}^{\dagger}(1)) \to H^2_{c,\Iw}(\mc{O}_{\infty},\mc{R}(1))[X\xi]
		\to 0,
	$$
	and $H^1_{c,\Iw}(\mc{O}_{\infty},\mc{R}(1))$ vanishes by the weak Leopoldt conjecture. Thus, $H^1_c(\mc{O},\mc{R}^{\dagger}(1)) \cong
	\mf{X}_{\theta}$ by the above description of the rightmost group in the sequence.
	Similarly, $H^1_c(\mc{O},\mc{R}(1))$ is trivial by the Leopoldt conjecture for abelian fields.
\end{proof}

\subsection{Questions and answers}

We first show that an analogue of Proposition \ref{zdagger} does indeed hold.

\begin{proposition} \label{zdagcycl}
	There exists a $\La_{\theta}$-module homomorphism 
	$$
		z^{\dagger} \colon \La_{\theta} \to H^1(\mc{O},\mc{R}^{\dagger}(1))
	$$
	such that the diagrams
	$$
		\SelectTips{cm}{} \xymatrix@C=35pt@R=8pt{
		\La_{\theta} \ar[r]^-{z^{\dagger}} \ar@{->>}[dd]  & H^1(\mc{O},\mc{R}^{\dagger}(1)) \ar[dd]
		&& \La_{\theta}  \ar[r]^-{z^{\dagger}} \ar@{->>}[dd] & H^1(\mc{O},\mc{R}^{\dagger}(1)) \ar[dd]  \\
		&&\mr{and} \\
		R \ar[r]^-{z^{\sharp}} & H^1(\mc{O},\mc{R}(1)) 
		&& R \ar[r]^-{z^{\dagger}_{\quo}} & H^1(\qp,\mc{R}^{\dagger}(1)) 
	} 
	$$
	commute.
\end{proposition}

\begin{proof}
	The definition of compactly supported cohomology and Lemma \ref{localtriv} provide a map of exact sequences
	$$
		\SelectTips{cm}{} \xymatrix{
		0 \ar[r] & H^1_c(\mc{O},\mc{R}^{\dagger}(1)) \ar[r] \ar[d] & H^1(\mc{O},\mc{R}^{\dagger}(1)) \ar[r] \ar[d] & H^1(\qp,\mc{R}^{\dagger}(1)) 
		\ar[r] \ar[d] & H^2_c(\mc{O},\mc{R}^{\dagger}(1)) \ar[d] \\
		0 \ar[r] & H^1_c(\mc{O},\mc{R}(1)) \ar[r] & H^1(\mc{O},\mc{R}(1)) \ar[r] & H^1(\qp,\mc{R}(1)) \ar[r] & H^2_c(\mc{O},\mc{R}(1)).
		}
	$$
	
	As already noted, $z^{\sharp}$ induces a map $z_{\quo}^{\sharp}$ that lifts to a map $z_{\quo}^{\dagger}$. The image of 
	$z_{\quo}^{\dagger}$ lies by definition in the $\Gamma$-invariant group of $H^1(\qp,\mc{R}^{\dagger}(1))$ and therefore maps
	to $H^2_c(\mc{O},\mc{R}^{\dagger}(1))^{\Gamma} = \mf{X}_{\theta}^{\Gamma}$ by Lemma \ref{cptcoh}.
	But $\mf{X}_{\theta}^{\Gamma}  = 0$ by weak Leopoldt (cf. \cite[(11.3.3) and (11.3.5)]{nsw}).
	
	Thus,
	there exists an element $x \in H^1(\mc{O},\mc{R}^{\dagger}(1))$ with image $z_{\quo}^{\dagger}(1)$ in $H^1(\qp,\mc{R}^{\dagger}(1))$,
	which since $H^1_c(\mc{O},\mc{R}(1)) = 0$ by Lemma \ref{cptcoh}, 
	necessarily then also has image $z^{\sharp}(1)$ in $H^1(\mc{O},\mc{R}(1))$. We can then take $z^{\dagger}$ as the
	unique $\La_{\theta}$-module homomorphism with $z^{\dagger}(1) = x$.
\end{proof}

Note that $z^{\dagger}$ in Proposition \ref{zdagcycl} is unique only up to an element of $H^1_c(\mc{O},\mc{R}^{\dagger}(1)) \cong 
\mf{X}_{\theta}$ (by Lemma \ref{cptcoh}). 
The analogue of Question \ref{zdagger} is the following.

\begin{question} \label{zdaggercycl}
	Does there exist a $\La_{\theta}$-module homomorphism
	$$
		z^{\dagger} \colon \La_{\theta} \to H^1(\mc{O},\mc{R}^{\dagger}(1))
	$$
	as in Proposition \ref{zdagcycl}
	such that the diagram
	$$
		\SelectTips{cm}{} \xymatrix@C=40pt{
		\La_{\theta} \ar[r]^-z \ar[d]^{z^{\dagger}} & H^1_{\Iw}(\mc{O}_{\infty},\mc{R}(1)) \ar[d] \\
		H^1(\mc{O},\mc{R}^{\dagger}(1)) \ar[r]^{1-\sigma^{-1}} & H^1(\mc{O},\mc{R}^{\dagger}(1))
		}
	$$
	commutes?
\end{question}

\begin{remark}
It is easy enough to construct a map $y^{\dagger} \colon R \to H^1_{\Iw}(\mc{O}_{\infty},\mc{R}(1))$ such that
$(1-\sigma^{-1}) y^{\dagger} \circ \ev_0$ is the image of $\xi(0)z$. For this, compose $z^{\sharp}$ 
with the multiplication-by-$\xi$ map $H^1(\mc{O},\mc{R}(1)) \to H^1(\mc{O}_{\infty},\mc{R}(1))$.
This, however, is not ideal: it is the analogue of multiplying by the derivative $\xi'$ in the setting of the other sections of this paper.
In general, one cannot do better than this if we ask for a map from $R$, rather than $\La_{\theta}$.
\end{remark}

If we suppose that $Y_{\theta} = 0$, then a map $z^{\dagger}$ as in Question \ref{zdaggercycl} does indeed exist.
To see this, we first compute the relevant Iwasawa modules under this assumption (assuming some standard results of classical
Iwasawa theory without reference).

\begin{lemma} \label{IwmodY0}
	If $Y_{\theta} = 0$, then the Coleman map fits in an isomorphism of exact sequences
	$$
		\SelectTips{cm}{} \xymatrix{
		0 \ar[r] & H^1_{\Iw}(\mc{O}_{\infty},\mc{R}(1)) \ar[r] \ar[d]^{\wr} & H^1_{\Iw}(\Q_{p,\infty},\mc{R}(1)) \ar[r] \ar[d]_{\Cole}^{\wr}
		& \mf{X}_{\theta} \ar[r] \ar[d]^{\wr} & 0 \\
		0 \ar[r] & \mfC \ar[r]^{\xi} & \mfC \ar[r] & \mfC/\xi\mfC \ar[r] & 0
		}
	$$
	of $\La_{\theta}$-modules. Moreover, we have a $\La_{\theta}$-module isomorphism $\tilde{Y}_{\theta} \cong R^{\sigma=1}/\xi(0)R^{\sigma=1}$.
\end{lemma}

\begin{proof}
	The exactness of the top row in the diagram follows from the Poitou-Tate sequence, the three lemmas of Section \ref{briefcohstudy},
	and, for right exactness, our assumption that $Y_{\theta} = 0$.
	Note that $\Cole$ gives an isomorphism $H^1_{\Iw}(\Q_{p,\infty},\mc{R}(1)) \cong \mfC$. 
	Since $Y_{\theta} = 0$, the Iwasawa module 
	$\mf{X}_{\theta}$ is then a quotient of $\mfC$. By Stickelberger theory, we know it is annihilated by $\xi$. The
	main conjecture tells us that the characteristic ideal of the maximal quotient of $\mf{X}_{\theta}$ upon which $\Delta_p$
	acts through a given character is generated by the image of $\xi$ in the resulting Iwasawa algebra.
	Note that $\mfC/\xi\mfC$ has this property: since $\xi(0) \neq 0$ in $R_{\sigma=1}$, the exact sequence $0 \to \La_{\theta} \to \mf{C}
	\to R^{\sigma=1} \to 0$ yields the exact sequence
	$$
		0 \to \La_{\theta}/\xi \to \mf{C}/\xi\mf{C} \to R^{\sigma=1}/\xi(0)R^{\sigma=1} \to 0
	$$
	of kernels and cokernels of multiplication by $\xi$.
	Since $\mf{X}_{\theta}$ has no finite $\La_{\theta}$-submodules, we must therefore
	have $\mf{X}_{\theta} \cong \mfC/\xi\mfC$.
	
	Since $Y_{\theta} = 0$, the group $\tilde{Y}_{\theta}$ is finite. Then $\mc{E}_{\theta} = \mc{C}_{\theta}$ by a standard argument,
	and $\mc{C}_{\theta} \cong \La_{\theta}$, generated by the projection of $(1-\zeta_N^{p^{-r}}\zeta_{p^r})_r$.
	As $\Cole \circ z = \xi$, the resulting composition 
	$$
		\La_{\theta} \xrightarrow{\sim} \mc{E}_{\theta} \to H^1_{\Iw}(\Q_{p,\infty},\mc{R}(1)) \xrightarrow{\sim} \mfC
	$$ 
	is given by multiplication by $\xi$. Kummer theory then gives a map of exact sequences
	$$
		\SelectTips{cm}{} \xymatrix{
		0 \ar[r] & \La_{\theta} \ar[r] \ar[d]^{\xi} & H^1_{\Iw}(\mc{O}_{\infty},\mc{R}(1)) \ar[r] \ar@{^{(}->}[d] & R^{\sigma=1} \ar[r] 
		\ar@{=}[d] &
		\tilde{Y}_{\theta} \ar[r] & 0, \\
		0 \ar[r] & \La_{\theta} \ar[r] & \mfC \ar[r] & R^{\sigma=1} \ar[r] & 0,
		}
	$$
	where the upper map to $R^{\sigma=1}$ is given by the valuations at the primes over $p$ (cf. \cite[(11.3.10)]{nsw}).
	Since $\mfC$ has no $\La$-torsion, any element $a \in H^1_{\Iw}(\mc{O}_{\infty},\mc{R}(1))$ 
	with $Xa = b \in \La_{\theta}$ is taken to $X^{-1} \xi b \in \mfC$, and therefore has image $\xi(0)b$ in $R^{\sigma=1}$.
	In other words, we have a surjection $\tilde{Y}_{\theta} \to R^{\sigma=1}/\xi(0)R^{\sigma=1}$. At the same time, $\xi$ annihilates
	$\mf{X}_{\theta}$, so $\xi(0)$ annihilates $\tilde{Y}_{\theta}$, and this map is an isomorphism. 
	
	We then have a surjective map $H^1_{\Iw}(\mc{O}_{\infty},\mc{R}(1)) \to R^{\sigma=1}$ given by $\xi(0)^{-1}$ times the valuation
	maps, and the resulting map $R^{\sigma=1} \to R^{\sigma=1}$ becomes multiplication by $\xi(0)$, or $\xi$. 
	This identifies $H^1_{\Iw}(\mc{O}_{\infty},\mc{R}(1))$ with $\xi\mfC$ as a subgroup of $\mfC$. In other words,
	we have an isomorphism $H^1_{\Iw}(\mc{O}_{\infty},\mc{R}(1)) \cong \mfC$ such that the resulting map $\mfC \to \mfC$
	is multiplication by $\xi$.
\end{proof}

Note that the composition of $z$ with the isomorphism $H^1_{\Iw}(\mc{O}_{\infty},\mc{R}(1)) \to \mfC$ of Lemma \ref{IwmodY0}
is the canonical injection $\La_{\theta} \to \mfC$.

\begin{proposition} \label{zdaggercyclY0}
	Suppose that $Y_{\theta} = 0$. Then Question \ref{zdaggercycl} has a positive answer.
\end{proposition}

\begin{proof}
	Since $\mfC$ has no $\La$-torsion and $\mf{X}_{\theta}$ is annihilated by $X\xi$, 
	the short exact sequence of Lemma \ref{IwmodY0} gives rise to the first row in the commutative diagram
	$$
		\SelectTips{cm}{} \xymatrix{
		& 0 \ar[d] & 0 \ar[d] & 0 \ar[d]  \\
		0 \ar[r] & \mf{X}_{\theta} \ar[r] \ar@{=}[d] & \mfC^{\dagger} \ar[r]^{\xi} \ar[d] & 
		\mfC^{\dagger} \ar[d] \ar[r] & \mf{X}_{\theta} \ar@{=}[d] \ar[r] & 0 \\
		0 \ar[r] & \mf{X}_{\theta} \ar[r] & H^1(\mc{O},\mc{R}^{\dagger}(1)) \ar[r] \ar[d] & \mfC^{\star} \ar[d] \
		\ar[r] & \mf{X}_{\theta} \ar[r] & 0\\
		&& H^2_{\Iw}(\mc{O}_{\infty},\mc{R}(1)) \ar[r]^-{\sim} \ar[d] & R_{\sigma=1} \ar[d] \\
		&& 0 & 0
		}
	$$
	with exact rows and columns. The exactness of the middle row is by the Poitou-Tate sequence, Theorem \ref{loccohdag}, and Lemmas
	\ref{localtriv} and \ref{cptcoh}, while the isomorphism of the final row is Lemma \ref{H2seq}.

	Let $x \in H^1(\mc{O},\mc{R}^{\dagger}(1))$ map to $1 \in R_{\sigma=1}$, and set $y = (1-\sigma^{-1})x \in \mfC^{\dagger}$,
	using the identification given by Lemma \ref{IwmodY0}.
	By definition of $\mfC^{\star}$, the image $x_{\loc}$ of $x$ in $\mfC^{\star}$
	has the property that $y_{\loc} = 
	(1-\sigma^{-1})x_{\loc} = \xi  \in \mfC^{\dagger}$. 
	This forces $\xi y = \xi$, and hence $y \equiv 1 \bmod X$ in the image of $\La_{\theta}$ in $\mfC^{\dagger}$. Choose
	$\lambda \in \La_{\theta}$ such that $\lambda y = 1$. We then define $z^{\dagger}$ as the unique $\La_{\theta}$-module
	homomorphism with $z^{\dagger}(1) = \lambda x$. 
	
	By construction, $x_{\loc} = \Cole^{\dagger} \circ z_{\quo}^{\dagger}(1)$. Since $z_{\quo}^{\dagger}(1)$ is 
	$\Gamma$-fixed, we have that $\lambda x_{\loc} = x_{\loc}$, and $z^{\dagger}$ restricts to $z_{\quo}^{\dagger}$. 
	Moreover, $(1-\sigma^{-1})z^{\dagger}(1) = 1 \in \mfC^{\dagger}$ is the image of
	$z(1)$, so we have the commutativity of the diagram in Question \ref{zdaggercycl}. Finally, $(1-\sigma^{-1})z^{\dagger}(1)$ has
	image $(1-\sigma^{-1})z^{\sharp}(1) \in H^1(\mc{O},\mc{R}(1))$, 
	and the norm of $z^{\sharp}(1)$ under the subgroup generated by $\sigma$ is trivial, so
	$z^{\dagger}(1)$ maps to $z^{\sharp}(1)$ as well.
\end{proof}

Without assuming that $Y_{\theta} = 0$, the existence of $z^{\dagger}$ as in Question \ref{zdaggercycl} requires the splitting of
the exact sequence of Lemma \ref{H2seq} as $R$-modules. 

\begin{ack}
	The author thanks Takako Fukaya and Kazuya Kato for their brilliant work on his conjecture and
	very helpful discussions.  He also thanks Barry Mazur and Preston Wake for insightful conversations related to this work,
	Ralph Greenberg for a helpful conversation on the test case of the last section, and Richard Hain for his explanations
	regarding Poincar\'e duality. The author also thanks the referee for many useful suggestions for improvements to and
	clarifications of our exposition that we have incorporated.
	This research was supported in part by the National Science Foundation under Grant No. 2101889.
\end{ack}

\renewcommand{\baselinestretch}{1}

\end{document}